\documentclass[11pt,a4paper]{amsart}
\pdfoutput=1
\usepackage[centering,margin=1in]{geometry}
\usepackage{amsmath,amsfonts,amsthm,amssymb,enumerate,enumitem,graphicx,lmodern,mathdots,mathrsfs,mathtools}
\usepackage[alphabetic,nobysame]{amsrefs}
\usepackage[stretch=10]{microtype}
\usepackage[pdfencoding=auto,psdextra]{hyperref}
\usepackage{xcolor}
\definecolor{dark-red}{rgb}{0.4,0.15,0.15}
\definecolor{dark-blue}{rgb}{0.15,0.15,0.4}
\definecolor{medium-blue}{rgb}{0,0,0.5}
\hypersetup{
	pdftitle={Test Vectors for Archimedean Period Integrals},
	pdfauthor={Peter Humphries and Yeongseong Jo},
	pdfnewwindow=true,
	colorlinks, linkcolor={dark-red},
	citecolor={dark-blue}, urlcolor={medium-blue}
}
\makeatletter
\newcommand{\bigboxplus}{
	\mathop{
		\vphantom{\bigoplus} 
		\mathchoice
		{\vcenter{\hbox{\resizebox{\widthof{$\displaystyle\bigoplus$}}{!}{$\boxplus$}}}}
		{\vcenter{\hbox{\resizebox{\widthof{$\bigoplus$}}{!}{$\boxplus$}}}}
		{\vcenter{\hbox{\resizebox{\widthof{$\scriptstyle\oplus$}}{!}{$\boxplus$}}}}
		{\vcenter{\hbox{\resizebox{\widthof{$\scriptscriptstyle\oplus$}}{!}{$\boxplus$}}}}
	}\displaylimits 
}
\newcommand{\bigboxtimes}{
	\mathop{
		\vphantom{\bigotimes} 
		\mathchoice
		{\vcenter{\hbox{\resizebox{\widthof{$\displaystyle\bigotimes$}}{!}{$\boxtimes$}}}}
		{\vcenter{\hbox{\resizebox{\widthof{$\bigotimes$}}{!}{$\boxtimes$}}}}
		{\vcenter{\hbox{\resizebox{\widthof{$\scriptstyle\otimes$}}{!}{$\boxtimes$}}}}
		{\vcenter{\hbox{\resizebox{\widthof{$\scriptscriptstyle\otimes$}}{!}{$\boxtimes$}}}}
	}\displaylimits 
}
\makeatother
\allowdisplaybreaks
\newcommand{\A}{\mathbb{A}}
\newcommand{\Agp}{\mathrm{A}}
\newcommand{\As}{\mathrm{As}}
\newcommand{\Bgp}{\mathrm{B}}
\newcommand{\BF}{\mathrm{BF}}
\renewcommand{\C}{\mathbb{C}}
\newcommand{\dee}{\partial}
\newcommand{\e}{\varepsilon}
\newcommand{\GL}{\mathrm{GL}}
\newcommand{\Hgp}{\mathrm{H}}
\newcommand{\HH}{\mathcal{H}}
\newcommand{\Mgp}{\mathrm{M}}
\newcommand{\MM}{\mathcal{M}}
\newcommand{\Ngp}{\mathrm{N}}
\newcommand{\NN}{\mathcal{N}}
\newcommand{\Ogp}{\mathrm{O}}
\newcommand{\OO}{\mathcal{O}}
\newcommand{\Pgp}{\mathrm{P}}
\newcommand{\pp}{\mathfrak{p}}
\newcommand{\PP}{\mathcal{P}}
\newcommand{\Q}{\mathbb{Q}}
\newcommand{\R}{\mathbb{R}}
\newcommand{\SO}{\mathrm{SO}}
\newcommand{\Scr}{\mathscr{S}}
\newcommand{\Sgp}{\mathrm{S}}
\newcommand{\Sp}{\mathrm{Sp}}
\newcommand{\Ugp}{\mathrm{U}}
\newcommand{\ur}{\mathrm{ur}}
\newcommand{\WW}{\mathcal{W}}
\newcommand{\Z}{\mathbb{Z}}
\newcommand{\Zgp}{\mathrm{Z}}
\DeclareMathOperator{\antidiag}{antidiag}
\DeclareMathOperator{\blockdiag}{blockdiag}
\DeclareMathOperator{\diag}{diag}
\DeclareMathOperator{\Hom}{Hom}
\DeclareMathOperator{\Ind}{Ind}
\DeclareMathOperator{\Mat}{Mat}
\DeclareMathOperator*{\Res}{Res}
\DeclareMathOperator{\sgn}{sgn}
\DeclareMathOperator{\Tr}{Tr}
\DeclareMathOperator{\vol}{vol}
\numberwithin{equation}{section}
\newtheorem{theorem}[equation]{Theorem}
\newtheorem{corollary}[equation]{Corollary}
\newtheorem{lemma}[equation]{Lemma}
\newtheorem{proposition}[equation]{Proposition}
\newtheorem*{problem}{Weak Test Vector Problem}
\newtheorem*{strongproblem}{Strong Test Vector Problem}
\theoremstyle{remark}
\newtheorem{remark}[equation]{Remark}
\newtheorem{example}[equation]{Example}
\theoremstyle{definition}
\newtheorem{definition}[equation]{Definition}

\begin{document}

\title{Test Vectors for Archimedean Period Integrals}

\author{Peter Humphries and Yeongseong Jo}

\address{Department of Mathematics, University of Virginia, Charlottesville, VA 22904, USA}

\email{\href{mailto:pclhumphries@gmail.com}{pclhumphries@gmail.com}}

\urladdr{\href{https://sites.google.com/view/peterhumphries/}{https://sites.google.com/view/peterhumphries/}}

\address{Department of Mathematics and Statistics, The University of Maine, Orono, ME 04469, USA}

\email{\href{mailto:jo.59@buckeyemail.osu.edu}{jo.59@buckeyemail.osu.edu}}



\begin{abstract}
We study period integrals involving Whittaker functions associated to generic irreducible Casselman--Wallach representations of $\GL_n(F)$, where $F$ is an archimedean local field. Via the archimedean theory of newforms for $\GL_n$ developed by the first author, we prove that newforms are weak test vectors for several period integrals, including the $\GL_n \times \GL_n$ Rankin--Selberg integral, the Flicker integral, and the Bump--Friedberg integral. By taking special values of these period integrals, we deduce that newforms are weak test vectors for Rankin--Selberg periods, Flicker--Rallis periods, and Friedberg--Jacquet periods. These results parallel analogous results in the nonarchimedean setting proven by the second author, which use the nonarchimedean theory of newforms for $\GL_n$ developed by Jacquet, Piatetski-Shapiro, and Shalika. By combining these archimedean and nonarchimedean results, we prove the existence of weak test vectors for certain global period integrals of automorphic forms.
\end{abstract}

\maketitle

\section{Introduction}

\subsection{Test Vectors for \texorpdfstring{$\GL_n \times \GL_n$}{GL\9040\231\80\327GL\9040\231} Rankin--Selberg Integrals}

A period integral of automorphic forms over a number field $F$ is said to be \emph{Eulerian} if it factorises as a product of period integrals over local fields. A quintessential example is the $\GL_n \times \GL_n$ Rankin--Selberg integral
\[\int\limits_{\Zgp_n(\A_F) \GL_n(F) \backslash \GL_n(\A_F)} \varphi_{\pi}(g) \varphi_{\sigma}(g) E(g,s;\Phi,\omega_{\pi}\omega_{\sigma}) \, dg\]
involving two automorphic forms $\varphi_{\pi}$ and $\varphi_{\sigma}$ lying in two automorphic representations $\pi$ and $\sigma$ of $\GL_n(\A_F)$ as well as a distinguished Eisenstein series $E(g,s;\Phi,\omega_{\pi}\omega_{\sigma})$ associated to a Schwartz--Bruhat function $\Phi \in \Scr(\A_F^n)$ and a product $\omega_{\pi}\omega_{\sigma}$ of central characters. If $\varphi_{\pi}$, $\varphi_{\sigma}$, and $\Phi$ are pure tensors, then by unfolding, this global period integral factorises as a product over all places $v$ of $F$ of local $\GL_n \times \GL_n$ Rankin--Selberg integrals
\[\Psi(s,W_{\pi_v},W_{\sigma_v},\Phi_v) \coloneqq \int\limits_{\Ngp_n(F_v) \backslash \GL_n(F_v)} W_{\pi_v}(g_v) W_{\sigma_v}(g_v) \Phi_v(e_n g_v) \left|\det g_v\right|_v^s \, dg_v,\]
where $W_{\pi_v}$ and $W_{\sigma_v}$ are the local Whittaker functions associated to $\varphi_{\pi}$ and $\varphi_{\sigma}$.

The local period integrals $\Psi(s,W_{\pi_v},W_{\sigma_v},\Phi_v)$ represent the local $\GL_n \times \GL_n$ Rankin--Selberg $L$-function $L(s,\pi_v \times \sigma_v)$, where $\pi_v$ and $\sigma_v$ are the generic irreducible admissible smooth representations of $\GL_n(F_v)$ occurring in the tensor product decompositions of $\pi$ and $\sigma$. More precisely, if $v$ is a nonarchimedean place with residue field of order $q$, then the quotient $\Psi(s,W_{\pi_v},W_{\sigma_v},\Phi_v) / L(s,\pi_v \times \sigma_v)$ is a polynomial in $q^s$ and $q^{-s}$, and in particular is entire. If $v$ is an archimedean place, then the quotient $\Psi(s,W_{\pi_v},W_{\sigma_v},\Phi_v) / L(s,\pi_v \times \sigma_v)$ is entire and of finite order in vertical strips.

While the quotient $\Psi(s,W_{\pi_v},W_{\sigma_v},\Phi_v) / L(s,\pi_v \times \sigma_v)$ is always entire regardless of the choice of Whittaker functions $W_{\pi_v}$ and $W_{\sigma_v}$ and Schwartz--Bruhat function $\Phi_v$, for many applications, one requires something stronger, namely that for particular choices of $W_{\pi_v}$, $W_{\sigma_v}$, and $\Phi_v$, this quotient be nicely behaved --- in particular, nonvanishing apart from a prescribed collection of values of $s \in \C$. When the representations $\pi_v$ and $\sigma_v$ are both unramified, there exists an explicit Schwartz--Bruhat function $\Phi_v$ for which this quotient is \emph{exactly} equal to $1$ when $W_{\pi_v}$ and $W_{\sigma_v}$ are chosen to be \emph{spherical} Whittaker functions. This motivates the following problem.

\begin{strongproblem}
Given generic irreducible admissible smooth representations $\pi_v$ and $\sigma_v$ of $\GL_n(F_v)$, determine the existence of Whittaker functions $W_{\pi_v}$ and $W_{\sigma_v}$ in the Whittaker models of $\pi_v$ and $\sigma_v$ and a Schwartz--Bruhat function $\Phi_v \in \Scr(F_v^n)$ for which
\[\Psi(s,W_{\pi_v},W_{\sigma_v},\Phi_v) = L(s,\pi_v \times \sigma_v).\]
\end{strongproblem}

We call such a triple $(W_{\pi_v},W_{\sigma_v},\Phi_v)$ a \emph{strong test vector} for the $\GL_n \times \GL_n$ Rankin--Selberg integral. In full generality, this problem remains unresolved, though some special cases have been settled:
\begin{itemize}[leftmargin=*]
\item When $n = 2$ and $F_v$ is nonarchimedean for several families of representations $\pi_v$ and $\sigma_v$ by Kim \cite[Chapters 4 and 5]{Kim10};
\item When $n = 2$ and $F_v$ is archimedean by Jacquet \cite[Theorem 7.2]{Jac72}, S-W.\ Zhang \cite[Proposition 2.5.2]{Zha01}, Miyazaki \cite[Theorem 6.1]{Miy18} and Hirano, Ishii, and Miyazaki \cite[Appendix A]{HIM21};
\item When at least one of the two representations $\pi_v$ and $\sigma_v$ is unramified by Kim \cite[Theorem 2.1.1]{Kim10} for $F_v$ nonarchimedean and by the first author \cite[Theorem 4.18]{Hum20a} for $F_v$ archimedean;
\item When both representations are principal series of prescribed forms and $F_v$ is archimedean by Ishii and Miyazaki \cite[Theorem 2.9]{IM20}.
\end{itemize}

We focus on a weaker yet more tractable problem, where we are satisfied with finding a triple $(W_{\pi_v},W_{\sigma_v},\Phi_v)$ for which $\Psi(s,W_{\pi_v},W_{\sigma_v},\Phi_v)$ is an explicit polynomial multiple of $L(s,\pi_v \times \sigma_v)$. Associated to $\pi_v$ and $\sigma_v$ are \emph{unramified} representations $\pi_{v,\ur}$ and $\sigma_{v,\ur}$ whose standard $L$-functions are such that $L(s,\pi_{v,\ur}) = L(s,\pi_v)$ and $L(s,\sigma_{v,\ur}) = L(s,\sigma_v)$. We then define the \emph{na\"{i}ve Rankin--Selberg $L$-function} associated to $\pi_v$ and $\sigma_v$ to be $L(s,\pi_{v,\ur} \times \sigma_{v,\ur})$, which we show is an explicit polynomial multiple of $L(s,\pi_v \times \sigma_v)$.

\begin{problem}
Given generic irreducible admissible smooth representations $\pi_v$ and $\sigma_v$ of $\GL_n(F_v)$, determine the existence of Whittaker functions $W_{\pi_v}$ and $W_{\sigma_v}$ in the Whittaker models of $\pi_v$ and $\sigma_v$ and a Schwartz--Bruhat function $\Phi_v \in \Scr(F_v^n)$ for which
\[\Psi(s,W_{\pi_v},W_{\sigma_v},\Phi_v) = L(s,\pi_{v,\ur} \times \sigma_{v,\ur}).\]
\end{problem}

We call such a triple $(W_{\pi_v},W_{\sigma_v},\Phi_v)$ a \emph{weak test vector} for the $\GL_n \times \GL_n$ Rankin--Selberg integral. The second author \cite[Theorem 1.1 (i)]{Jo21} resolved this problem when $F_v$ is nonarchimedean via the theory of nonarchimedean newforms due to Jacquet, Piatetski-Shapiro, and Shalika \cite{JP-SS81}. We resolve this problem when $F_v$ is archimedean via the theory of archimedean newforms introduced by the first author \cite{Hum20a}.

\begin{theorem}
\label{thm:GLnRS}
Let $F_v$ be an archimedean local field and let $\pi_v$ and $\sigma_v$ be generic irreducible admissible smooth representations of $\GL_n(F_v)$. Then there exist Whittaker functions $W_{\pi_v}$ and $W_{\sigma_v}$ in the Whittaker models of $\pi_v$ and $\sigma_v$ and a Schwartz--Bruhat function $\Phi_v \in \Scr(F_v^n)$ for which
\[\Psi(s,W_{\pi_v},W_{\sigma_v},\Phi_v) = L(s,\pi_{v,\ur} \times \sigma_{v,\ur}).\]
\end{theorem}

A more precise statement of this result is given in \hyperref[thm:GLnxGLn]{Theorem \ref*{thm:GLnxGLn}}.

\subsection{Test Vectors for Archimedean Period Integrals}
\label{sect:mainresults}

\hyperref[thm:GLnRS]{Theorem \ref*{thm:GLnRS}} is but one of many results that we prove pertaining to weak test vectors for archimedean period integrals. There are many local period integrals involving integrals of Whittaker functions that represent distinguished $L$-functions. As well as for the $\GL_n \times \GL_n$ Rankin--Selberg integral, we resolve the analogue of the weak test vector problem for the following:
\begin{itemize}[leftmargin=*]
\item The modified $\GL_n \times \GL_n$ Rankin--Selberg integral introduced by Sakellaridis given by \eqref{eqn:Saknn}. This represents $L(s_2,\pi_v \times \sigma_v) L(s_1 - \frac{n - 1}{2}, \sigma_v)$, the product of the $\GL_n \times \GL_n$ Rankin--Selberg $L$-function and the standard $L$-function. In \hyperref[thm:Saknn]{Theorem \ref*{thm:Saknn}}, we prove the existence of a weak test vector $(W_{\pi_v},W_{\sigma_v},\Phi_v)$ for which this integral is equal to $L(s_2,\pi_{v,\ur} \times \sigma_{v,\ur}) L(s_1 - \frac{n - 1}{2}, \sigma_v)$.
\item The modified $\GL_n \times \GL_{n - 1}$ Rankin--Selberg integral introduced by Sakellaridis given by \eqref{eqn:Saknn-1}. This represents $L(s_2 + \frac{1}{2},\pi_v \times \sigma_v) L(s_1 - \frac{n - 2}{2}, \sigma_v)$, the product of the $\GL_n \times \GL_{n - 1}$ Rankin--Selberg $L$-function and the standard $L$-function. In \hyperref[thm:Saknn-1]{Theorem \ref*{thm:Saknn-1}}, we prove that when $\sigma_v$ is unramified, there exists of a strong test vector $(W_{\pi_v},W_{\sigma_v},\Phi_v)$ for which this integral is equal to $L(s_2 + \frac{1}{2},\pi_v \times \sigma_v) L(s_1 - \frac{n - 1}{2}, \sigma_v)$.
\item The Flicker integral given by \eqref{eqn:Flickerint}. This represents the Asai $L$-function $L(s,\pi_v,\As)$. In \hyperref[thm:Flicker]{Theorem \ref*{thm:Flicker}}, we prove the existence of a weak test vector $(W_{\pi_v},\Phi_v)$ for which this integral is equal to $L(s,\pi_{v,\ur},\As)$.
\item The Bump--Friedberg integral given by \eqref{eqn:BFint}. This represents $L(s_1,\pi_v) L(s_2,\pi_v, \wedge^2)$, the product of the standard $L$-function and the exterior square $L$-function. In \hyperref[thm:BF-reduction]{Theorem \ref*{thm:BF-reduction}}, we prove the existence of a weak test vector $(W_{\pi_v},\Phi_v)$ for which this integral is equal to $L(s_1,\pi_v) L(s_2,\pi_{v,\ur}, \wedge^2)$.
\end{itemize}

The period integrals listed above either involve a single complex variable $s$ or a pair of complex variables $s_1,s_2$. In certain situations, these complex variables can be specialised to special values, at which point these period integrals are linear functionals for particular representations. We list below the periods of interest that arise from specialising these complex variables.
\begin{itemize}[leftmargin=*]
\item The $\GL_n \times \GL_n$ Rankin--Selberg period given by \eqref{eqn:RSperiod2}. When $\pi_v$ and $\sigma_v$ are unitary, we show in \hyperref[thm:RSperiod]{Theorem \ref*{thm:RSperiod}} the existence of a test vector $(W_{\pi_v},W_{\sigma_v})$ for which this period is equal to the special $L$-value $L(1,\pi_{v,\ur} \times \sigma_{v,\ur}) / L(n,\omega_{\pi_{v,\ur}} \omega_{\sigma_{v,\ur}})$, where $\omega_{\pi_{v,\ur}}$ and $\omega_{\sigma_{v,\ur}}$ denote the central characters of $\pi_{v,\ur}$ and $\sigma_{v,\ur}$.
\item The Flicker--Rallis period given by \eqref{eqn:FRperiod2}. When $\pi_v$ is unitary, we show in \hyperref[thm:FRperiod]{Theorem \ref*{thm:FRperiod}} the existence of a test vector $W_{\pi_v}$ for which this period is equal to the special $L$-value $L(1,\pi_{v,\ur},\As) / L(n, \omega_{\pi_{v,\ur}}|_{\R^{\times}})$.
\item The Friedberg--Jacquet period given by \eqref{eqn:FJperiod}. When $\pi_v$ is unitary, we show in \hyperref[thm:FJperiod]{Theorem \ref*{thm:FJperiod}} the existence of a test vector $W_{\pi_v}$ for which this period is equal to the special $L$-value $L(\frac{1}{2},\pi_v) L(1,\pi_{v,\ur},\wedge^2) / L(\frac{n}{2},\omega_{\pi_{v,\ur}})$.
\end{itemize}

\subsection{Comparisons between Archimedean and Nonarchimedean Results}
\label{sec:Intro-Comparision}

Our results on test vectors for archimedean period integrals parallel analogous results of the second author \cite{Jo21} in the nonarchimedean setting. The strategy of proof in both settings is similar: via the Iwasawa decomposition and the fact that Whittaker newforms transform on the right under the maximal compact subgroup $K_v$ of $\GL_n(F_v)$ in a prescribed manner, these integrals can be reduced to integrals over a diagonal torus $\Agp_n(F_v)$.

When $F_v$ is nonarchimedean, the behaviour of the Whittaker newform when restricted to the diagonal torus has a closed form via the Shintani--Casselman--Shalika formula \cites{CS80,Shi76} together with work of Matringe \cite{Mat13} and Miyauchi \cite{Miy14}. This allows one to \emph{directly} evaluate the period integrals of interest.

When $F_v$ is archimedean, on the other hand, no such closed form exists for the Whittaker newform when restricted to the diagonal torus. Instead, there exists a \emph{propagation formula} due to the first author \cite[Lemmata 9.8 and 9.17]{Hum20a} that gives a recursive formula for this Whittaker newform in terms of a Whittaker newform associated to a representation of $\GL_{n - 1}(F_v)$; for unramified representations, such an identity is due independently to Gerasimov, Lebedev, and Oblezin \cite[Proposition 4.1]{GLO08} and to Ishii and Stade \cite[Proposition 2.1]{IS13}. In this way, the period integrals of interest can be evaluated via \emph{induction}. This approach has been previously used to successfully solve the strong test vector problem for the unramified $\GL_n \times \GL_n$ Rankin--Selberg integral \cite[Theorem 1.1]{Sta02} and the unramified Bump--Friedberg integral \cite[Theorem 3.3]{Sta01}. For the modified $\GL_n \times \GL_n$ and $\GL_n \times \GL_{n - 1}$ Rankin--Selberg integrals introduced by Sakellaridis and for the Flicker integral, our results are new even for the unramified case.

Known results for weak test vectors for nonarchimedean period integrals encompass more than just the analogues of the archimedean results listed in \hyperref[sect:mainresults]{Section \ref*{sect:mainresults}}. Weak test vectors have been determined for $\GL_n \times \GL_m$ Rankin--Selberg integrals with $n > m$ \cite[Theorem 1.1]{BKL20}, for Jacquet--Shalika integrals \cite[Theorem 1.1]{MY13}, and for Bump--Ginzburg integrals \cite[Theorem 6.3]{Jo21}. In turn, these give weak test vectors for $\GL_n \times \GL_m$ Rankin--Selberg periods, Jacquet--Shalika periods, and Bump--Ginzburg periods.

In the archimedean setting, on the other hand, we do \emph{not} expect to be able to find such weak test vectors for these various period integrals, with the exception of the $\GL_n \times \GL_{n - 1}$ Rankin--Selberg integral, if we enforce the condition that these test vectors be right $K_v$-finite. Indeed, even if all representations are unramified and all Whittaker functions are spherical, it is widely believed that these period integrals are \emph{not} a nonzero polynomial multiple of the $L$-function associated to such a period integral \cite[\S 2.6]{Bum89}. A prototypical example of this phenomenon is the identity for the unramified $\GL_n \times \GL_{n - 2}$ archimedean Rankin--Selberg integral due to Ishii and Stade \cite[Theorem 3.2]{IS13}, namely
\begin{multline*}
\int\limits_{\Ngp_{n - 2}(\R) \backslash \GL_{n - 2}(\R)} W_{\pi_{\infty}} \begin{pmatrix} g_{\infty} & 0 \\ 0 & 1_2 \end{pmatrix} W_{\sigma_{\infty}}(g_{\infty}) \left|\det g_{\infty}\right|_{\R}^{s - 1} \, dg_{\infty}	\\
= L(s,\pi_{\infty} \times \sigma_{\infty}) \frac{1}{4\pi i} \int_{\sigma - i\infty}^{\sigma + i\infty} \frac{L\left(w,\widetilde{\pi_{\infty}}\right)}{L(s + w,\sigma_{\infty})} \, dw.
\end{multline*}

\subsection{Global Applications}

In conjunction with the nonarchimedean results on test vectors in \cite{Jo21}, our archimedean results have global applications. Each local period integral has a global counterpart, and we are able to show the existence of weak test vectors for these global period integrals. More precisely, we resolve the weak test vector problem for the following:
\begin{itemize}[leftmargin=*]
\item The global $\GL_n \times \GL_n$ Rankin--Selberg integral given by \eqref{eqn:globalRSint} in \hyperref[thm:globalRSint]{Theorem \ref*{thm:globalRSint}}.
\item The global modified $\GL_n \times \GL_n$ Rankin--Selberg integral introduced by Sakellaridis given by \eqref{RS-Sakellaridis-equalrank} in \hyperref[thm:Saknnglobal]{Theorem \ref*{thm:Saknnglobal}}.
\item The global modified $\GL_n \times \GL_{n - 1}$ Rankin--Selberg integral introduced by Sakellaridis given by \eqref{eqn:GLnXGLmRS} in \hyperref[thm:Saknn-1global]{Theorem \ref*{thm:Saknn-1global}}.
\item The global Flicker integral given by \eqref{eqn:Flickerintglobal} in \hyperref[thm:globalFlickerint]{Theorem \ref*{thm:globalFlickerint}}.
\item The global Bump--Friedberg integral given by \eqref{eqn:globalBFinteven} and \eqref{eqn:globalBFintodd} in \hyperref[thm:BF-reduction]{Theorem \ref*{thm:globalBF}}.
\end{itemize}
By specialising the value of the complex variables $s$ or $s_1,s_2$ in these period integrals, our results give weak test vectors for certain periods:
\begin{itemize}[leftmargin=*]
\item The global $\GL_n \times \GL_n$ Rankin--Selberg period given by \eqref{eqn:RSGlobalSV} in \hyperref[thm:RSGlobalSV]{Theorem \ref*{thm:RSGlobalSV}}.
\item The global Flicker--Rallis period given by \eqref{eqn:globalFRperiod} in \hyperref[thm:globalFRperiod]{Theorem \ref*{thm:globalFRperiod}}.
\item The global Friedberg--Jacquet period given by \eqref{eqn:globalFJperiod} in \hyperref[thm:globalFJperiod]{Theorem \ref*{thm:globalFJperiod}}.
\end{itemize}

\section{Preliminaries}

From here on, we work over an archimedean local field. Since all that follows is local, we drop the usage of the subscript $v$. We list below some standard facts on the representation theory of $\GL_n(F)$ with $F$ archimedean. Much of this is well-known and appears elsewhere in the literature; see in particular \cite[\S 2]{Hum20a}.

\subsection{Groups and Measures}

\subsubsection{Local Fields}

Let $F$ be an archimedean local field, so that $F$ is either $\R$ or $\C$. We denote by $|\cdot|_F$ the absolute value on $F$, which we normalise such that
\[
|x|_F = \begin{dcases*}
\max \{ x, -x \} & if $F=\R$, \\
x \overline{x} & if $F=\C$.
\end{dcases*}
\]
When the local field is clear from the context, we omit the subscript $F$ in our notation and simply write $|\cdot|=|\cdot|_F$. We also let $\|\cdot\|\coloneqq |\cdot|^{1/2}_{\C}$ denote the standard module on $\C$.

We fix a nontrivial additive character $\psi_F=\psi$ of $F$. For $F=\R$, we choose $\psi(x)\coloneqq  \exp(2\pi i x)$, while for $F=\C$, we choose $\psi(x)\coloneqq \exp(2\pi i(x+\overline{x}))$; in \hyperref[sect:Flickersect]{Section \ref*{sect:Flickersect}}, we will also work with a slightly different nontrivial additive character $\psi_{\C/\R}$ of $\C$ that is trivial when restricted to $\R$. We normalise the Haar measure $dx$ on $F$ so that it is self-dual with respect to $\psi$. For $F=\R$, $dx$ is simply the Lebesgue measure, whereas for $F=\C$, $dx$ is twice the Lebesgue measure. The multiplicative Haar measure $d^{\times} x$ on $F^{\times}$ is $\zeta_F(1)|x|^{-1}dx$, where
\[ \zeta_F(s) \coloneqq \begin{dcases*}
\pi^{-\frac{s}{2}} \Gamma \left(  \dfrac{s}{2}\right)   & if $F=\R$, \\
2(2 \pi)^{-s} \Gamma(s) & if $F=\C$.
\end{dcases*}\]

\subsubsection{Groups and Haar Measures}

We write $1_n$ to denote the $n \times n$ identity matrix. Let $\{ e_i. : 1 \leq i \leq n \}$ be the standard row basis of $F^n$. Let $\Pgp(F)=\Pgp_{(n_1,\ldots,n_r)}(F)$ denote the standard upper parabolic subgroup of $\GL_n(F)$ of type $(n_1,n_2,\ldots,n_r)$ with $n=n_1+n_2+\cdots+n_r$. This has the Levi decomposition $\Pgp(F)= \Ngp_{\Pgp}(F) \Mgp_{\Pgp}(F)$, where the block-diagonal Levi subgroup $\Mgp_{\Pgp}(F)$ is isomorphic to $\GL_{n_1}(F) \times \cdots \times \GL_{n_r}(F)$, while the unipotent radical $\Ngp_{\Pgp}(F)$ of $\Pgp(F)$ consists of upper triangular matrices with block-diagonal entries $(1_{n_1},\ldots,1_{n_r})$. When $\Pgp(F)$ is the standard Borel (and minimal parabolic) subgroup $\Pgp_{(1,\ldots,1)}(F)\eqqcolon\Bgp_n(F)$ of upper triangular matrices, we write $\Ngp_{\Pgp}(F) \eqqcolon \Ngp_n(F) \cong F^{n(n-1)/2}$, the subgroup of unipotent upper triangular matrices, and $\Mgp_{\Pgp}(F)\eqqcolon \Agp_n(F) \cong (F^{\times})^n$, the subgroup of diagonal matrices. We write $\Zgp_n(F) \cong F^{\times}$ to denote to the centre consisting of scalar matrices. We define $\Pgp_n(F)$, the mirabolic subgroup of $\GL_n(F)$, given by
\[
 \Pgp_n(F)= \left\{ \begin{pmatrix} h  & \prescript{t}{}{x} \\ 0 & 1 \end{pmatrix} : h \in \GL_{n-1}(F), \ x \in F^{n-1} \right\}.
\]
The modulus character of a parabolic subgroup $\Pgp(F)=\Pgp_{(n_1,\ldots,n_r)}(F)$ is
\[
  \delta_{\Pgp(F)}(m)=\prod_{j=1}^r \left|\det m_j\right|^{n-2(n_1+n_2+ \cdots +n_{j-1})-n_j}
\]
for any $m=\blockdiag(m_1,\ldots,m_r) \in \Mgp_{\Pgp}(F)$. We let $K_n$ denote the maximal compact subgroup of $\GL_n(F)$, unique up to conjugacy, so that
\[
K_n= \begin{dcases*}
\Ogp(n) & if $F=\R$, \\
\Ugp(n)  & if $F=\C$.
\end{dcases*}
\]

We have the Iwasawa decomposition $\GL_n(F)=\Ngp_n(F)\Agp_n(F)K_n$. We normalise the Haar measure $dg$ on $\GL_n(F) \ni g$ via the Iwasawa decomposition $g=uak$ with $u \in \Ngp_n(F)$, $a \in \Agp_n(F)$, and $k \in K_n$, so that
\[
dg = \delta^{-1}_{\Bgp_n(F)}(a) \, du \, d^{\times}a \, dk.
\]
Here $du = \prod_{1 \leq i < j \leq n} du_{i,j}$, where $du_{i,j}$ is the Haar measure on $F$, while $d^{\times}a = \prod_{j = 1}^{n} d^{\times}a_j$, where $d^{\times}a_j$ is the Haar measure on $F^{\times}$, and $dk$ is the probability Haar measure on $K_n$. The modulus character of $\Bgp_n(F)$ is simply $\delta_{\Bgp_n(F)}(a) = \prod_{j = 1}^{n} |a_j|^{n-2j+1}$.

The Iwasawa decomposition also gives rise to the Haar measure on $\Ngp_n(F) \backslash \GL_n(F)$ via $dg = \delta_{B_n(F)}^{-1}(a) \, d^{\times}a \, dk$ for $g = ak$ with $a \in \Agp_n(F)$ and $k \in K_n$. There is an alternate expression for this Haar measure via the Iwasawa decomposition $g = (z1_n) \begin{psmallmatrix} h & 0 \\ 0 & 1 \end{psmallmatrix} k$ associated to the maximal parabolic subgroup $\Pgp_{(n - 1,1)}(F)$ of $\GL_n(F)$ of type $(n - 1,1)$. Here $z \in F^{\times}$, $h \in \Ngp_{n - 1}(F) \backslash \GL_{n - 1}(F)$, and $k \in K_n$; the Haar measure on $\Ngp_n(F) \backslash \GL_n(F)$ becomes
\[dg = \left|\det h\right|^{-1} \, d^{\times}z \, dh \, dk,\]
where $d^{\times}z$ denotes the multiplicative Haar measure on $F^{\times}$ and $dh$ denotes the Haar measure on $\Ngp_{n - 1}(F) \backslash \GL_{n - 1}(F)$.

\subsection{Representations of \texorpdfstring{$\GL_n(F)$}{GL\9040\231(F)}}

\subsubsection{Isobaric Sums}

Given representations $(\pi_1,V_{\pi_1}), \ldots, (\pi_r,V_{\pi_r})$ of $\GL_{n_1}(F), \ldots, \GL_{n_r}(F)$ with $n=n_1+\cdots+n_r$, we form the representation $\bigboxtimes_{j = 1}^{r} \pi_j$ of $\Mgp_{\Pgp}(F)$, where $\boxtimes$ denote the outer tensor product. We then extend this representation trivially to a representation of $\Pgp(F)$. We obtain a normalised parabolically induced representation $(\pi,V_{\pi})$ of $\GL_n(F)$ by
\[
\pi\coloneqq  \Ind^{\GL_n(F)}_{\Pgp(F)} \bigboxtimes_{j=1}^r \pi_j,
\]
where $V_{\pi}$ denote the space of smooth functions $f : \GL_n(F) \to V_{\pi_1} \otimes \cdots \otimes V_{\pi_r}$, upon which $\pi$ on $V_{\pi}$ via  right translation, namely, $(\pi(h)\cdot f)(g)\coloneqq f(gh)$, that satisfy
\[
f(umg)=\delta^{1/2}_{\Pgp(F)}(m) \bigotimes_{j=1}^r \pi_j(m_j)\cdot f(g),
\]
for any $u \in \Ngp_{\Pgp}(F)$, $m=\blockdiag(m_1,\ldots,m_r) \in \Mgp_{\Pgp}(F)$, and $g \in \GL_n$. The induced representation $\pi$ is called the \emph{isobaric sum} of $\pi_1,\ldots, \pi_r$, which we denote by
\[
\pi\coloneqq \bigboxplus_{j = 1}^{r} \pi_j.
\]

\subsubsection{Essential Square-Integrable Representations}

Essential square-integrable representations of $\GL_n(\C)$ exist only for $n = 1$. An essentially square-integrable representation of $\GL_1(\R)=\R^{\times}$ must be a character of the form $\pi(x)=e^{i\kappa \arg(x)} |x|^t_{\C}$ for some $\kappa \in \Z$ and $t \in \C$, where $e^{i\arg(x)} \coloneqq x/\|x\|$.

Essential square-integrable representations of $\GL_n(\R)$ exist only for $n \in \{ 1, 2\}$. An essentially square-integrable representation of $\GL_1(\R)=\R^{\times}$ must be a character of the form $\pi(x)=\sgn(x)^{\kappa}|x|^t_{\R}$ for some $\kappa \in \{0,1 \}$ and $t \in \C$, where $\sgn(x) \coloneqq x/\|x\|$. 
We view $\GL_1(\C)$ as a subgroup of $\GL_2(\R)$ via the identification $a+ib \mapsto \begin{psmallmatrix} a & b \\ -b & a \end{psmallmatrix}$. For $\kappa \neq 0$, the essential discrete series representation of weight $\|\kappa \|+1$,
\[
D_{\|\kappa \|+1} \otimes \left|\det\right|^t_{\R}\coloneqq \Ind^{\GL_2(\R)}_{\GL_1(\C)} e^{i\kappa \arg}|\cdot|^t_{\C} \cong \Ind^{\GL_2(\R)}_{\GL_1(\C)} e^{-i\kappa\arg}|\cdot|^t_{\C}
\]
is essentially square-integrable. Every essentially square-integrable representation of $\GL_2(\R)$ is of the form $\pi=D_{\kappa}\otimes \left|\det\right|^t_{\R}$ for some integer $\kappa \geq 2$ and $t \in \C$.

\subsubsection{Induced Representations of Whittaker and Langlands Types}

A representation $\pi$ of $\GL_n(F)$ is said to be an \emph{induced representation of Whittaker type} if it is the isobaric sum of $\pi_1,\pi_2,\ldots,\pi_r$ and each $\pi_j$ is essentially square-integrable. Such a representation is  an admissible smooth Fr\'{e}chet representation of moderate growth and of finite length. Induced representations of Whittaker type are Casselman--Wallach representations \cites{Cas89,Wal92}, which is to say admissible smooth Fr\'{e}chet representations of moderate growth and finite length. In addition, if each $\pi_j$ is of the form $\sigma_j \otimes \left|\det\right|^{t_j}$, where $\sigma_j$ is irreducible, unitary, square-integrable, and $\Re(t_1) \geq \Re(t_2) \geq \cdots \geq \Re(t_r)$, then $\pi$ is said to be an \emph{induced representation of Langlands type}.

Induced representations of Whittaker and Langlands type need not be irreducible. Nonetheless, every generic irreducible Casselman--Wallach representation of $\GL_n(F)$ is isomorphic to some (necessarily irreducible) induced representation of Langlands type. For this reason, we will work more generally with induced representations of Langlands type, since this encompasses generic irreducible Casselman--Wallach representations.

A \emph{spherical induced representation of Whittaker type} of $\GL_n(F)$ is a representation of the form $\pi = \bigboxplus_{j = 1}^{n} \pi_j$ with each $\pi_j$ an unramified character of $F^{\times}$, namely a character of the form $\pi_j = |\cdot|^{t_j}$ with $t_j \in \C$. Such a representation has a $K_n$-fixed vector, which is unique up to scalar multiplication; this is called the \emph{spherical vector}.

\subsubsection{The Whittaker Model}

Let $(\pi,V_{\pi})$ be an induced representation of Whittaker type. We let $\psi_n$ denote an additive character of $\Ngp_n(F)$ defined by
\[
\psi_n(u)\coloneqq \psi \left( \sum_{j=1}^{n-1} u_{j,j+1} \right)
\]
for all $u \in \Ngp_n(F)$. A \emph{Whittaker functional} $\Lambda : V_{\pi} \to \C$ is a continuous linear functional that satisfies
\[
\Lambda(\pi(u) \cdot v)=\psi_n(u) \Lambda(v)
\]
for all $v \in V_{\pi}$ and $u \in \Ngp_n(F)$. If $\pi$ is additionally irreducible, then the space ${\Hom}_{\Ngp_n(F)}(\pi,\psi)$ of Whittaker functionals of $\pi$ is at most one-dimensional. If the space is one-dimensional, it admits a unique functional up to scalar multiplication, and the representation $\pi$ is said to be \emph{generic}.

Let $\WW(\pi,\psi)$ denote the \emph{Whittaker model} of $\pi$, which is the image of $V_{\pi}$ under the map $v \mapsto \Lambda(\pi(\cdot)\cdot v)$. The Whittaker model $\WW(\pi,\psi)$ consists of Whittaker functions $W : \GL_n(F) \to \C$ of the form $W(g)\coloneqq \Lambda(\pi(g) \cdot v)$. Every induced representation of Langlands type is generic and isomorphic to its unique Whittaker model $\WW(\pi,\psi)$. Although an induced representation of Whittaker type $\pi$ affords a one-dimensional space of Whittaker functionals, the map $v \mapsto \Lambda(\pi(\cdot)\cdot v)$ need not to be injective, so the Whittaker model may only be a model of a quotient of $\pi$.

Let $\pi = \bigboxplus_{j = 1}^{r} \pi_j$ be an induced representation of Whittaker type of $\GL_n(F)$, so that each $\pi_j$ is of the form $e^{i\kappa_j \arg} |\cdot|^{t_j}$ if $F = \C$ or either $\sgn^{\kappa_j}|\cdot|^{t_j}$ or $D_{\kappa_j} \otimes \left|\det\right|^{t_j}$ if $F = \R$. The induced model of $\pi$ consists of smooth functions $f : \GL_n(F) \to \bigotimes_{j = 1}^{r} V_{\pi_j}$. Each $V_{\pi_j}$ may itself be identified with a space of smooth functions from $\GL_{n_j}(F)$ to $\C$ \cite[Lemma 8.1]{Hum20a}. Evaluating such a function at the identity $1_{n_j}$ for each $j \in \{1,\ldots,r\}$, we may thereby view an element $f$ of the induced model $V_{\pi}$ of $\pi$ as a smooth function from $\GL_n(F)$ to $\C$ \cite[Corollary 8.2]{Hum20a}.

Given $f \in V_{\pi}$, we define the \emph{Jacquet integral} 
\[
W_f(g)\coloneqq \int_{\Ngp_n(F)} f(w_nug) \overline{\psi_n}(u) \, du,
\]
where $w_n \coloneqq \antidiag(1,\ldots,1)$ denotes the long Weyl element in $\GL_n(F)$. This integral converges absolutely if $\Re(t_1) > \Re(t_2) > \cdots > \Re(t_r)$ and defines a Whittaker function $W_f \in \WW(\pi,\psi)$; that is, as a function of $f \in V_{\pi}$, $\Lambda(f)\coloneqq W_f(1_n)$ defines a Whittaker functional, which is therefore unique up to scalar multiplication. Moreover, the Jacquet integral provides a Whittaker functional for \emph{all} induced representations of Whittaker type, and not just those for which $\Re(t_1) > \Re(t_2) > \cdots > \Re(t_r)$, via analytic continuation in the sense of Wallach \cite{Wal92}.

\subsubsection{$L$-Functions}

We put
\[d_F\coloneqq [F:\R] = \begin{dcases*}
1 & if $F = \R$,	\\
2 & if $F = \C$.
\end{dcases*}\]
The integral representation of the zeta function $\zeta_F(s)$ is
\[
\zeta_F(s)=\int_{F^{\times}} \exp(-d_F\pi\|x\|^2) |x|^s \, d^{\times}x,
\]
which converges absolutely for $\Re(s) > 0$ and extends meromorphically to the entire complex plane. In particular, if $\omega$ is an unramified character of $F^{\times}$, so that $\omega = |\cdot|^t$ for some $t \in \C$, then for $\Re(s) > -\Re(t)$, we have that
\begin{equation}
\label{eqn:Lsomega}
L(s,\omega) = \int_{F^{\times}} \omega(x) \exp(-d_F\pi\|x\|^2) |x|^s \, d^{\times}x.
\end{equation}

Given an induced representation of Whittaker type $\pi = \bigboxplus_{j = 1}^{r} \pi_j$ of $\GL_n(F)$, the local Langlands correspondence, as explicated by Knapp \cite{Kna94}, gives that the standard $L$-function of $\pi$ \cite{GJ72} is
\[L(s,\pi) = \prod_{j = 1}^{r} L(s,\pi_j),\]
where the $L$-function of the essentially square-integrable representation $\pi_j$ is
\[L(s,\pi_j) = \begin{dcases*}
\zeta_{\C}\left(s + t_j + \frac{\|\kappa_j\|}{2}\right) & if $F = \C$ and $\pi_j = e^{i\kappa_j \arg} |\cdot|_{\C}^{t_j}$,	\\
\zeta_{\R}(s + t_j + \kappa_j) & if $F = \R$ and $\pi_j = \sgn^{\kappa_j} |\cdot|_{\R}^{t_j}$,	\\
\zeta_{\R}\left(s + t_j + \frac{\kappa_j - 1}{2}\right) \zeta_{\R}\left(s + t_j + \frac{\kappa_j + 1}{2}\right) & if $F = \R$ and $\pi_j = D_{\kappa_j} \otimes \left|\det\right|_{\R}^{t_j}$.
\end{dcases*}\]

\section{Archimedean Newform Theory}

We now survey the theory of newforms for induced representations of Whittaker type of $\GL_n(F)$ introduced by the first author \cite{Hum20a}. While newforms over nonarchimedean fields are usually defined in terms of vectors fixed under certain congruence subgroups, newforms over archimedean fields are instead defined in terms of vectors lying in distinguished $K_n$-types. We first recall some properties of representations of $K_n$, as well as properties of distinguished models of such representations.

\subsection{Representation Theory of \texorpdfstring{$K_n$}{K\9040\231}}

The equivalence classes of finite-dimensional irreducible representations of the orthogonal group
\[\Ogp(n)\coloneqq \{ k \in \Mat_{n \times n}(\R) : k \prescript{t}{}{k} = 1_n\}\]
are parametrised by the set of highest weights, which may be identified with $n$-tuples of nonnegative integers of the form
\[\mu = \{\mu_1,\ldots,\mu_m, \underbrace{\eta,\ldots,\eta}_{\text{$n-2m$ times}}, \underbrace{0,\ldots,0}_{\text{$m$ times}}\} \in \Z^n,\]
where $m \in \{0,\ldots,\lfloor \frac{n}{2}\rfloor \}$, $\mu_1 \geq \cdots \geq \mu_m \geq 1$, and $\eta \in \{0,1\}$.

Similarly, the equivalence class of finite-dimensional irreducible representations of the unitary group
\[
\Ugp(n)\coloneqq \{ k \in \Mat_{n \times n}(\C) : k \prescript{t}{}{\overline{k}} =1_n  \}
\]
are parametrised by the set of highest weights, which may be identified with $n$-tuples of integers $\mu=(\mu_1,\ldots,\mu_n) \in \Z^n$ that are nonincreasing, so that $\mu_1 \geq \cdots \geq \mu_n$.

In both settings, to each $\tau \in \widehat{K_n}$, the set of equivalence class of irreducible representations of $K_n$, one can associate a nonnegative integer $\deg \tau$ called the \emph{Howe degree} of $\tau$ \cite{How89}. The Howe degree of an irreducible representation $\tau$ of highest weight $\mu$ is
\[
\deg \tau \coloneqq \sum_{j=1}^n \| \mu_j \|.
\]

When $n = 1$, these irreducible representations are simply characters. Characters $\chi$ of $K_1$ are of the form
\[\chi(x) = \begin{dcases*}
\sgn(x)^{\kappa} & if $K_1 = \Ogp(1)$,	\\
e^{i\kappa \arg(x)} & if $K_1 = \Ugp(1)$,
\end{dcases*}\]
where $\kappa \in \{0,1\}$ if $K_1=\Ogp(1) \cong \Z/2\Z$ and $\kappa \in \Z$ if $K_1=\Ugp(1) \cong \R \slash \Z \cong S^1$. In either case, the conductor exponent of $\chi$ is $c(\chi) \coloneqq \|\kappa\|$.

\subsection{Spaces of Homogeneous Harmonic Polynomials}
\label{sec:HHPR}

Let $\chi$ be a character of $\Ogp(1)$. We let $m$ be a nonnegative integer for which $m \geq c(\chi)$ and $m \equiv c(\chi) \pmod{2}$. Let $\PP_{\chi,m}(\R^n)$ denote the space of degree $m$ homogeneous polynomials with central character $\chi$, and let $\HH_{\chi,m}(\R^n)$ denote the subspace of $\PP_{\chi,m}(\R^n)$ of harmonic homogeneous polynomials; $\HH_{\chi,m}(\R^n)$ is a model for the irreducible representation of $\Ogp(n)$ with central character $\chi$ and highest weight $(m,0,\ldots,0)$, which we denote by $\tau_{\chi,m}$.

Similarly, let $\chi$ be a character of $\Ugp(1)$. We let $m$ be a nonnegative integer for which $m \geq c(\chi)$ and $m \equiv c(\chi) \pmod{2}$. Let $\PP_{\chi,m}(\C^n)$ denote the space of degree $m$ homogeneous polynomials with central character $\chi$, and let $\HH_{\chi,m}(\C^n)$ denote the subspace of $\PP_{\chi,m}(\C^n)$ of harmonic homogeneous polynomials; $\HH_{\chi,m}(\C^n)$ is a model for the irreducible representation of $\Ugp(n)$ with central character $\chi$ and highest weight $(m_1,0,\ldots,0,-m_2)$, which we also denote by $\tau_{\chi,m}$. The nonnegative integers $m_1,m_2$ are such that $m_1 + m_2 = m$ and $m_1 - m_2 = \ell$, where $\chi = e^{i\ell \arg}$, so that $c(\chi) = \max\{\ell,-\ell\}$, and elements of $\PP_{\chi,m}(\C^n)$ have bidegree $(m_1,m_2)$.

In both cases, every irreducible representation of $K_n$ whose restriction to $K_{n - 1}$ contains the trivial representation is of the form $\tau_{\chi,m}$ for some character $\chi$ of $K_1$ and some integer $m \geq c(\chi)$; moreover, the Howe degree $\deg \tau_{\chi,m}$ of $\tau_{\chi,m}$ is simply the nonnegative integer $m$. The action of the group $K_n \ni k$ on the space $\PP_{\chi,m}(F^n) \ni P$ is via right translation, namely $(\tau_{\chi,m}(k)\cdot P)(x) \coloneqq P(xk)$. The homogeneity of a polynomial $P$ in $\PP_{\chi,m}(F^n)$ means that
\begin{equation}
\label{eqn:homogeneity}
P(\lambda x) = \chi\left(\frac{\lambda}{\|\lambda\|}\right) \|\lambda\|^m P(x)
\end{equation}
for all $x \in F^n$ and $\lambda \in F^{\times}$, while polynomials in $\HH_{\chi,m}(F^n)$ are additionally annihilated by the Laplacian
\[\Delta = \begin{dcases*}
\sum_{j=1}^n\frac{\dee^2}{\dee x^2_j} & if $F = \R$,	\\
4 \sum_{i=1}^n \frac{\dee^2}{\dee x_j \dee \overline{x_j}} & if $F = \C$.
\end{dcases*}\]

We record the following key properties of homogeneous polynomials in $\PP_{\chi,m}(F^n)$.

\begin{proposition}[{Cf.\ \cite[Lemmata 7.1 and 7.7]{Hum20a}}]
\label{prop:Pcircunique}
There exists a unique $K_{n - 1}$-invariant polynomial $P_{\chi,m}^{\circ}$ in $\HH_{\chi,m}(F^n)$ satisfying $P_{\chi,m}^{\circ}(e_n) = 1$, where the group $K_{n - 1}$ is embedded in $K_n$ via $k' \mapsto \begin{psmallmatrix} k' & 0 \\ 0 & 1  \end{psmallmatrix}$. In particular, for all $k \in K_n$, we have that
\begin{equation}
\label{eqn:Pconj}
P_{{\chi},m}^{\circ}(e_n k)=\overline{P_{\chi,m}^{\circ}}(e_n k^{-1}).
\end{equation}
\end{proposition}

We define an inner product on $\PP_{\chi,m}(F^n) \ni P,Q$ via
\[
\langle P,Q \rangle \coloneqq \int_{K_n} P(e_n k) \overline{Q}(e_nk) \, dk.
\]

Our first utilisation of $P^{\circ}_{{\chi},m}$ is the following, which is known as the addition theorem for $\HH_{\chi,m}(F^n)$.

\begin{proposition}[{Cf.\ \cite[Theorem 2.9]{AH12}}]
\label{prop:addition}
Let $\{ Q_{\ell}\}$ be an orthonormal basis of $\HH_{\chi,m}(F^n)$. Then for any $x \in \R^n$ and $k \in K_n$, we have that
\[
\sum_{\ell=1}^{\dim \tau_{\chi,m}} Q_{\ell}(x) \overline{Q_{\ell}}(e_nk)=\dim \tau_{\chi,m} P_{{\chi},m}^{\circ}(x k^{-1}).
\]
\end{proposition}

We make crucial use of the fact that for all $P \in \HH_{\chi,m}(F^n)$ and $k \in K_n$, $P(e_nk)$ is equal to a matrix coefficient of $\tau_{\chi,m}$. This can be thought of as an explicit form of Schur orthogonality. 

\begin{proposition}[{Cf.\ \cite[Lemmata 7.5 and 7.12]{Hum20a}}]
The reproducing kernel for $\HH_{\chi,m}(F^n)$ is the homogeneous harmonic polynomial $(\dim \tau_{\chi,m}) P^{\circ}_{\chi,m}(x)$, while the reproducing kernel for $\PP_{\chi,\ell}(F^n)$ is the homogeneous polynomial
\begin{equation}
\label{eqn:Preproducing}
\sum_{\substack{m = c(\chi) \\ m \equiv c(\chi) \hspace{-.2cm} \pmod{2}}}^{\ell} (x \prescript{t}{}{\overline{x}})^{\frac{\ell - m}{2}} ( \dim \tau_{\chi,m}) P^{\circ}_{\chi,m}(x),
\end{equation}
so that for all $k \in K_n$,
\begin{align}
\label{eqn:Hmatrixcoeff}
P(e_n k) & = \int_{K_n} P(e_n k' k) (\dim \tau_{\chi,m}) \overline{P_{\chi,m}^{\circ}}(e_n k') \, dk' & & \text{for all $P \in \HH_{\chi,m}(F^n)$,}	\\
\label{eqn:Pmatrixcoeff}
P(e_n k) & = \int_{K_n} P(e_n k' k) \sum_{\substack{m = c(\chi) \\ m \equiv c(\chi) \hspace{-.2cm} \pmod{2}}}^{\ell}  ( \dim \tau_{\chi,m}) \overline{P_{\chi,m}^{\circ}}(e_n k') \, dk' & & \text{for all $P \in \PP_{\chi,\ell}(F^n)$.}
\end{align}
\end{proposition}

\begin{proof}
The first assertion is justified in \cite[Lemmata 7.5 and 7.12]{Hum20a}. The second assertion follows from the first upon recalling that every homogeneous polynomial $P \in \PP_{\chi,\ell}(F^n)$ of degree $\ell$ and central character $\chi$ admits a decomposition of the form
\[P(x) = \sum_{\substack{m = c(\chi) \\ m \equiv c(\chi) \hspace{-.2cm} \pmod{2}}}^{\ell} \left(x \prescript{t}{}{\overline{x}}\right)^{\frac{\ell - m}{2}} P_m(x)\]
for some harmonic homogeneous polynomials $P_m \in \HH_{\chi,m}(F^n)$ \cite[Theorem 12.1.3]{Rud08}.
\end{proof}

\subsection{Archimedean Newform Theory}  

Let $(\pi,V_{\pi})$ be an induced representation of Whittaker type of $\GL_n(F)$. Since $\pi$ is admissible, $\Hom_{K_n}(\tau,\pi|_{K_n})$ is finite-dimensional for any irreducible representation $\tau$ of $K_n$. We say that $\tau$ is a \emph{$K_n$-type of $\pi$} if $\Hom_{K_n}(\tau,\pi|_{K_n})$ is nontrivial, and we call $\dim \Hom_{K_n}(\tau,\pi|_{K_n})$ the \emph{multiplicity} of $\tau$ in $\pi$. The fundamental result proven in \cite{Hum20a} is the existence of a distinguished $K_n$-type of $\pi$ that occurs with multiplicity one.

\begin{theorem}[{\cite[Theorem 4.7]{Hum20a}}]
Let $(\pi,V_{\pi})$ be an induced representation of Whittaker type of $\GL_n(F)$. Among the $K_n$-types $\tau_{\chi,m}$ of $\pi$ whose restriction to $K_{n - 1}$ contains the trivial representation, there exists a unique such $K_n$-type of minimal Howe degree $m$. Furthermore, this $K_n$-type $\tau_{\chi,m}$ occurs with multiplicity one, and the subspace of $V_{\pi}$ of $\tau_{\chi,m}$-isotypic $K_{n - 1}$-invariant vectors is one-dimensional.
\end{theorem}

\begin{remark}
By considering the restrictions of $\pi$ and $\tau_{\chi,m}$ to the centre of $K_n$, we observe that $\Hom_{K_n}(\tau_{\chi,m}, \pi|_{K_n})$ is trivial if the central character $\chi$ of $\tau_{\chi,m}$ is not equal to
\[\chi_{\pi} \coloneqq \omega_{\pi}|_{K_1},\]
the restriction of the central character $\omega_{\pi}$ of $\pi$ to $K_1$.
\end{remark}

\begin{definition}[{\cite[Definition 4.8]{Hum20a}}]
Let $(\pi,V_{\pi})$ be an induced representation of Whittaker type of $\GL_n(F)$. We define the \emph{newform $K_n$-type} $\tau_{\chi_{\pi},c(\pi)}$ to be the $K_n$-type of minimal Howe degree $m = c(\pi)$ whose restriction to $K_{n - 1}$ contains the trivial representation. We define the \emph{conductor exponent} $c(\pi)$ of $\pi$ to be the Howe degree of the newform $K_n$-type. The nonzero $\tau_{\chi_{\pi},c(\pi)}$-isotypic $K_{n - 1}$-invariant vector $v^{\circ} \in V_{\pi}$, unique up to scalar multiplication, is called the \emph{newform} of $\pi$.
\end{definition}

\begin{remark}
As proven by the first author in \cite{Hum20b}, this definition of the conductor exponent and the newform, as well as the existence of the newform $K_n$-type, is consistent with the \emph{nonarchimedean} definition of the conductor exponent and the newform first introduced by Jacquet, Piatetski-Shapiro, and Shalika \cite{JP-SS81}.
\end{remark}

When $V_{\pi}$ is the induced model of $\pi$, we may normalise the newform $v^{\circ} \in V_{\pi}$ as in \cite[Corollary 8.17 and Definition 9.2]{Hum20a}, which gives us a \emph{canonically normalised newform}. When $\pi$ is an induced representation of Langlands type, the Whittaker model $\WW(\pi,\psi)$ is a model of $\pi$, which is given by the analytic continuation of the Jacquet integral of the induced model. The image of the canonically normalised newform under this map is called the \emph{Whittaker newform} and denoted by $W_{\pi}^{\circ}$ \cite[\S 9.1]{Hum20a}. When $\pi$ is unramified, the Whittaker newform $W_{\pi}^{\circ}$ is simply the spherical Whittaker function.

The following lemma may be thought as the archimedean analogue of \cite[Theorem 4.16]{Hum20b}.

\begin{lemma}[{Cf.\ \cite[Theorem 4.17]{Hum20a}}]
\label{lem:Whitnewforminvariance}
For all $g \in \GL_n(F)$ and $k' \in K_{n - 1}$, the Whittaker newform $W_{\pi}^{\circ} \in \WW(\pi,\psi)$ of an induced representation of Langlands type $(\pi,V_{\pi})$ of $\GL_n(F)$ satisfies
\begin{align}
\label{eqn:Whitnewformintegral}
\dim \tau_{\chi_{\pi},c(\pi)} \int_{K_n} W_{\pi}^{\circ}(gk) P_{\chi_{\pi},c(\pi)}^{\circ}(e_n k^{-1}) \, dk & = W_{\pi}^{\circ}(g),	\\
\label{eqn:WhitnewformKn-1}
W^{\circ}_{\pi} \left(g \begin{pmatrix} k' & 0 \\ 0 & 1 \end{pmatrix} \right) & = W^{\circ}_{\pi}(g).
\end{align}
\end{lemma}

\begin{proof}
The identity \eqref{eqn:WhitnewformKn-1} is simply the fact that $W_{\pi}^{\circ}$ is right $K_{n - 1}$-invariant. Since $W_{\pi}^{\circ}$ is additionally $\tau_{\chi_{\pi},c(\pi)}$-isotypic, it satisfies
\[\dim \tau_{\chi_{\pi},c(\pi)} \int_{K_n} W_{\pi}^{\circ}(gk) \Tr \tau_{\chi_{\pi},c(\pi)}(k^{-1}) \, dk = W_{\pi}^{\circ}(g)\]
for all $g \in \GL_n(F)$. Upon replacing $g$ with $g \begin{psmallmatrix} k' & 0 \\ 0 & 1 \end{psmallmatrix}$, integrating over $K_{n - 1} \ni k'$, and making the change of variables $k' \mapsto k'^{-1}$ and $k \mapsto \begin{psmallmatrix} k' & 0 \\ 0 & 1 \end{psmallmatrix} k$, we deduce that
\[\dim \tau_{\chi_{\pi},c(\pi)} \int_{K_n} W_{\pi}^{\circ}(gk) \int_{K_{n - 1}} \Tr \tau_{\chi_{\pi},c(\pi)}\left(k^{-1} \begin{pmatrix} k' & 0 \\ 0 & 1 \end{pmatrix}\right) \, dk' \, dk = W_{\pi}^{\circ}(g).\]
From \hyperref[prop:Pcircunique]{Proposition \ref*{prop:Pcircunique}} and \eqref{eqn:Hmatrixcoeff}, the inner integral is equal to
\[\frac{\left\langle \tau_{\chi_{\pi},c(\pi)}(k^{-1}) \cdot P_{\chi_{\pi},c(\pi)}^{\circ}, P_{\chi_{\pi},c(\pi)}^{\circ}\right\rangle}{\left\langle P_{\chi_{\pi},c(\pi)}^{\circ}, P_{\chi_{\pi},c(\pi)}^{\circ}\right\rangle} = P_{\chi_{\pi},c(\pi)}^{\circ}(e_n k^{-1}).\qedhere\]
\end{proof}

\section{Langlands Parameters}

We now associate an induced representation of Langlands type $\pi$ of $\GL_n(F)$ to a distinguished spherical induced representation of Langlands type $\pi_{\ur}$ of $\GL_n(F)$ defined in terms of the Langlands parameters of $\pi$.

\begin{definition}
The Langlands parameters associated to an induced representation of Whittaker type $\pi = \bigboxplus_{j = 1}^{r} \pi_j$ of $\GL_n(F)$ are the $n$-tuple of complex numbers $(\alpha_{\pi,1},\ldots,\alpha_{\pi,n})$ given by
\[\alpha_{\pi,\ell} \coloneqq \begin{dcases*}
t_j + \frac{\|\kappa_j\|}{2} & if $\ell = n_1 + \cdots + n_j$ and $\pi_j = e^{i\kappa_j \arg} |\cdot|_{\C}^{t_j}$,	\\
t_j + \kappa_j & if $\ell = n_1 + \cdots + n_j$ and $\pi_j = \sgn^{\kappa_j} |\cdot|_{\R}^{t_j}$,	\\
t_j + \frac{\kappa_j + 1}{2} & if $\ell = n_1 + \cdots + n_j - 1$ and $\pi_j = D_{\kappa_j} \otimes \left|\det\right|_{\R}^{t_j}$,	\\
t_j + \frac{\kappa_j - 1}{2} & if $\ell = n_1 + \cdots + n_j$ and $\pi_j = D_{\kappa_j} \otimes \left|\det\right|_{\R}^{t_j}$.
\end{dcases*}\]
\end{definition}

\begin{proposition}
\label{prop:piurexists}
Given an induced representation of Whittaker type $\pi$ of $\GL_n(F)$, there exists a spherical induced representation of Langlands type $\pi_{\ur}$ of $\GL_n(F)$ for which $L(s,\pi) = L(s,\pi_{\ur})$.
\end{proposition}

\begin{proof}
There exists a permutation $\sigma$ for which the Langlands parameters of $\pi$ are such that $\Re(\alpha_{\pi,\sigma(1)}) \geq \cdots \geq \Re(\alpha_{\pi,\sigma(n)})$. Let $\pi_{\ur} \coloneqq \bigboxplus_{j = 1}^{n} |\cdot|^{\alpha_{\pi,\sigma(j)}}$. This isobaric sum of unramified characters is a spherical induced representation of Langlands type of $\GL_n(F)$ that satisfies
\begin{align*}
L(s,\pi_{\ur}) & = \prod_{j = 1}^{n} \zeta_F\left(s + \alpha_{\pi,\sigma(j)}\right)	\\
& = \prod_{j = 1}^{r} \begin{dcases*}
\zeta_{\C}\left(s + t_j + \frac{\|\kappa_j\|}{2}\right) & if $\pi_j = e^{i\kappa_j \arg} |\cdot|_{\C}^{t_j}$,	\\
\zeta_{\R}(s + t_j + \kappa_j) & if $\pi_j = \sgn^{\kappa_j} |\cdot|_{\R}^{t_j}$,	\\
\zeta_{\R}\left(s + t_j + \frac{\kappa_j - 1}{2}\right) \zeta_{\R}\left(s + t_j + \frac{\kappa_j + 1}{2}\right) & if $\pi_j = D_{\kappa_j} \otimes \left|\det\right|_{\R}^{t_j}$,
\end{dcases*}	\\
& = \prod_{j = 1}^{r} L(s,\pi_j)	\\
& = L(s,\pi).\qedhere
\end{align*}
\end{proof}

\begin{remark}
\hyperref[prop:piurexists]{Proposition \ref*{prop:piurexists}} ensures the existence of such a spherical representation $\pi_{\ur}$ of $\GL_n(F)$, but does not guarantee the \emph{uniqueness}. Indeed, if there are at least two Langlands parameters whose real parts are equal but whose imaginary parts are not, then there exists more than one permutation $\sigma$ for which the condition $\Re(\alpha_{\pi,\sigma(1)}) \geq \cdots \geq \Re(\alpha_{\pi,\sigma(n)})$ is met, and hence more than one spherical induced representation of Langlands type having the same $L$-function as $\pi$. Nonetheless, the spherical Whittaker function $W_{\pi_{\ur}}^{\circ} \in \WW(\pi_{\ur},\psi)$ \emph{is} uniquely determined, since if two spherical induced representations of Langlands type have the same $L$-function, then their spherical Whittaker functions are equal due to the Whittaker--Plancherel theorem (see \cite[Lemma 10.5]{Hum20a} and \hyperref[lem:WtoWur]{Lemma \ref*{lem:WtoWur}}).
\end{remark}

\begin{remark}
A result analogous to \hyperref[prop:piurexists]{Proposition \ref*{prop:piurexists}} also holds in the nonarchimedean setting. A notable difference in this setting is that if $\pi$ is ramified, then $\pi_{\ur}$ is a spherical representation of $\GL_m(F)$ with $m < n$.
\end{remark}

The central character $\omega_{\pi_{\ur}}$ of $\pi_{\ur}$ is closely related to the central character $\omega_{\pi}$ of $\pi$.

\begin{lemma}
\label{lem:omegapitoomegapiur}
Let $\pi$ be an induced representation of Langlands type of $\GL_n(F)$. Then for all $z \in F^{\times}$,
\[\omega_{\pi}(z) \overline{\chi_{\pi}}\left(\frac{z}{\|z\|}\right) \|z\|^{c(\pi)} = \omega_{\pi_{\ur}}(z).\]
\end{lemma}

\begin{proof}
We write $\pi = \bigboxplus_{j = 1}^{r} \pi_j$, so that the central character of $\pi$ is $\omega_{\pi} = \prod_{j = 1}^{r} \omega_{\pi_j}$, where the central character $\omega_{\pi_j}$ of $\pi_j$ is
\[\omega_{\pi_j} = \begin{dcases*}
e^{i\kappa_j \arg} |\cdot|_{\C}^{t_j} & if $\pi_j = e^{i\kappa_j \arg} |\cdot|_{\C}^{t_j}$,	\\
\sgn^{\kappa_j} |\cdot|_{\R}^{t_j} & if $\pi_j = \sgn^{\kappa_j} |\cdot|_{\R}^{t_j}$,	\\
\sgn^{\kappa_j \hspace{-.2cm} \pmod{2}} |\cdot|_{\R}^{t_j} & if $\pi_j = D_{\kappa_j} \otimes \left|\det\right|_{\R}^{t_j}$.
\end{dcases*}\]
Thus for $F = \C$, $\omega_{\pi} = e^{i\kappa_{\pi}\arg} |\cdot|_{\C}^{t_{\pi}}$ where $\kappa_{\pi} \coloneqq \sum_{j = 1}^{n} \kappa_j$ and $t_{\pi} \coloneqq \sum_{j = 1}^{n} t_j$, while for $F = \R$, we have that $\omega_{\pi} = \sgn^{\kappa_{\pi}} |\cdot|_{\R}^{t_{\pi}}$, where $\kappa_{\pi} \in \{0,1\}$ is such that $\kappa_{\pi} \equiv \sum_{j = 1}^{r} \kappa_j \pmod{2}$ and $t_{\pi} \coloneqq \sum_{j = 1}^{r} t_j$. In particular,
\[\chi_{\pi}\left(\frac{z}{\|z\|}\right) = \begin{dcases*}
e^{i\kappa_{\pi} \arg(z)} & if $F = \C$,	\\
\sgn(z)^{\kappa_{\pi}} & if $F = \R$.
\end{dcases*}\]
From \cite[Theorem 4.15 and Section 5.2.1]{Hum20a}, we have that $c(\pi) = \sum_{j = 1}^{r} \|\kappa_j\|$, so that
\[\|z\|^{c(\pi)} = \begin{dcases*}
|z|_{\C}^{(\|\kappa_1\| + \cdots + \|\kappa_r\|)/2} & if $F = \C$,	\\
|z|_{\R}^{\kappa_1 + \cdots + \kappa_r} & if $F = \R$.
\end{dcases*}\]
From this, we see that
\[\omega_{\pi}(z) \overline{\chi_{\pi}}\left(\frac{z}{\|z\|}\right) \|z\|^{c(\pi)} = |z|^{(t_1 + \cdots + t_r) + (\|\kappa_1\| + \cdots + \|\kappa_r\|)/d_F}.\]
On the other hand, the central character of $\pi_{\ur}$ is $\omega_{\pi_{\ur}} = |\cdot|^{t_{\pi_{\ur}}}$ with $t_{\pi_{\ur}} \coloneqq \sum_{j = 1}^{r} (t_j + \|\kappa_j\|/d_F)$. Thus
\[\omega_{\pi_{\ur}}(z) = |z|^{(t_1 + \cdots + t_r) + (\|\kappa_1\| + \cdots + \|\kappa_r\|)/d_F}.\qedhere\]
\end{proof}

We introduce the following by-now well-known lemma. Over nonarchimedean fields, this lemma plays an important role in the proofs of the stability of $\gamma$-factors \cite[Lemma 3.2]{JS85} and of the local converse theorem \cite{Zha18}. Our case is the archimedean analogue (cf.\ \cite[Lemma 10.5]{Hum20a}) of a result of Jacquet, Piateski-Shapiro, and Shalika \cite[Lemme (3.5)]{JP-SS81}.

\begin{lemma}
\label{lem:WtoWur}
Let $\pi$ be an induced representation of Langlands type of $\GL_n(F)$. Then 
\[W_{\pi}^{\circ}\begin{pmatrix} g & 0 \\ 0 & 1 \end{pmatrix} = W_{\pi_{\ur}}^{\circ}\begin{pmatrix} g & 0 \\ 0 & 1 \end{pmatrix}
\quad \text{for all $g \in \GL_{n - 1}(F)$.}
\]
\end{lemma}

\begin{proof}
We claim that for every spherical induced representation of Langlands type $\sigma$ of $\GL_{n - 1}(F)$ with spherical Whittaker function $W_{\sigma}^{\circ} \in \WW(\sigma,\overline{\psi})$,
\[\int\limits_{\Ngp_{n - 1}(F) \backslash \GL_{n - 1}(F)} \left(W_{\pi}^{\circ}\begin{pmatrix} g & 0 \\ 0 & 1 \end{pmatrix} - W_{\pi_{\ur}}^{\circ}\begin{pmatrix} g & 0 \\ 0 & 1 \end{pmatrix}\right) W_{\sigma}^{\circ}(g) \left|\det g\right|^{s - \frac{1}{2}} \, dg = 0\]
for $\Re(s)$ sufficiently large; by the Whittaker--Plancherel theorem (see \cite[Lemma 10.5]{Hum20a}), this implies the desired equality. Indeed, we have by \cite[Theorem 4.17]{Hum20a} that
\begin{align*}
\int\limits_{\Ngp_{n - 1}(F) \backslash \GL_{n - 1}(F)} W_{\pi}^{\circ}\begin{pmatrix} g & 0 \\ 0 & 1 \end{pmatrix} W_{\sigma}^{\circ}(g) \left|\det g\right|^{s - \frac{1}{2}} \, dg & = L(s,\pi \times \sigma),	\\
\int\limits_{\Ngp_{n - 1}(F) \backslash \GL_{n - 1}(F)} W_{\pi_{\ur}}^{\circ}\begin{pmatrix} g & 0 \\ 0 & 1 \end{pmatrix} W_{\sigma}^{\circ}(g) \left|\det g\right|^{s - \frac{1}{2}} \, dg & = L(s,\pi_{\ur} \times \sigma),
\end{align*}
and these are equal as we can write $\sigma = \bigboxplus_{j = 1}^{n - 1} |\cdot|^{t_j}$ for some $t_j \in \C$, and then
\[L(s,\pi \times \sigma) = \prod_{j = 1}^{n - 1} L(s + t_j,\pi) = \prod_{j = 1}^{n - 1} L(s + t_j,\pi_{\ur}) = L(s,\pi_{\ur} \times \sigma).\qedhere\]
\end{proof}

\section{Rankin--Selberg Integrals}
\label{sect:RSsect}

Given induced representations of Whittaker type $\pi$ and $\sigma$ of $\GL_n(F)$, Whittaker functions $W_{\pi} \in \WW(\pi,\psi)$ and $W_{\sigma} \in \WW(\sigma,\overline{\psi})$, and a Schwartz--Bruhat function $\Phi \in \Scr(F^n)$, the $\GL_n \times \GL_n$ Rankin--Selberg integral is defined by
\[\Psi(s,W_{\pi},W_{\sigma},\Phi) \coloneqq \int\limits_{\Ngp_n(F) \backslash \GL_n(F)} W_{\pi}(g) W_{\sigma}(g) \Phi(e_n g) \left|\det g\right|^s \, dg.\]
The local Rankin--Selberg $L$-function $L(s,\pi \times \sigma)$ is defined via the local Langlands correspondence as delineated in \cite{Kna94}. This integral converges absolutely for $\Re(s)$ sufficiently large and extends meromorphically to the entire complex plane. Jacquet \cite[Theorem 2.3]{Jac09} has shown that $\Psi(s,W_{\pi},W_{\sigma},\Phi)$ is a holomorphic multiple of $L(s,\pi \times \sigma)$ and that the quotient 
\[
\frac{\Psi(s,W_{\pi},W_{\sigma},\Phi)}{L(s,\pi \times \sigma)}
\]
is of finite order in vertical strips.

\subsection{Test Vectors for \texorpdfstring{$\GL_n \times \GL_n$}{GL\9040\231\80\327GL\9040\231} Rankin--Selberg Integrals and Periods}

For $x \in \Mat_{n \times m}(F)$, we define the Schwartz--Bruhat function $\Phi_{\ur} \in \Scr(\Mat_{n \times m}(F))$ by
\begin{equation}
\label{eqn:Phiurdefeq}
 \Phi_{\ur}(x) \coloneqq \exp(-d_F \pi \Tr(x \prescript{t}{}{\overline{x}})) = \begin{dcases*}
\exp (-\pi \Tr (x \prescript{t}{}{x}) )& if $F=\R$, \\
\exp(-2\pi \Tr (x \prescript{t}{}{\overline{x}})  & if $F=\C$.
\end{dcases*}
\end{equation}
It is readily seen that $\Phi_{\ur}(xk)=\Phi_{\ur}(k'x)=\Phi_{\ur}(x)$ for all $k \in K_m$ and $k' \in K_n$; that is, $\Phi_{\ur}$ is right $K_m$-invariant and left $K_n$-invariant.

Our goal is to evaluate the $\GL_n \times \GL_n$ Rankin--Selberg integral when $\pi$ and $\sigma$ may be ramified. Our first step is to reduce this integral to an integral over $\Ngp_{n - 1}(F) \backslash \GL_{n - 1}(F)$.

\begin{proposition}
\label{prop:RSkey}
Let $\pi$ and $\sigma$ be induced representations of Langlands type of $\GL_n(F)$ with Whittaker newforms $W_{\pi}^{\circ} \in \WW(\pi,\psi)$ and $W_{\sigma}^{\circ} \in \WW(\sigma,\overline{\psi})$. Let $\Phi \in \Scr(F^n)$ be the right $K_n$-finite Schwartz--Bruhat function of the form $\Phi(x) = P(x) \exp(-d_F \pi x \prescript{t}{}{\overline{x}})$, where the distinguished homogeneous polynomial $P \in \PP_{\overline{\chi_{\pi} \chi_{\sigma}},c(\pi) + c(\sigma)}(F^n)$ is given by
\begin{equation}
\label{eqn:P(x)defeq}
P(x) \coloneqq \sum_{\substack{m = c(\chi_{\pi} \chi_{\sigma}) \\ m \equiv c(\chi_{\pi} \chi_{\sigma}) \hspace{-.2cm} \pmod{2}}}^{c(\pi) + c(\sigma)} (x \prescript{t}{}{\overline{x}})^{\frac{c(\pi) + c(\sigma) - m}{2}} (\dim \tau_{\chi_{\pi} \chi_{\sigma},m}) \overline{P_{\chi_{\pi} \chi_{\sigma},m}^{\circ}}(x).
\end{equation}
Then the $\GL_n \times \GL_n$ Rankin--Selberg integral $\Psi(s,W_{\pi}^{\circ},W_{\sigma}^{\circ},\Phi)$ is equal to
\begin{equation}
\label{eqn:RSperiod}
 L(ns,\omega_{\pi_{\ur}} \omega_{\sigma_{\ur}}) \int\limits_{\Ngp_{n - 1}(F) \backslash \GL_{n - 1}(F)} W_{\pi}^{\circ}\begin{pmatrix} h & 0 \\ 0 & 1 \end{pmatrix} W_{\sigma}^{\circ}\begin{pmatrix} h & 0 \\ 0 & 1 \end{pmatrix} \left|\det h\right|^{s - 1} \, dh.
\end{equation}
\end{proposition}

\begin{proof}
The absolute convergence of the ensuing integral can be justified from \cite[Proposition 3.3 and Lemma 3.5]{Jac09}. By the Iwasawa decomposition $g = (z1_n) \begin{psmallmatrix} h & 0 \\ 0 & 1 \end{psmallmatrix} k$ for $\Ngp_n(F) \backslash \GL_n(F)$, the $\GL_n \times \GL_n$ Rankin--Selberg integral $\Psi(s,W_{\pi}^{\circ},W_{\sigma}^{\circ},\Phi)$ is equal to
\begin{multline*}
\int_{F^{\times}} \omega_{\pi} \omega_{\sigma}(z) |z|^{ns} \int\limits_{\Ngp_{n - 1}(F) \backslash \GL_{n - 1}(F)} \left|\det h\right|^{s - 1}	\\
\times \int_{K_n} W_{\pi}^{\circ}\left(\begin{pmatrix} h & 0 \\ 0 & 1 \end{pmatrix} k\right) W_{\sigma}^{\circ}\left(\begin{pmatrix} h & 0 \\ 0 & 1 \end{pmatrix} k\right) \Phi(ze_n k) \, dk \, dh \, d^{\times}z.
\end{multline*}
We now insert the identities \eqref{eqn:Whitnewformintegral} for both $W_{\pi}^{\circ}(g)$ and $W_{\sigma}^{\circ}(g)$ with $g$ replaced by $\begin{psmallmatrix} h & 0 \\ 0 & 1 \end{psmallmatrix} k$ and the variables of integration being $k_1 \in K_n$ for $W_{\pi}^{\circ}$ and $k_2 \in K_n$ for  $W_{\sigma}^{\circ}$. We then interchange the order of integration and make the change of variables $k \mapsto k^{-1}$, $k_1 \mapsto k^{-1} k_1$, and $k_2 \mapsto k^{-1} k_2$, arriving at
\begin{multline*}
\int_{F^{\times}} \omega_{\pi} \omega_{\sigma}(z) |z|^{ns} \int\limits_{\Ngp_{n - 1}(F) \backslash \GL_{n - 1}(F)} \left|\det h\right|^{s - 1} \int_{K_n} W_{\pi}^{\circ}\left(\begin{pmatrix} h & 0 \\ 0 & 1 \end{pmatrix} k_1\right) \int_{K_n} W_{\sigma}^{\circ}\left(\begin{pmatrix} h & 0 \\ 0 & 1 \end{pmatrix} k_2\right)	\\
\times \dim \tau_{\chi_{\pi},c(\pi)} \dim \tau_{\chi_{\sigma},c(\sigma)} \int_{K_n} P_{\chi_{\pi},c(\pi)}^{\circ}(e_n k_1^{-1} k) P_{\chi_{\sigma},c(\sigma)}^{\circ}(e_n k_2^{-1} k) \Phi(z e_n k^{-1}) \, dk \, dk_2 \, dk_1 \, dh \, d^{\times}z.
\end{multline*}
By the addition theorem, \hyperref[prop:addition]{Proposition \ref*{prop:addition}}, the last line is
\[\sum_{\ell_1 = 1}^{\dim \tau_{\chi_{\pi},c(\pi)}} \sum_{\ell_2 = 1}^{\dim \tau_{\chi_{\sigma},c(\sigma)}} Q_{\ell_1}(e_n k_1^{-1}) Q_{\ell_2}'(e_n k_2^{-1}) \int_{K_n} \overline{Q_{\ell_1}}(e_n k) \overline{Q_{\ell_2}'}(e_n k) \Phi(z e_n k^{-1}) \, dk,\]
where $\{Q_{\ell_1}\}$ and $\{Q_{\ell_2}'\}$ are orthonormal bases of $\HH_{\chi_{\pi},c(\pi)}(F^n)$ and $\HH_{\chi_{\sigma},c(\sigma)}(F^n)$.

To proceed further, we observe that for $z \in F^{\times}$ and $k \in K_n$,we have that
\[\Phi(ze_n k^{-1}) = \overline{\chi_{\pi} \chi_{\sigma}}\left(\frac{z}{\|z\|}\right) \|z\|^{c(\pi) + c(\sigma)} e^{-d_F\pi \|z\|^2} \sum_{\substack{m = c(\chi_{\pi} \chi_{\sigma}) \\ m \equiv c(\chi_{\pi} \chi_{\sigma}) \hspace{-.2cm} \pmod{2}}}^{c(\pi) + c(\sigma)} (\dim \tau_{\chi_{\pi} \chi_{\sigma},m}) P_{\chi_{\pi} \chi_{\sigma},m}^{\circ}(e_n k)\]
by the definition of the Schwartz--Bruhat function $\Phi \in \Scr(F^n)$, the homogeneity of $P_{\chi_{\pi} \chi_{\sigma},m}^{\circ}$ as in \eqref{eqn:homogeneity}, and the identity \eqref{eqn:Pconj}. By \eqref{eqn:Preproducing}, the sum over $m$ is the reproducing kernel for $\PP_{\overline{\chi_{\pi} \chi_{\sigma}},c(\pi) + c(\sigma)}(F^n)$, and so the integral over $K_n \ni k$ is $\overline{Q_{\ell_1}}(e_n) \overline{Q_{\ell_2}'}(e_n)$ by \eqref{eqn:Pmatrixcoeff}. Using the addition theorem, \hyperref[prop:addition]{Proposition \ref*{prop:addition}}, in \emph{reverse} and then using \eqref{eqn:Whitnewformintegral} to evaluate the integrals over $K_n \ni k_1$ and $K_n \ni k_2$, we find that $\Psi(s,W_{\pi}^{\circ},W_{\sigma}^{\circ},\Phi)$ is equal to
\begin{multline*}
\int_{F^{\times}} \omega_{\pi} \omega_{\sigma}(z) \overline{\chi_{\pi} \chi_{\sigma}}\left(\frac{z}{\|z\|}\right) \|z\|^{c(\pi) + c(\sigma)} |z|^{ns} e^{-d_F\pi \|z\|^2} \, d^{\times}z	\\
\times \int\limits_{\Ngp_{n - 1}(F) \backslash \GL_{n - 1}(F)} W_{\pi}^{\circ}\begin{pmatrix} h & 0 \\ 0 & 1 \end{pmatrix} W_{\sigma}^{\circ}\begin{pmatrix} h & 0 \\ 0 & 1 \end{pmatrix} \left|\det h\right|^{s - 1} \, dh.
\end{multline*}
It remains to recall \hyperref[lem:omegapitoomegapiur]{Lemma \ref*{lem:omegapitoomegapiur}}, which, by \eqref{eqn:Lsomega}, shows that the integral over $F^{\times} \ni z$ is $L(ns,\omega_{\pi_{\ur}} \omega_{\sigma_{\ur}})$.
\end{proof}

\begin{remark}
The same proof remains valid in the nonarchimedean setting using the theory of the newform $K_n$-type and $p$-adic spherical harmonics \cite{Hum20b}.
\end{remark}

With this in hand, we are now able to evaluate this Rankin--Selberg integral by reducing to the spherical case.

\begin{theorem}
\label{thm:GLnxGLn}
With the notation and hypotheses of \hyperref[prop:RSkey]{Proposition \ref*{prop:RSkey}}, the $\GL_n \times \GL_n$ Rankin--Selberg integral 
\[
\Psi(s,W_{\pi}^{\circ},W_{\sigma}^{\circ},\Phi) 
\]
is equal to $L(s,\pi_{\ur} \times \sigma_{\ur})$.
\end{theorem}

\begin{proof}
According to \hyperref[lem:WtoWur]{Lemma \ref*{lem:WtoWur}} and \hyperref[prop:RSkey]{Proposition \ref*{prop:RSkey}}, we have the identity
\[\Psi(s,W_{\pi}^{\circ},W_{\sigma}^{\circ},\Phi) = L(ns,\omega_{\pi_{\ur}} \omega_{\sigma_{\ur}}) \int\limits_{\Ngp_{n - 1}(F) \backslash \GL_{n - 1}(F)} W_{\pi_{\ur}}^{\circ}\begin{pmatrix} h & 0 \\ 0 & 1 \end{pmatrix} W_{\sigma_{\ur}}^{\circ}\begin{pmatrix} h & 0 \\ 0 & 1 \end{pmatrix} \left|\det h\right|^{s - 1} \, dh.\]
On the other hand, with $\Phi_{\ur} \in \Scr(F^n)$ given by \eqref{eqn:Phiurdefeq}, the same calculation shows that the right-hand side above is equal to $\Psi(s,W_{\pi_{\ur}}^{\circ},W_{\sigma_{\ur}}^{\circ},\Phi_{\ur})$. Stade's formula \cite[Theorem 1.1]{Sta02} then gives the identity
\[\Psi(s,W_{\pi_{\ur}}^{\circ},W_{\sigma_{\ur}}^{\circ},\Phi_{\ur}) = L(s,\pi_{\ur} \times \sigma_{\ur}).\qedhere\]
\end{proof}

\begin{remark}
Stade's formula is only proved for $F = \R$ in \cite{Sta02}, though from \cite[Proposition 2.1]{Sta95}, the same method can be used to prove this formula for $F = \C$. Different proofs of Stade's formula, valid for $F \in \{\R,\C\}$, were given in more general settings by the first author \cite[Theorem 4.18]{Hum20a} and independently by Ishii and Miyazaki \cite[Theorem 2.9]{IM20}; both proofs are based on the work of Jacquet \cite{Jac09}.
\end{remark}

As previously highlighted, the triple $(W_{\pi}^{\circ}, W_{\sigma}^{\circ}, \Phi)$ is a weak test vector, rather than a strong test vector, for the $\GL_n \times \GL_n$ Rankin--Selberg integral, since in general the na\"{i}ve Rankin--Selberg $L$-function $L(s,\pi_{\ur} \times \sigma_{\ur})$ is not equal to $L(s,\pi \times \sigma)$. Nonetheless, these two $L$-functions are closely related.

\begin{proposition}
\label{prop:RSpolynomial}
Given induced representations of Whittaker type $\pi$ of $\GL_n(F)$ and $\sigma$ of $\GL_m(F)$, there exists a polynomial $p(s)$ for which 
\[
L(s,\pi_{\ur} \times \sigma_{\ur}) = p(s) L(s,\pi \times \sigma).
\]
\end{proposition}

\begin{proof}
For $\pi = \bigboxplus_{j = 1}^{r} \pi_j$ and $\sigma = \bigboxplus_{\ell = 1}^{\rho} \sigma_{\ell}$, with each $\pi_j$ and $\sigma_{\ell}$ essentially square-integrable, we have that
\[L(s,\pi \times \sigma) = \prod_{j = 1}^{r} \prod_{\ell = 1}^{\rho} L(s,\pi_j \times \sigma_{\ell})\]
via the local Langlands correspondence \cite{Kna94}. Thus it suffices to consider the case where $\pi$ and $\sigma$ are both essentially square-integrable.

Suppose first that $F = \C$, so that $\pi = e^{i\kappa \arg} |\cdot|_{\C}^{t}$, and $\sigma = e^{i\lambda \arg} |\cdot|_{\C}^{u}$. Then since $\zeta_{\C}(s) \coloneqq 2(2\pi)^{-s} \Gamma(s)$, and recalling the fact that $\Gamma(s + 1) = s\Gamma(s)$, we have that
\begin{multline*}
\frac{L(s,\pi_{\ur} \times \sigma_{\ur})}{L(s,\pi \times \sigma)} = \frac{\zeta_{\C}\left(s + t + u + \frac{\|\kappa\| + \|\lambda\|}{2}\right)}{\zeta_{\C}\left(s + t + u + \frac{\|\kappa + \lambda\|}{2}\right)}	\\
= \begin{dcases*}
1 & if $\sgn(\kappa) = \sgn(\lambda)$,	\\
(2\pi)^{-\min\{\|\kappa\|,\|\lambda\|\}} \prod_{m = 0}^{\min\{\|\kappa\|,\|\lambda\|\} - 1} \left(s + t + u + \frac{\|\kappa + \lambda\|}{2} + m\right) & if $\sgn(\kappa) \neq \sgn(\lambda)$.
\end{dcases*}
\end{multline*}

Next suppose that $F = \R$ and that $\pi = \sgn^{\kappa} |\cdot|_{\R}^{t}$, and $\sigma = \sgn^{\lambda} |\cdot|_{\R}^{u}$. Then since $\zeta_{\R}(s) \coloneqq \pi^{-s/2} \Gamma(s/2)$, we have that
\[\frac{L(s,\pi_{\ur} \times \sigma_{\ur})}{L(s,\pi \times \sigma)} = \frac{\zeta_{\R}\left(s + t + u + \kappa + \lambda\right)}{\zeta_{\R}\left(s + t + u + \|\kappa - \lambda\|\right)} = \begin{dcases*}
1 & if $(\kappa,\lambda) \neq (1,1)$,	\\
\pi^{-1} \left(\frac{s + t + u}{2}\right) & if $(\kappa,\lambda) = (1,1)$.
\end{dcases*}\]

If $F = \R$, $\pi = D_{\kappa} \otimes \left|\det\right|_{\R}^{t}$, and $\sigma = \sgn^{\lambda} |\cdot|_{\R}^{u}$, then
\begin{align*}
\frac{L(s,\pi_{\ur} \times \sigma_{\ur})}{L(s,\pi \times \sigma)} & = \frac{\zeta_{\R}\left(s + t + u + \frac{\kappa - 1}{2} + \lambda\right)\zeta_{\R}\left(s + t + u + \frac{\kappa + 1}{2} + \lambda\right)}{\zeta_{\R}\left(s + t + u + \frac{\kappa - 1}{2}\right)\zeta_{\R}\left(s + t + u + \frac{\kappa + 1}{2}\right)}	\\
& = \begin{dcases*}
1 & if $\lambda = 0$,	\\
\pi^{-1} \left(\frac{s + t + u + \frac{\kappa - 1}{2}}{2}\right) & if $\lambda = 1$.
\end{dcases*}
\end{align*}
An analogous identity holds if  $\pi = \sgn^{\kappa} |\cdot|_{\R}^{t}$, and $\sigma = D_{\lambda} \otimes \left|\det\right|_{\R}^{u}$.

Finally, if $F = \R$, $\pi = D_{\kappa} \otimes \left|\det\right|_{\R}^{t}$, and $\sigma = D_{\lambda} \otimes \left|\det\right|_{\R}^{u}$, then
\begin{align*}
\frac{L(s,\pi_{\ur} \times \sigma_{\ur})}{L(s,\pi \times \sigma)} & = \frac{\zeta_{\R}\left(s + t + u + \frac{\kappa + \lambda}{2} - 1\right)\zeta_{\R}\left(s + t + u + \frac{\kappa + \lambda}{2}\right)}{\zeta_{\R}\left(s + t + u + \frac{\kappa + \lambda}{2} - 1\right)\zeta_{\R}\left(s + t + u + \frac{\kappa + \lambda}{2}\right)}	\\
& \hspace{2cm} \times \frac{\zeta_{\R}\left(s + t + u + \frac{\kappa + \lambda}{2}\right)\zeta_{\R}\left(s + t + u + \frac{\kappa + \lambda}{2} + 1\right)}{\zeta_{\R}\left(s + t + u + \frac{\|\kappa - \lambda\|}{2}\right)\zeta_{\R}\left(s + t + u + \frac{\|\kappa - \lambda\|}{2} + 1\right)}	\\
& = \pi^{-\min\{\kappa,\lambda\}} \prod_{m = 0}^{\left\lceil\frac{1}{2} \min\{\kappa,\lambda\}\right\rceil - 1} \left(\frac{s + t + u + \frac{\|\kappa - \lambda\|}{2}}{2} + m\right)	\\
& \hspace{2cm} \times \prod_{m = 0}^{\left\lfloor\frac{1}{2} \min\{\kappa,\lambda\}\right\rfloor - 1} \left(\frac{s + t + u + \frac{\|\kappa - \lambda\|}{2} + 1}{2} + m\right).\qedhere
\end{align*}
\end{proof}

When $s = 1$, the integral appearing in \eqref{eqn:RSperiod} is known as the \emph{Rankin--Selberg period}. In general, this integral need not converge at $s = 1$. When $\pi$ and $\sigma$ are \emph{unitary}, however, convergence is guaranteed by the following lemma. We omit the proof of this lemma, since it is standard \cite[Proposition 3.17]{JS81a}; it relies upon bounds for Whittaker functions by a gauge, namely \cite[\S 4 Propositions 2 and 3]{JS90}.

\begin{lemma}
\label{RS-convergence}
Let $\pi$ and $\sigma$ be unitary generic irreducible Casselman--Wallach representations of $\GL_n(F)$. For any $W_{\pi} \in \WW(\pi,\psi)$ and $W_{\sigma} \in \WW(\sigma,\overline{\psi})$, the integral
\[
\int\limits_{\Ngp_{n - 1}(F) \backslash \GL_{n - 1}(F)} W_{\pi}\begin{pmatrix} h & 0 \\ 0 & 1 \end{pmatrix} W_{\sigma} \begin{pmatrix} h & 0 \\ 0 & 1 \end{pmatrix} \left|\det h\right|^{s - 1} \, dh
\]
converges absolutely for $\Re(s) \geq 1$.
\end{lemma}

Archimedean components of cuspidal automorphic representations are unitary generic irreducible Casselman--Wallach representations twisted by a (possibly nonunitary) unramified character. In this regard, our unitary assumption is sufficient for potential global applications therein (cf.\ \hyperref[sec:Global]{Section \ref*{sec:Global}}). For $\sigma=\widetilde{\pi}$, the Rankin--Selberg period (in a slightly modified form as the inner product) appears in the work of Feigon, Lapid, and Offen \cite[Appendix A.1]{FLO}, Gelbart, Jacquet, and Rogawski \cite[Lemma 3.3]{GJR01}, and W.\ Zhang \cite[(3.2)]{Zha14}), which are all based on the pioneering result of Jacquet and Shalika  \cite[Proposition 3.17]{JS81a}. Notably, Venkatesh \cite[\S 7]{Ven06} has evaluated the Rankin--Selberg period for nonarchimedean $F$ when $\sigma = \widetilde{\pi}$ and both Whittaker functions are newforms, so that this is simply the square of the $L^2$-norm of the newform. We prove an archimedean analogue.

\begin{theorem}
\label{thm:RSperiod}
Let $\pi$ and $\sigma$ be unitary generic irreducible Casselman--Wallach representations of $\GL_n(F)$ with Whittaker newforms $W_{\pi}^{\circ} \in \WW(\pi,\psi)$ and $W_{\sigma}^{\circ} \in \WW(\sigma,\overline{\psi})$. Then the $\GL_n \times \GL_n$ Rankin--Selberg period
\begin{equation}
\label{eqn:RSperiod2}
\int\limits_{\Ngp_{n - 1}(F) \backslash \GL_{n - 1}(F)} W_{\pi}^{\circ}\begin{pmatrix} h & 0 \\ 0 & 1 \end{pmatrix} W_{\sigma}^{\circ}\begin{pmatrix} h & 0 \\ 0 & 1 \end{pmatrix} \, dh
\end{equation}
is equal to
\[
\frac{L(1,\pi_{\ur} \times \sigma_{\ur})}{L(n,\omega_{\pi_{\ur}} \omega_{\sigma_{\ur}})}.
\]
\end{theorem}

\begin{proof}
Having \hyperref[RS-convergence]{Lemma \ref*{RS-convergence}} in mind, we take $s=1$ in \hyperref[prop:RSkey]{Proposition \ref*{prop:RSkey}} and invoke \hyperref[thm:GLnxGLn]{Theorem \ref*{thm:GLnxGLn}}.
\end{proof}

We say that $\pi \otimes \sigma$ is \emph{$\GL_n(F)$-distinguished} if $\Hom_{\GL_n(F)}(\pi \otimes \sigma,\mathbf{1})$ is nontrivial. If $\pi$ and $\sigma$ are irreducible, this condition amounts to saying that $\sigma \simeq \widetilde{\pi}$.

\begin{remark}[{\cite[\S 10]{Bar03}}]
\label{rmk:InvarinatRSperiod}
We define a $\Pgp_n(F)$-invariant bilinear form $\beta : \WW(\pi,\psi) \times \WW(\sigma,\overline{\psi}) \to \C$ by
\begin{equation}
\label{eqn:beta}
\beta(W_{\pi},W_{\sigma}) \coloneqq \int\limits_{\Ngp_{n - 1}(F) \backslash \GL_{n - 1}(F)} W_{\pi}\begin{pmatrix} h & 0 \\ 0 & 1 \end{pmatrix} W_{\sigma}\begin{pmatrix} h & 0 \\ 0 & 1 \end{pmatrix} \, dh.
\end{equation}
It can be deduced from \hyperref[thm:RSperiod]{Theorem \ref*{thm:RSperiod}} that $\beta$ is a nontrivial bilinear form. Furthermore, Baruch \cite{Bar03} has shown that any $\Pgp_n(F)$-invariant pairing is $\GL_n(F)$-invariant; the proof is purely local, whereas a local-to-global approach can be extracted from \cite[Proposition 3.1]{Zha14}. Thus $\beta$ gives rise to a nontrivial $\GL_n(F)$-invariant bilinear form on $\WW(\pi,\psi) \times \WW(\sigma,\overline{\psi})$.
\end{remark}

\subsection{Modified Rankin--Selberg Integrals by Sakellaridis}
\label{sec:ModifiedRS}

Sakellaridis \cite[\S 5]{Sak12} has introduced new types of $\GL_n \times \GL_n$ and $\GL_n \times \GL_{n - 1}$ Rankin--Selberg integrals that extend the classical theory of Rankin--Selberg integrals due to Jacquet, Piatetski-Shapiro, and Shalika \cite{JP-SS83}. We solve the weak test vector problems for these integrals. For these modified Rankin--Selberg integrals, we let $G^{\diag}$ denote the image of $G$ under the diagonal embedding $\GL_n \hookrightarrow \GL_n \times \GL_n$ given by $g \mapsto (g,g)$ with $G$ a subgroup of $\GL_n(F)$.

\begin{theorem}
\label{thm:Saknn}
Let $\pi$ and $\sigma$ be induced representations of Langlands type of $\GL_n(F)$ with Whittaker newforms $W^{\circ}_{\pi} \in \WW(\pi,\psi)$ and $W^{\circ}_{\sigma} \in \WW(\sigma,\overline{\psi})$. Let $\Phi \in \Scr(F^n)$ be the right $K_n$-finite Schwartz--Bruhat function of the form $\Phi(x) = P(x) \exp(-d_F \pi x \prescript{t}{}{\overline{x}})$, where the distinguished homogeneous polynomial $P \in \PP_{\overline{\chi_{\pi} \chi_{\sigma}},c(\pi) + c(\sigma)}(F^n)$ is given by \eqref{eqn:P(x)defeq}, and let $\Phi^{\ast} \in \Scr(\Mat_{n \times n}(F))$ be the bi-$K_n$-finite Schwartz--Bruhat function of the form $\Phi^{\ast}(x) = P^{\ast}(e_n x) \exp(-d_F \pi \Tr(x \prescript{t}{}{\overline{x}}))$, where the distinguished homogeneous polynomial $P^{\ast} \in \HH_{\overline{\chi_{\sigma}},c(\sigma)}(F^n)$ is given by
\[P^{\ast}(x) \coloneqq (\dim \tau_{\chi_{\sigma},c(\sigma)}) \overline{P_{\chi_{\sigma},c(\sigma)}^{\circ}}(x).\]
Then for $\Re(s_1)$ and $\Re(s_2)$ sufficiently large, the modified $\GL_n(F) \times \GL_n(F)$ Rankin--Selberg integral by Sakellaridis
\begin{equation}
\label{eqn:Saknn}
\int\limits_{\Ngp^{\diag}_n(F)  \backslash \GL_n(F) \times \GL_n(F)} W^{\circ}_{\pi}(g_1) W^{\circ}_{\sigma}(g_2) \Phi^{\ast}(g^{-1}_1g_2) \Phi(e_ng_1) \left| \frac{\det g_2}{\det g_1} \right|^{s_1} \left|\det g_1\right|^{s_2} \, dg_2 \, dg_1
\end{equation}
is equal to
\[
L\left(s_2, \pi_{\ur} \times \sigma_{\ur} \right) L\left(s_1-\frac{n-1}{2}, \sigma \right).
\]
\end{theorem}

\begin{proof}
We make the change of variables $g_2 \mapsto g_1 g_2$. The ensuing integral becomes
\[
\int\limits_{\Ngp_n(F) \backslash \GL_n(F)}  W^{\circ}_{\pi}(g_1) \Phi(e_ng_1) \left|\det g_1\right|^{s_2} \int_{\GL_n(F)} W^{\circ}_{\sigma}(g_1 g_2) \Phi^{\ast}(g_2) \left|\det g_2\right|^{s_1} \, dg_2 \, dg_1.
\]
The absolute convergence of the above double integral follows from \cite[Lemma 3.2 (ii), Proposition 3.3, and Lemma 3.5]{Jac09}. Applying \cite[Lemma 9.6]{Hum20a} to the innermost integral yields
\[
 L\left(s_1-\frac{n-1}{2}, \sigma\right) \int\limits_{\Ngp_n(F) \backslash \GL_n(F)}  W^{\circ}_{\pi}(g_1)  W^{\circ}_{\sigma}(g_1) \Phi(e_n g_1)  \left|\det g_1\right|^{s_2} \, dg_1.
\]
The latter is the $\GL_n \times \GL_n$ Rankin--Selberg integral, which is $L(s_2,\pi_{\ur} \times \sigma_{\ur})$ by \hyperref[thm:GLnxGLn]{Theorem \ref*{thm:GLnxGLn}}.
\end{proof}

\begin{theorem}
\label{thm:Saknn-1}
Let $\pi$ be an induced representation of Langlands type of $\GL_{n}(F)$ with Whittaker newform $W^{\circ}_{\pi} \in \WW(\pi,\psi)$ and let $\sigma$ be a spherical induced representation of Langlands type of $\GL_{n-1}(F)$ with spherical Whittaker function $W^{\circ}_{\sigma} \in \WW(\sigma,\overline{\psi})$. Then for $\Re(s_1)$ and $\Re(s_2)$ sufficiently large, the modified $\GL_{n}(F) \times \GL_{n-1}(F)$ Rankin--Selberg integral by Sakellaridis
\begin{equation}
\label{eqn:Saknn-1}
\int\limits_{\Ngp^{\diag}_{n-1}(F)  \backslash \GL_{n-1}(F) \times \GL_{n-1}(F)} W^{\circ}_{\pi} \begin{pmatrix} g_1 &0 \\0 & 1 \end{pmatrix} W^{\circ}_{\sigma}(g_2) \Phi_{\ur}\left(  g_1^{-1} g_2 \right) \left| \frac{\det g_2}{\det g_1} \right|^{s_1} \left|\det g_1\right|^{s_2} \, dg_2 \, dg_1
\end{equation}
 is equal to
\[
L\left(s_2+\frac{1}{2}, \pi \times \sigma \right) L\left(s_1-\frac{n-2}{2}, \sigma \right)
\]
where the Schwartz--Bruhat function $\Phi_{\ur} \in  \Scr( \Mat_{(n-1) \times (n-1)}(F))$ is given by \eqref{eqn:Phiurdefeq}.
\end{theorem}

\begin{proof}
We make the change of variables $g_2 \mapsto g_1 g_2$ in order to obtain
\[
\int\limits_{\Ngp_{n-1}(F) \backslash \GL_{n-1}(F)} W^{\circ}_{\pi}\begin{pmatrix} g_1 & 0 \\ 0&1 \end{pmatrix} \left|\det g_1\right|^{s_2} \int_{\GL_{n-1}(F)} W^{\circ}_{\sigma} (g_1 g_2) \Phi_{\ur}(g_2) \left|\det g_2\right|^{s_1} \, dg_2 \, dg_1.
\]
The absolute convergence of the above double integral can be shown by \cite[Lemma 3.2 (ii) and Section 8.3]{Jac09}. Then \cite[Lemma 9.6]{Hum20a} implies that this is equal to
\[
L\left(s_1-\frac{n-2}{2}, \sigma \right) \int\limits_{\Ngp_{n-1}(F) \backslash \GL_{n-1}(F)} W^{\circ}_{\pi}  \begin{pmatrix} g_1 & 0 \\ 0 &1 \end{pmatrix}  W^{\circ}_{\sigma}(g_1) \left|\det g_1\right|^{s_2} \, dg_1.
\]
From \cite[Theorem 4.17]{Hum20a}, the latter integral is simply $L(s_2+1/2, \pi \times \sigma)$.
\end{proof}

We take this occasion to complete weak test vector problems for modified Rankin--Selberg integrals over a nonarchimedean local field $F$. The missing ingredient is a convolution section identity for the Whittaker newform, namely the nonarchimedean analogue of \cite[Lemma 9.6]{Hum20a}. The identity can be viewed as a generalization of an identity of Godement--Jacquet \cite[Lemma 6.10]{GJ72} (cf.\ \cite[(5-2)]{Sak12}).

\begin{lemma}
\label{lem:GJConvolution}
Let $F$ be a nonarchimedean local field and let $\pi$ be an induced representation of Langlands type of $\GL_n(F)$ with Whittaker newform $W^{\circ}_{\pi} \in \WW(\pi,\psi)$. Then for all $h \in \GL_n(F)$ and for $\Re(s)$ sufficiently large, we have that
\[
\int_{\GL_n(F)} W^{\circ}_{\pi}(hg)\Phi^{\ast}(g) \left|\det g\right|^{s+\frac{n-1}{2}} \, dg = L(s,\pi)  W^{\circ}_{\pi}(h),
\]
where $\Phi^{\ast} \in \Scr(\Mat_{n \times n}(F))$ is the bi-$K_n$-finite Schwartz--Bruhat function
\begin{multline}
\label{GJ-SchwartzBruhat}
\Phi^{\ast}(x) \\
\coloneqq  \begin{dcases*}
1 & if $c(\pi) = 0$, $x \in \Mat_{n \times n}(\OO)$, and $x_{n,1},\ldots,x_{n,n} \in \OO$,	\\
\frac{\omega_{\pi}^{-1}(x_{n,n})}{\vol(K_0(\pp^{c(\pi)}))} & if $c(\pi) > 0$, $x \in \Mat_{n \times n}(\OO)$, $x_{n,1},\ldots,x_{n,n - 1} \in \pp^{c(\pi)}$, and $x_{n,n} \in \OO^{\times}$,	\\
0 & otherwise.
\end{dcases*}
\end{multline}
\end{lemma}

Here $\OO$ denotes the ring of integers of $F$, $\pp$ denotes the maximal ideal of $\OO$, $c(\pi)$ denotes the conductor exponent of $\pi$, and for $m > 0$, $K_0(\pp^m)$ denotes the congruence subgroup
\[K_0(\pp^m) \coloneqq \{k \in \GL_n(\OO) : k_{n,1},\ldots,k_{n,n - 1} \in \pp^m\},\]
which has volume $q^{-(m - 1)(n - 1)} (q - 1) / (q^n - 1)$, where $q \coloneqq \# \OO / \pp$.

\begin{proof}
Let $\sigma = \bigboxplus_{j = 1}^{n} |\cdot|^{t_j}$ and $\sigma_0 \coloneqq \bigboxplus_{j = 2}^{n} |\cdot|^{t_j}$ be induced representations of Langlands type of $\GL_n(F)$ and $\GL_{n - 1}(F)$ with spherical Whittaker functions $W_{\sigma}^{\circ} \in \WW(\sigma,\overline{\psi})$ and $W_{\sigma_0}^{\circ} \in \WW(\sigma_0,\overline{\psi})$. By \cite[Theorem 2.1.1]{Kim10}, the  $\GL_n \times \GL_n$ Rankin--Selberg integral
\[\Psi(s,W_{\pi}^{\circ},W_{\sigma}^{\circ},\Phi^{\circ}) \coloneqq \int\limits_{\Ngp_n(F) \backslash \GL_n(F)} W_{\pi}^{\circ}(g) W_{\sigma}^{\circ}(g) \Phi^{\circ}(e_n g) \left|\det g\right|^s \, dg\]
is equal to
\[L(s,\pi \times \sigma) = L(s + t_1,\pi) L(s,\pi \times \sigma_0),\]
where $\Phi^{\circ} \in \Scr(F^n)$ is the Schwartz--Bruhat function
\begin{equation}
\label{RS-SchwartzBruhat}
\Phi^{\circ}(x_1,\ldots,x_n) \coloneqq \begin{dcases*}
1 & if $c(\pi) = 0$ and $x_1,\ldots,x_n \in \OO$,	\\
\frac{\omega_{\pi}^{-1}(x_n)}{\vol(K_0(\pp^{c(\pi)}))} & if $c(\pi) > 0$, $x_1,\ldots,x_{n - 1} \in \pp^{c(\pi)}$, and $x_n \in \OO^{\times}$,	\\
0 & otherwise.
\end{dcases*}
\end{equation}
On the other hand, we may insert into the $\GL_n \times \GL_n$ Rankin--Selberg integral the propagation formula \cite[Lemma 4.1]{Hum21}
\begin{multline*}
W_{\sigma}^{\circ}(g) = \left|\det g\right|^{t_1 + \frac{n - 1}{2}} \int_{\GL_{n - 1}(F)} W_{\sigma_0}^{\circ}(h) \left|\det h\right|^{-t_1 - \frac{n}{2}}	\\
\times \int\limits_{\Mat_{(n - 1) \times 1}(F)} \Phi'\left(h^{-1} \begin{pmatrix} 1_{n - 1} & v \end{pmatrix} g\right) \psi(e_{n - 1} v) \, dv \, dh,
\end{multline*}
where $\Phi' \in \Scr(\Mat_{(n - 1) \times n}(F))$ is the Schwartz--Bruhat function
\[\Phi'(x) \coloneqq \begin{dcases*}
1 & if $x \in \Mat_{(n - 1) \times n}(\OO)$,	\\
0 & otherwise,
\end{dcases*}\]
unfold the integration, and make the change of variables $g \mapsto \begin{psmallmatrix} h & 0 \\ 0 & 1 \end{psmallmatrix} g$ in order to see that
\begin{multline*}
\Psi(s,W_{\pi}^{\circ},W_{\sigma}^{\circ},\Phi^{\circ})	\\
= \int\limits_{\Ngp_{n - 1}(F) \backslash \GL_{n - 1}(F)} W_{\sigma_0}^{\circ}(h) \left|\det h\right|^{s - \frac{1}{2}} \int_{\GL_n(F)} W_{\pi}^{\circ}\left(\begin{pmatrix} h & 0 \\ 0 & 1 \end{pmatrix} g\right) \Phi^{\ast}(g) \left|\det g\right|^{s + t_1 + \frac{n - 1}{2}} \, dg \, dh
\end{multline*}
since $\Phi^{\ast}(g) = \Phi'(\begin{pmatrix} 1_{n - 1} & 0 \end{pmatrix} g) \Phi^{\circ}(e_n g)$. So letting $W_{\pi}' \in \WW(\pi,\psi)$ be given by
\[W_{\pi}'(h) \coloneqq \frac{1}{L(w,\pi)} \int_{\GL_n(F)} W_{\pi}^{\circ}(hg) \Phi^{\ast}(g) \left|\det g\right|^{w + \frac{n - 1}{2}} \, dg\]
with $\Re(w)$ sufficiently large, we see that for every induced representation of Langlands type $\sigma_0$ of $\GL_{n - 1}(F)$ with spherical Whittaker function $W_{\sigma_0}^{\circ} \in \WW(\sigma_0,\overline{\psi})$,
\[\int\limits_{\Ngp_{n - 1}(F) \backslash GL_{n - 1}(F)} \left(W_{\pi}'\begin{pmatrix} h & 0 \\ 0 & 1 \end{pmatrix} - W_{\pi}^{\circ}\begin{pmatrix} h & 0 \\ 0 & 1 \end{pmatrix}\right) W_{\sigma_0}^{\circ}(h) \left|\det h\right|^{s - \frac{1}{2}} \, dh = 0\]
for $\Re(s)$ sufficiently large due to \cite[Th\'{e}or\`{e}me (4)]{JP-SS81}. Since $W_{\pi}'$ is right $\GL_{n - 1}(\OO)$-invariant, as we may make the change of variables $g \mapsto \begin{psmallmatrix} k'^{-1} & 0 \\ 0 & 1 \end{psmallmatrix} g$ and use the fact that $\Phi$ is left $\GL_{n - 1}(\OO)$-invariant, we therefore have that $W_{\pi}'\begin{psmallmatrix} h & 0 \\ 0 & 1 \end{psmallmatrix} = W_{\pi}^{\circ}\begin{psmallmatrix} h & 0 \\ 0 & 1 \end{psmallmatrix}$ for all $h \in \GL_{n - 1}(F)$ by \cite[Lemme (3.5)]{JP-SS81}. Invoking the uniqueness of the Kirillov model of $\pi$, we deduce that $W_{\pi}'(g) = W_{\pi}^{\circ}(g)$ for all $g \in \GL_n(F)$.
\end{proof}

We present following consequences of the convolution section identity in \hyperref[lem:GJConvolution]{Lemma \ref*{lem:GJConvolution}}, which follow in the exact same manner as in the archimedean setting.

\begin{corollary} 
\label{cor:Saknn-nonarch}
Let $F$ be a nonarchimedean local field, and let $\pi$ and $\sigma$ be induced representations of Langlands type of $\GL_n(F)$ with Whittaker newforms $W_{\pi}^{\circ} \in \WW(\pi,\psi)$ and $W_{\sigma}^{\circ} \in \WW(\sigma,\overline{\psi})$. Let $\Phi^{\ast} \in \Scr(\Mat_{n \times n}(F))$ and $\Phi^{\circ} \in \Scr(F^n)$ be the Schwartz--Bruhat functions given by \eqref{GJ-SchwartzBruhat} and \eqref{RS-SchwartzBruhat}. Then for $\Re(s_1)$ and $\Re(s_2)$ sufficiently large, the modified $\GL_n(F) \times \GL_n(F)$ Rankin--Selberg integral by Sakellaridis
\[
\int\limits_{\Ngp^{\diag}_n(F)  \backslash \GL_n(F) \times \GL_n(F)} W^{\circ}_{\pi}(g_1) W^{\circ}_{\sigma}(g_2) \Phi^{\ast}(g^{-1}_1g_2) \Phi^{\circ}(e_ng_1) \left|\frac{\det g_2}{\det g_1} \right|^{s_1} \left|\det g_1\right|^{s_2} \, dg_2 \, dg_1
\]
is equal to
\[
L\left(s_2, \pi_{\ur} \times \sigma_{\ur} \right) L\left(s_1-\frac{n-1}{2}, \sigma \right).
\]
\end{corollary}

\begin{corollary}
\label{cor:Saknn-1-nonarch}
Let $F$ be a nonarchimedean local field, let $\pi$ be an induced representation of Langlands type of $\GL_{n}(F)$ with Whittaker newform $W^{\circ}_{\pi} \in \WW(\pi,\psi)$, and let $\sigma$ be a spherical induced representation of Langlands type of $\GL_{n - 1}(F)$ with spherical Whittaker function  $W^{\circ}_{\sigma} \in \WW(\sigma,\overline{\psi})$. Then for $\Re(s_1)$ and $\Re(s_2)$ sufficiently large, the modified $\GL_{n}(F) \times \GL_{n-1}(F)$ Rankin--Selberg integral by Sakellaridis
\[
\int\limits_{\Ngp^{\diag}_{n-1}(F)  \backslash \GL_{n-1}(F) \times \GL_{n-1}(F)} W^{\circ}_{\pi} \begin{pmatrix} g_1 &0 \\0 & 1 \end{pmatrix} W^{\circ}_{\sigma}(g_2) \Phi_{\ur}\left(  g_1^{-1} g_2 \right) \left| \frac{\det g_2}{\det g_1} \right|^{s_1} \left|\det g_1\right|^{s_2} \, dg_2 \, dg_1
\]
is equal to
\[
L\left(s_2+\frac{1}{2}, \pi \times \sigma\right) L\left(s_1-\frac{n-2}{2}, \sigma \right)
\]
where the Schwartz--Bruhat function $\Phi_{\ur} \in  \Scr( \Mat_{(n-1) \times (n-1)}(F))$ is given by
\[\Phi_{\ur}(x) \coloneqq \begin{dcases*}
1 & if $x \in \Mat_{(n - 1) \times (n - 1)}(\OO)$,	\\
0 & otherwise.
\end{dcases*}\]
\end{corollary}

For a pair of spherical induced representations of Langlands type over nonarchimedean local fields, such formul\ae{} are due to Sakellaridis \cite[\S 5]{Sak12}.

\section{Flicker Integrals}
\label{sect:Flickersect}

We define an additive character on $\C$ by $\psi_{\C/\R}(z) \coloneqq e^{-2\pi(z - \overline{z})}$, so that $\psi_{\C/\R}(x + iy) = e^{-4\pi iy}$ for $x,y \in \R$; this additive character is trivial when restricted to $\R$. We let
\[\psi_{\C/\R,n}(u) \coloneqq \psi_{\C/\R}\left(\sum_{j = 1}^{n - 1} u_{j,j + 1}\right)\]
denote the corresponding character of $\Ngp_n(\C) \ni u$. Given an induced representation of Whittaker type $\pi$ of $\GL_n(\C)$, a Whittaker function $W_{\pi} \in \WW(\pi,\psi_{\C/\R})$, and a Schwartz--Bruhat function $\Phi \in \Scr(\R^n)$, the $\GL_n$ Flicker integral \cites{Fli93,FZ95} is defined by
\[\Psi(s,W_{\pi},\Phi) \coloneqq \int\limits_{\Ngp_n(\R) \backslash \GL_n(\R)} W_{\pi}(g) \Phi(e_n g) \left|\det g\right|_{\R}^s \, dg.\]
Once more, this integral converges absolutely for $\Re(s)$ sufficiently and extends meromorphically to the entire complex plane. The local Asai $L$-function $L(s,\pi,\As)$ is defined via the local Langlands correspondence as described accurately in \cite[Section 3.2]{B-P21}. Beuzart--Plessis \cite[Theorem 3.5]{B-P21} has shown that $\Psi(s,W_{\pi},\Phi)$ is a holomorphic multiple of $L(s,\pi,\As)$ and that the quotient
\[
\frac{\Psi(s,W_{\pi},\Phi)}{ L(s,\pi,\As)}
\]
 is of finite order in vertical strips.

\subsection{The Spherical Calculation}

We recall the Iwasawa decomposition
\[\GL_n(\C) = \Ngp_n(\C)\Agp_n(\C)\Ugp(n).\]
Since every element of $\Agp_n(\C)$ may be written as the product of an element of $\Agp_n(\R)$ and of $\Agp_n(\C) \cap \Ugp(n)$, we also have the Iwasawa decomposition
\[\GL_n(\C) = \Ngp_n(\C)\Agp_n(\R)\Ugp(n).\]
We write $g = uak$, where $g \in \GL_n(\C)$, $u \in \Ngp_{n}(\C)$, $a = \diag(a_1,\ldots,a_{n}) \in \Agp_{n}(\R)$, and $k \in \Ugp(n)$. The Haar measure on $\GL_n(\C)$ becomes
\[
dg = 2^{n} \delta_{\Bgp_n(\R)}^{-2}(a) \, dk \, d^{\times}a \, du.
\]
Here the additional factor of $2^n$ arises from the fact that the diagonal torus is chosen to be $\Agp_n(\R)$ in place of $\Agp_n(\C)$, recalling that the Haar measure on $\C$ is twice the Lebesgue measure, while for $a \in \Agp_n(\R)$, we observe that $\delta_{\Bgp_n(\R)}^{-2}(a) = \delta_{\Bgp_n(\C)}^{-1}(a)$.

We first require the following propagation formula.

\begin{lemma}
\label{propformula}
Let $\pi = \bigboxplus_{j = 1}^{n} |\cdot|_{\C}^{t_j}$ and $\pi_0 \coloneqq \bigboxplus_{j = 2}^{n} |\cdot|_{\C}^{t_j}$ be spherical induced representations of Langlands type of $\GL_n(\C)$ and $\GL_{n - 1}(\C)$ respectively with spherical Whittaker functions $W_{\pi}^{\circ} \in \WW(\pi,\psi_{\C/\R})$ and $W_{\pi_0}^{\circ} \in \WW(\pi_0,\psi_{\C/\R})$. Then for $a = \diag(a_1,\ldots,a_n) \in \Agp_n(\R)$, we have that
\begin{multline}
\label{eqn:propformula}
W_{\pi}^{\circ}(a) = 2^{n - 1} \left|\det a\right|_{\R}^{2t_1} \delta_{\Bgp_n(\R)}(a) \int_{\Agp_{n - 1}(\R)} W_{\pi_0}^{\circ}(a') \prod_{j = 1}^{n - 1} \exp\left(-2\pi a_j^{\prime 2} a_{j + 1}^{-2}\right) 	\\
\times \exp\left(-2\pi a_j^{\prime -2} a_j^2\right) \left|\det a'\right|_{\R}^{-2t_1} \delta_{\Bgp_{n-1}(\R)}^{-1}(a') \, d^{\times}a'.
\end{multline}
\end{lemma}

\begin{proof}
From \cite[Lemma 9.14]{Hum20a}, we have the identity
\begin{multline*}
W_{\pi_0}^{\circ}\begin{pmatrix} g & 0 \\ 0 & a_n \end{pmatrix} = \left|\det g\right|_{\C}^{t_1 + \frac{n - 1}{2}} |a_n|_{\C}^{t_1 - \frac{n - 1}{2}}	\\
\times \int_{\GL_{n - 1}(\C)} W_{\pi_0}^{\circ}(h) \Phi_1(h^{-1}g) \Phi_2(a_n^{-1} e_{n - 1} h) \left|\det h\right|_{\C}^{-t_1 - \frac{n}{2} + 1} \, dh
\end{multline*}
for $g \in \GL_n(\C)$ and $a_n \in \C^{\times}$, where $\Phi_1 \in \Scr(\Mat_{(n - 1) \times (n - 1)}(\C))$ and $\Phi_2 \in \Scr(\Mat_{1 \times (n - 1)}(\C))$ are given by
\[\Phi_1(x_1) \coloneqq \exp(-2\pi \Tr(x_1 \prescript{t}{}{\overline{x_1}})), \qquad \Phi_2(x_2) \coloneqq \exp(-2\pi x_2 \prescript{t}{}{\overline{x_2}}).\]

We employ the Iwasawa decomposition in order to write $h = u'a'k'$, where $u' \in \Ngp_{n - 1}(\C)$, $a' = \diag(a_1',\ldots,a_{n - 1}') \in \Agp_{n - 1}(\R)$, and $k' \in \Ugp(n - 1)$. As $W_{\pi_0}^{\circ}(u'a'k') = \psi_{\C/\R,n - 1}(u') W_{\pi_0}^{\circ}(a')$, the integral over $\Ugp(n - 1) \ni k'$ is trivial.

Now we specify $g = \diag(a_1,\ldots,a_{n - 1}) \in \Agp_{n - 1}(\R)$ and $a_n \in \R$, so that $\Phi_2(a_n^{-1} e_{n - 1} h) = \exp(-2\pi a_{n - 1}^{\prime 2} a_n^{-2})$. We make the change of variables $u' \mapsto u^{\prime -1}$, then $u_{i,j}' \mapsto a_i' a_j^{-1} u_{i,j}'$, and finally evaluate the integrals over $\C \ni u_{i,j}'$; they are equal to $1$ if $j \neq i + 1$ and to $\exp(-2\pi a_i^{\prime 2} a_{i + 1}^{-2})$ if $j = i + 1$. This gives the desired identity.
\end{proof}

We also require a convolution section identity.

\begin{lemma}
\label{convsectid}
Let $\pi = \bigboxplus_{j = 1}^{n} |\cdot|_{\C}^{t_j}$ be a spherical induced representation of Langlands type of $\GL_n(\C)$ with spherical Whittaker functions $W_{\pi}^{\circ} \in \WW(\pi,\psi_{\C/\R})$. Then for $a' = \diag(a_1',\ldots,a_n') \in \Agp_n(\R)$, we have that for $\Re(s)$ sufficiently large,
\begin{multline}
\label{eqn:convsectid}
L(s,\pi) W_{\pi}^{\circ}(a') \\
= 2^n \int_{\Agp_n(\R)} W_{\pi}^{\circ}(a'a) \prod_{j = 1}^{n} \exp(-2\pi a_j^2) \prod_{j = 1}^{n - 1} \exp\left(-2\pi a_j^{\prime 2} a_{j + 1}^{\prime -2} a_{j + 1}^{-2}\right) \left|\det a\right|_{\R}^{2s} \delta_{\Bgp_n(\R)}^{-1}(a) \, d^{\times}a.
\end{multline}
\end{lemma}

\begin{proof}
From \cite[Lemma 9.6]{Hum20a}, we have the identity
\[L(s,\pi) W_{\pi}^{\circ}(a') = \int_{\GL_{n - 1}(\C)} W_{\pi}^{\circ}(a' h) \Phi_{\ur}(h) \left|\det h\right|_{\C}^{s + \frac{n - 1}{2}} \, dh,\]
with $\Phi_{\ur} \in \Scr(\Mat_{n \times n}(\C))$ given by \eqref{eqn:Phiurdefeq}. We employ the Iwasawa decomposition in order to write $h = uak$, where $u \in \Ngp_n(\C)$, $a = \diag(a_1,\ldots,a_n) \in \Agp_n(\R)$, and $k \in \Ugp(n)$. As $W_{\pi}^{\circ}(a'uak) = \psi_{\C/\R,n - 1}(a'ua^{\prime -1}) W_{\pi}^{\circ}(a'a)$, the integral over $\Ugp(n) \ni k$ is trivial. We make the change of variables $u_{i,j} \mapsto a_j^{-1} u_{i,j}$, and finally evaluate the integrals over $\C \ni u_{i,j}$; they are equal to $1$ if $j \neq i + 1$ and to $\exp(-2\pi a_i^{\prime 2} a_{i + 1}^{\prime -2} a_{i + 1}^{-2})$ if $j = i + 1$. We arrive at the identity \eqref{eqn:convsectid}.
\end{proof}

We are now able to prove that when $\pi$ is a spherical induced representation of Langlands type of $\GL_n(\C)$, the spherical Whittaker function $W_{\pi}^{\circ} \in \WW(\pi,\psi_{\C/\R})$ is a strong test vector for the Flicker integral.

\begin{theorem}
\label{thm:Flickerspher}
Let $\pi = \bigboxplus_{j = 1}^{n} |\cdot|_{\C}^{t_j}$ be a spherical induced representation of Langlands type of $\GL_n(\C)$ with spherical Whittaker function $W_{\pi}^{\circ} \in \WW(\pi,\psi_{\C/\R})$. Then with $\Phi_{\ur} \in \Scr(\R^n)$ given by \eqref{eqn:Phiurdefeq}, the Flicker integral $\Psi(s,W_{\pi}^{\circ},\Phi_{\ur})$ is equal to
\[
L(s,\pi,\As) \coloneqq \prod_{j = 1}^{n} \zeta_{\R}(s + 2t_j) \prod_{1 \leq j < \ell \leq n} \zeta_{\C}(s + t_j + t_{\ell}).
\]
\end{theorem}

\begin{proof}
We prove this by induction. The base case $n = 1$ is trivially true since the $\GL_1$ Flicker integral is simply
\[\int_{\R^{\times}} \exp(-\pi x^2)  |x|_{\R}^{s + 2t_1}\, d^{\times}x = \zeta_{\R}(s + 2t_1) = L(s,\pi,\As).\]
For the general case, we shall show that if $\pi = \bigboxplus_{j = 1}^{n} |\cdot|_{\C}^{t_j}$ and $\pi_0 \coloneqq \bigboxplus_{j = 2}^{n} |\cdot|_{\C}^{t_j}$, then the $\GL_n$ Flicker integral satisfies the identity
\begin{equation}
\label{eqn:inductionhyp}
\Psi(s,W_{\pi}^{\circ},\Phi_{\ur}) = \zeta_{\R}(s + 2t_1) L(s + t_1,\pi_0) I(s,W_{\pi_0}^{\circ},\Phi_{\ur}),
\end{equation}
where the $\GL_{n - 1}$ Flicker integral on the right-hand side involves the spherical Whittaker function $W_{\pi_0}^{\circ} \in \WW(\pi_0,\psi_{\C/\R})$ and the Schwartz--Bruhat function $\Phi_{\ur} \in \Scr(\R^{n - 1})$ given by \eqref{eqn:Phiurdefeq}. Since
\[L(s + t_1,\pi_0) = \prod_{\ell = 2}^{n} \zeta_{\C}(s + t_1 + t_{\ell}),\]
this implies the result by the induction hypothesis.

We first note that the $\Psi(s,W_{\pi}^{\circ},\Phi_{\ur})$ is equal to
\[\int_{\Agp_n(\R)} W_{\pi}^{\circ}(a) \exp(-\pi a_n^2) \left|\det a\right|_{\R}^s \delta_{\Bgp_n(\R)}^{-1}(a) \, d^{\times}a\]
via the Iwasawa decomposition $g = ak$, since the integral over $\Ogp(n) \ni k$ is trivial. We insert the propagation formula \eqref{eqn:propformula} and relabel $a = \diag(a_1,\ldots,a_n) \in \Agp_n(\R)$ as $\begin{psmallmatrix} y & 0 \\ 0 & a \end{psmallmatrix}$, where now $y \in \R^{\times}$ and $a = \diag(a_1,\ldots,a_{n - 1}) \in \Agp_{n - 1}(\R)$. The Flicker integral $\Psi(s,W_{\pi}^{\circ},\Phi_{\ur})$ becomes
\begin{multline*}
2^{n - 1} \int_{\Agp_{n - 1}(\R)} \exp(-\pi a_{n - 1}^2) \left|\det a\right|_{\R}^{s + 2t_1} \int_{\Agp_{n - 1}(\R)} W_{\pi_0}^{\circ}(a') \prod_{j = 1}^{n - 2} \exp\left(-2\pi a_{j + 1}^{\prime -2} a_j^2\right)	\\
\times \prod_{j = 1}^{n - 1} \exp\left(-2\pi a_j^{\prime 2} a_j^{-2}\right) \left|\det a'\right|_{\R}^{-2t_1} \delta_{\Bgp_{n-1}(\R)}^{-1}(a') \int_{\R^{\times}} \exp(-2\pi a_1^{\prime - 2} y^2) |y|_{\R}^{s + 2t_1} \, d^{\times}y \, d^{\times}a' \, d^{\times}a.
\end{multline*}
We make the change of variables $y \mapsto 2^{-1/2} a_1' y$, then the change of variables $a_j' \mapsto 2^{-1/2} a_j' a_j$ and $a_j \mapsto 2^{1/2} a_j^{\prime -1} a_{j + 1}' a_j$, and finally $a_{n - 1} \mapsto 2^{1/2} a_{n - 1}$. We arrive at the identity
\begin{multline*}
\int_{\R^{\times}} \exp(-\pi y^2) |y|_{\R}^{s + 2t_1} \, d^{\times}y \int_{\Agp_{n - 1}(\R)} \exp(-\pi a_{n - 1}^{\prime 2}) \left|\det a'\right|_{\R}^s \delta_{\Bgp_{n-1}(\R)}^{-1}(a') \\
\times 2^{n - 1} \int_{\Agp_{n - 1}(\R)} W_{\pi_0}^{\circ}(a'a) \prod_{j = 1}^{n - 1} \exp(-2\pi a_j^2)	  \prod_{j = 1}^{n - 2} \exp\left(-2\pi a_j^{\prime 2} a_{j + 1}^{\prime -2} a_{j + 1}^{-2}\right)	\\
\times \left|\det a\right|_{\R}^{2s + 2t_1} \delta_{\Bgp_{n-1}(\R)}^{-1}(a) \, d^{\times}a \, d^{\times}a'.
\end{multline*}
The integral over $\R^{\times} \ni y$ is $\zeta_{\R}(s + 2t_1)$, while the last two lines are $L(s + t_1,\pi_0) W_{\pi_0}^{\circ}(a')$ by the convolution section identity \eqref{eqn:convsectid}. So this is
\[\zeta_{\R}(s + 2t_1) L(s + t_1,\pi_0) \int_{\Agp_{n - 1}(\R)} W_{\pi_0}^{\circ}(a') \exp(-\pi a_{n - 1}^{\prime 2}) \left|\det a'\right|_{\R}^s \delta_{\Bgp_{n-1}(\R)}^{-1}(a') \, d^{\times}a',\]
which is precisely \eqref{eqn:inductionhyp} by the Iwasawa decomposition $g = a'k'$, since the integral over $\Ogp(n - 1) \ni k'$ is trivial.
\end{proof}

\subsection{Test Vectors for Flicker Integrals and Flicker--Rallis Periods}

We proceed to the more general case where $\pi$ may be ramified. Our first step is to reduce the Flicker integral to an integral over $\Ngp_{n - 1}(\R) \backslash \GL_{n - 1}(\R)$.

\begin{proposition}
\label{prop:FlickerKey}
Let $\pi$ be an induced representation of Langlands type of $\GL_n(\C)$ with Whittaker newform $W_{\pi}^{\circ} \in \WW(\pi,\psi_{\C/\R})$. Let $\Phi \in \Scr(\R^n)$ be the right $\Ogp(n)$-finite Schwartz--Bruhat function of the form $\Phi(x) = P(x) \exp(-\pi x \prescript{t}{}{x})$, where the distinguished homogeneous polynomial $P \in \PP_{\chi_{\pi}|_{\Ogp(1)},c(\pi)}(\R^n)$ is given by
\[P(x) \coloneqq \sum_{\substack{m = c(\chi_{\pi}|_{\Ogp(1)}) \\ m \equiv c(\chi_{\pi}|_{\Ogp(1)}) \hspace{-.2cm} \pmod{2}}}^{c(\pi)} \left(x \prescript{t}{}{x}\right)^{\frac{c(\pi) - m}{2}} (\dim \tau_{\chi_{\pi}|_{\Ogp(1)},m}) \overline{P_{\chi_{\pi}|_{\Ogp(1)},m}^{\circ}}(x).\]
Then the Flicker integral $\Psi(s,W_{\pi}^{\circ},\Phi)$ is equal to
\begin{equation}
\label{eqn:FRperiod}
L\left(ns, \omega_{\pi_{\ur}}|_{\R^{\times}}\right) \int\limits_{\Ngp_{n - 1}(\R) \backslash \GL_{n - 1}(\R)} W_{\pi}^{\circ}\begin{pmatrix} g & 0 \\ 0 & 1 \end{pmatrix} \left|\det g\right|_{\R}^{s - 1} \, dg.
\end{equation}
\end{proposition}

The proof of this proceeds along similar lines to that of \hyperref[prop:RSkey]{Proposition \ref*{prop:RSkey}}.

\begin{proof}
We use the Iwasawa decomposition $g = (z1_n) \begin{psmallmatrix} h & 0 \\ 0 & 1 \end{psmallmatrix} k$ for $\Ngp_n(\R) \backslash \GL_n(\R)$ to see that the Flicker integral $\Psi(s,W_{\pi}^{\circ},\Phi)$ is equal to
\[\int_{\R^{\times}} \omega_{\pi}(z) |z|_{\R}^{ns} \int\limits_{\Ngp_{n - 1}(\R) \backslash \GL_{n - 1}(\R)} \left|\det h\right|_{\R}^{s - 1} \int_{\Ogp(n)} W_{\pi}^{\circ}\left(\begin{pmatrix} h & 0 \\ 0 & 1 \end{pmatrix} k\right) \Phi(ze_n k) \, dk \, dh \, d^{\times}z.\]
We insert the identity \eqref{eqn:Whitnewformintegral} for $W_{\pi}^{\circ}(g)$ with $g$ replaced by $\begin{psmallmatrix} h & 0 \\ 0 & 1 \end{psmallmatrix} k$ and the variable of integration being $k' \in \Ugp(n)$, then interchange the order of integration and make the change of variables $k \mapsto k^{-1}$ and $k' \mapsto k k'$. We end up with
\begin{multline*}
\int_{\R^{\times}} \omega_{\pi}(z) |z|_{\R}^{ns} \int\limits_{\Ngp_{n - 1}(\R) \backslash \GL_{n - 1}(\R)} \left|\det h\right|_{\R}^{s - 1} \int_{\Ugp(n)} W_{\pi}^{\circ}\left(\begin{pmatrix} h & 0 \\ 0 & 1 \end{pmatrix} k'\right)	\\
\times \dim \tau_{\chi_{\pi},c(\pi)} \int_{\Ogp(n)} P_{\chi_{\pi},c(\pi)}^{\circ}(e_n k^{\prime -1} k^{-1}) \Phi(z e_n k^{-1}) \, dk \, dk' \, dh \, d^{\times}z.
\end{multline*}
By the addition theorem, \hyperref[prop:addition]{Proposition \ref*{prop:addition}}, the last line is
\[\sum_{\ell = 1}^{\dim \tau_{\chi_{\pi},c(\pi)}} Q_{\ell}(e_n k^{\prime -1}) \int_{\Ogp(n)} \overline{Q_{\ell}}(e_n k) \Phi(z e_n k^{-1}) \, dk,\]
where $\{Q_{\ell}\}$ is an orthonormal basis of $\HH_{\chi_{\pi},c(\pi)}(\C^n)$. By the homogeneity of $Q_{\ell}$, \eqref{eqn:homogeneity}, we observe that the restrictions of these polynomials to $\R^n$ are elements of $\PP_{\chi_{\pi}|_{\Ogp(1)},c(\pi)}(\R^n)$.

We now use the fact that
\[\Phi(ze_n k^{-1}) = \overline{\chi_{\pi}}\left(\frac{z}{\|z\|}\right) \|z\|^{c(\pi)} e^{-\pi z^2} \sum_{\substack{m = c(\chi_{\pi}|_{\Ogp(1)}) \\ m \equiv c(\chi_{\pi}|_{\Ogp(1)}) \hspace{-.2cm} \pmod{2}}}^{c(\pi)} (\dim \tau_{\chi_{\pi}|_{\Ogp(1)},m}) P_{\chi_{\pi}|_{\Ogp(1)},m}^{\circ}(e_n k)\]
for $z \in \R^{\times}$ and $k \in \Ogp(n)$ by the homogeneity of $P_{\chi_{\pi}|_{\Ogp(1)},m}^{\circ}$ as in \eqref{eqn:homogeneity} coupled with the identity \eqref{eqn:Pconj}. By \eqref{eqn:Preproducing}, the sum over $m$ is the reproducing kernel of $\PP_{\chi_{\pi}|_{\Ogp(1)},c(\pi)}(\R^n)$, and so the integral over $\Ogp(n) \ni k$ is $\overline{Q_{\ell}}(e_n)$ by \eqref{eqn:Pmatrixcoeff}. Using the addition theorem, \hyperref[prop:addition]{Proposition \ref*{prop:addition}}, in \emph{reverse} and then using \eqref{eqn:Whitnewformintegral} to evalute the integral over $\Ugp(n) \ni k'$, we find that $I(s,W_{\pi}^{\circ},\Phi)$ is equal to
\[\int_{\R^{\times}} \omega_{\pi}(z) \overline{\chi_{\pi}}\left(\frac{z}{\|z\|}\right) \|z\|^{c(\pi)} |z|_{\R}^{ns} e^{-\pi z^2} \, d^{\times}z \int\limits_{\Ngp_{n - 1}(\R) \backslash \GL_{n - 1}(\R)} W_{\pi}^{\circ}\begin{pmatrix} h & 0 \\ 0 & 1 \end{pmatrix} \left|\det h\right|_{\R}^{s - 1} \, dh.\]
It remains to recall \hyperref[lem:omegapitoomegapiur]{Lemma \ref*{lem:omegapitoomegapiur}}, which, by \eqref{eqn:Lsomega}, shows that the integral over $\R^{\times} \ni z$ is $L(ns,\omega_{\pi_{\ur}}|_{\R^{\times}})$.
\end{proof}

\begin{remark}
Just as for \hyperref[prop:RSkey]{Proposition \ref*{prop:RSkey}}, the same proof remains valid in the nonarchimedean setting.
\end{remark}

Finally, we use \hyperref[thm:Flickerspher]{Theorem \ref*{thm:Flickerspher}} to show that when $\pi$ is an induced representation of Langlands type of $\GL_n(\C)$, the newform $W_{\pi}^{\circ} \in \WW(\pi,\psi_{\C/\R})$ is a weak test vector for the Flicker integral. The proof proceeds by reducing the problem to the spherical case.

\begin{theorem}
\label{thm:Flicker}
With the notation and hypotheses of \hyperref[prop:FlickerKey]{Proposition \ref*{prop:FlickerKey}}, the Flicker integral 
\begin{equation}
\label{eqn:Flickerint}
\Psi(s,W_{\pi}^{\circ},\Phi) \coloneqq \int\limits_{\Ngp_n(\R) \backslash \GL_n(\R)} W_{\pi}^{\circ}(g) \Phi(e_n g) \left|\det g\right|_{\R}^s \, dg
\end{equation}
is equal to $L(s,\pi_{\ur},\As)$.
\end{theorem}

\begin{proof}
By \hyperref[lem:WtoWur]{Lemma \ref*{lem:WtoWur}} and \hyperref[prop:FlickerKey]{Proposition \ref*{prop:FlickerKey}}, we have that
\[\Psi(s,W_{\pi}^{\circ},\Phi) = L\left(ns, \omega_{\pi_{\ur}}|_{\R^{\times}}\right) \int\limits_{\Ngp_{n - 1}(\R) \backslash \GL_{n - 1}(\R)} W_{\pi_{\ur}}^{\circ}\begin{pmatrix} h & 0 \\ 0 & 1 \end{pmatrix} \left|\det h\right|_{\R}^{s - 1} \, dh.\]
On the other hand, the same calculation shows that the right-hand side is equal to $\Psi(s,W_{\pi_{\ur}}^{\circ},\Phi_{\ur})$ with $\Phi_{\ur} \in \Scr(\R^n)$ given by \eqref{eqn:Phiurdefeq}. It remains to invoke \hyperref[thm:Flickerspher]{Theorem \ref*{thm:Flickerspher}}.
\end{proof}

In general, the na\"{i}ve Asai $L$-function $L(s,\pi_{\ur}, \As)$ is not equal to $L(s,\pi,\As)$. Nonetheless, these two $L$-functions are closely related.

\begin{proposition}
Given an induced representation of Whittaker type $\pi$ of $\GL_n(\C)$, there exists a polynomial $p(s)$ for which 
\[
L(s,\pi_{\ur},\As) = p(s) L(s,\pi,\As).
\]
\end{proposition}

\begin{proof}
For $\pi = \bigboxplus_{j = 1}^{n} \pi_j$, we have that
\[L(s,\pi,\As) = \prod_{j = 1}^{n} L(s,\pi_j,\As) \prod_{1 \leq j < \ell \leq n} L(s, \pi_j \times \pi_{\ell})\]
via the local Langlands correspondence (cf.\ \cite[Lemma 3.2.1]{B-P21}). Recalling \hyperref[prop:RSpolynomial]{Proposition \ref*{prop:RSpolynomial}}, it thereby suffices to consider the case where $\pi$ is essentially square-integrable.

For $\pi$ essentially square-integrable, so that $\pi = e^{i\kappa \arg} |\cdot|_{\C}^{t}$ and $\pi_{\ur} = |\cdot|_{\C}^{t + \|\kappa\|/2}$, the fact that $\zeta_{\R}(s) \coloneqq \pi^{-s/2} \Gamma(s/2)$ and that $\Gamma(s + 1) = s\Gamma(s)$ means that for $\kappa' \in \{0,1\}$ satisfying $\kappa \equiv \kappa' \pmod{2}$,
\[\frac{L(s,\pi_{\ur},\As)}{L(s,\pi,\As)} = \frac{\zeta_{\R}\left(s + 2t + \|\kappa\|\right)}{\zeta_{\R}\left(s + 2t + \kappa'\right)} = \pi^{-\frac{\|\kappa\| - \kappa'}{2}} \prod_{m = 0}^{\frac{1}{2}(\|\kappa\| - \kappa') - 1} \left(\frac{s + 2t + \kappa'}{2} + m\right).\qedhere\]
\end{proof}

When $s = 1$, the integral appearing in \eqref{eqn:FRperiod} is known as the \emph{Flicker--Rallis period}. In general, this integral need not converge at $s = 1$. When $\pi$ is \emph{unitary}, however, convergence is guaranteed by the following lemma, whose proof we omit as it is a standard application \citelist{\cite{Kem15b}*{Lemma 7.1}; \cite{FZ95}*{Proposition 1}} of bounds for Whittaker functions by a gauge \cite[\S 4 Propositions 2 and 3]{JS90}.

\begin{lemma}
\label{Flicker-convergence}
Let $\pi$ be a unitary generic irreducible Casselman--Wallach representation of $\GL_n(\C)$. For any $W_{\pi} \in \WW(\pi,\psi_{\C/\R})$, the integral
\[
\int\limits_{\Ngp_{n - 1}(\R) \backslash \GL_{n - 1}(\R)} W_{\pi} \begin{pmatrix} h & 0 \\ 0 & 1 \end{pmatrix} \left|\det h\right|_{\R}^{s - 1} \, dh
\]
converges absolutely for $\Re (s) \geq 1$.
\end{lemma}

We explicitly evaluate the Flicker--Rallis period at the Whittaker newform. The nonvanishing of this integral (in a slightly modified form) is implicitly described in work of Gelbart, Jacquet, and Rogawski \cite[\S 2]{GJR01},  Kemarsky \citelist{\cite{Kem15a}*{\S 1}; \cite{Kem15b}*{\S 8}}, and W.\ Zhang \cite[(3.14) and (3.21)]{Zha14}.
 
\begin{theorem}
\label{thm:FRperiod}
Let $\pi$ be a unitary generic irreducible Casselman--Wallach representation of $\GL_n(\C)$ with Whittaker newform $W_{\pi}^{\circ} \in \WW(\pi,\psi_{\C/\R})$. Then the Flicker--Rallis period
\begin{equation}
\label{eqn:FRperiod2}
\int\limits_{\Ngp_{n - 1}(\R) \backslash \GL_{n - 1}(\R)} W_{\pi}^{\circ}\begin{pmatrix} h & 0 \\ 0 & 1 \end{pmatrix}  \, dh
\end{equation}
is equal to
\[
\frac{L(1,\pi_{\ur},\As)}{L\left(n, \omega_{\pi_{\ur}}|_{\R^{\times}}\right) }.
\]
\end{theorem}

\begin{proof}
With \hyperref[Flicker-convergence]{Lemma \ref*{Flicker-convergence}} in hand, we take $s=1$ in \hyperref[prop:FlickerKey]{Proposition \ref*{prop:FlickerKey}} and combine this with \hyperref[thm:Flicker]{Theorem \ref*{thm:Flicker}}.
\end{proof}

\begin{remark}
The nonarchimedean analogue of this result has been resolved by Anandavardhanan and Matringe \cite[Theorem 1.1]{AM17}.
\end{remark}

We say that $\pi$ is \emph{$\GL_n(\R)$-distinguished} if $\Hom_{\GL_n(\R)}(\pi,\mathbf{1})$ is nontrivial.

\begin{remark}[{\cite[\S 6]{Kem15b}}]
\label{rmk:InvarinatFRperiod}
We define the $\Pgp_n(\R)$-invariant linear functional $\vartheta^{\flat} : \WW(\pi,\psi_{\C/\R}) \to \C$ by
\begin{equation}
\label{eqn:varthetaflat}
\vartheta^{\flat}(W_{\pi}) \coloneqq \int\limits_{\Ngp_{n - 1}(\R) \backslash \GL_{n - 1}(\R)} W_{\pi} \begin{pmatrix} h & 0 \\ 0 & 1 \end{pmatrix}  \, dh.
\end{equation}
Owing to \hyperref[thm:FRperiod]{Theorem \ref*{thm:FRperiod}}, the linear functional $\vartheta^{\flat}$ is nontrivial. In addition, Kemarsky \cite[Theorem 1.1]{Kem15a} shows that a $\Pgp_n(\R)$-invariant linear functional extends to a $\GL_n(\R)$-invariant linear functional in a purely local manner. Alternatively, a local-to-global method may be adapted as in \cite[p.\ 185--186]{GJR01} (cf.\ \cite[Proposition 3.2]{Zha14}). Hence $\vartheta^{\flat}$ gives rise to a nontrivial $\GL_n(\R)$-invariant linear functional on the Whittaker model $\WW(\pi,\psi_{\C/\R})$.
\end{remark}

\section{Bump--Friedberg Integrals}
\label{sect:BFsect}

For $n = 2m$, we define the embedding $J: \GL_m(F) \times \GL_m(F) \to \GL_n(F)$ by
\[
J(g,g^{\prime})_{k,\ell} \coloneqq \begin{dcases*}
g_{i,j} &  if $k=2i-1$ and $\ell=2j-1$, \\
g^{\prime}_{i,j} &  if $k=2i$ and $l=2j$, \\
0 &  otherwise.
\end{dcases*}
\]
We denote by $\Mgp_{m,m}(F) \cong \GL_m(F) \times \GL_m(F)$ the standard Levi subgroup of $\GL_n(F)$ associated with the partition $(m,m)$. Let
\[
w_{m,m} \coloneqq \begin{pmatrix} 1 & 2 & \cdots & m & | & m+1 & m+2 & \cdots & 2m \\ 1 & 3 & \cdots & 2m-1 & | &  2 & 4 & \cdots &2m \end{pmatrix}.
\]
We then set
\begin{align*}
\Hgp_{m,m}(F) & \coloneqq w_{m,m}\Mgp_{m,m}(F)w^{-1}_{m,m}	\\
& = \left\{ J(g,g')=w_{m,m}\diag(g,g')w^{-1}_{m,m} : \diag(g,g') \in  \Mgp_{m,m}(F) \right\}.
\end{align*}
Similarly, for $n = 2m + 1$, we define the embedding $J: \GL_{m+1}(F) \times \GL_m(F) \to \GL_n(F)$ by
\[
J(g,g^{\prime})_{k,\ell} \coloneqq \begin{dcases*}
g_{i,j} & if $k=2i-1$ and $\ell=2j-1$, \\
g^{\prime}_{i,j} & if $k=2i$ and $\ell=2j$, \\
0 &  otherwise.
\end{dcases*}
\]
We denote by $\Mgp_{m + 1,m}(F) \cong \GL_{m + 1}(F) \times \GL_m(F)$ the standard Levi subgroup associated to the partition $(m+1,m)$ of $2m+1$. Let $w_{m+1,m} \coloneqq w_{m+1,m+1}|_{\GL_{2m+1}(F)}$, so that
\[
w_{m+1,m} \coloneqq \begin{pmatrix} 1 & 2 & \cdots & m+1 & | & m+2 & m+3 & \cdots &2m & 2m+1 \\ 1 & 3 & \cdots & 2m+1 & | &  2 & 4 & \cdots & 2m-2 &2m \end{pmatrix},
\]
and then set
\begin{align*}
\Hgp_{m+1,m}(F) & \coloneqq w_{m+1,m}\Mgp_{m+1,m}(F)w^{-1}_{m+1,m}	\\
& = \left\{ J(g,g')=w_{m+1,m}\diag(g,g')w^{-1}_{m+1,m} : \diag(g,g') \in  \Mgp_{m+1,m}(F) \right\}.
\end{align*}

To make the above description much more transparent, we provide prototypical elements in the cases of $\Hgp_{2,2}(F)$ and $\Hgp_{3,2}(F)$.

\begin{example}
The group $ \Hgp_{2,2}(F)$ consists of invertible matrices of the form
\[
J \left( \begin{pmatrix} g_{1,1} & g_{1,2} \\ g_{2,1} & g_{2,2} \end{pmatrix}, \begin{pmatrix} g'_{1,1} & g'_{1,2} \\ g'_{2,1} & g'_{2,2} \end{pmatrix}  \right)= \begin{pmatrix} g_{1,1} &0&g_{1,2}&0 \\ 0&g'_{1,1}&0&g'_{1,2}\\ g_{2,1} &0&g_{2,2}&0 \\ 0&g'_{2,1}&0&g'_{2,2} \end{pmatrix}.
\]
The group $\Hgp_{3,2}(F)$ consists of invertible matrices of the form
\[
J \left( \begin{pmatrix} g_{1,1} & g_{1,2}  & g_{1,3} \\ g_{2,1} & g_{2,2} & g_{2,3} \\ g_{3,1} & g_{3,2} & g_{3,3} \end{pmatrix}, \begin{pmatrix} g'_{1,1} & g'_{1,2} \\ g'_{2,1} & g'_{2,2} \end{pmatrix}  \right) = \begin{pmatrix} g_{1,1} &    0  &g_{1,2}&    0   & g_{1,3} \\  0    &g'_{1,1}&  0   &g'_{1,2}&   0  \\ g_{2,1} &    0  &g_{2,2}&     0  & g_{2,3} \\ 0     &g'_{2,1}&   0  &g'_{2,2}&    0 \\ g_{3,1} &   0   &g_{3,2}&    0  & g_{3,3} \end{pmatrix}.
\]
\end{example}

Given an induced representation of Whittaker type $\pi$ of $\GL_n(F)$, a Whittaker function $W_{\pi} \in \WW(\pi,\psi)$, and a Schwartz--Bruhat function $\Phi \in \Scr(F^m)$, where $m \coloneqq \lfloor \frac{n}{2}\rfloor$, the Bump--Friedberg integral \cite{BF90} is given by
\begin{multline*}
B(s_1,s_2,W_{\pi},\Phi)	\\
\coloneqq \begin{dcases*}
\int\limits_{\Ngp_m(F) \backslash \GL_m(F)}  \int\limits_{\Ngp_m(F) \backslash \GL_m(F)}  W_{\pi}(J(g,g')) \Phi(e_m g') \left|\det g\right|^{s_1-\frac{1}{2}} \left|\det g'\right|^{s_2-s_1 + \frac{1}{2}} \, dg \, dg' &	\\
& \hspace{-3cm} for $n = 2m$,	\\
\int\limits_{\Ngp_m(F) \backslash \GL_m(F)}  \int\limits_{\Ngp_{m+1}(F) \backslash \GL_{m+1}(F)}  W_{\pi}(J(g,g')) \Phi(e_{m+1}g) \left|\det g\right|^{s_1} \left|\det g'\right|^{s_2-s_1} \, dg \, dg' &	\\
& \hspace{-3cm} for $n = 2m + 1$.
\end{dcases*}
\end{multline*}
Alternatively, one can write this as an integral over $(\Ngp_n(F) \cap \Hgp_{m,m}(F)) \backslash \Hgp_{m,m}(F)$ for $n = 2m$ and as an integral over $(\Ngp_n(F) \cap \Hgp_{m + 1,m}(F)) \backslash \Hgp_{m + 1,m}(F)$ for $n = 2m + 1$ (cf. \cite{Mat15, Mat17}). The Bump--Friedberg integrals converges absolutely for $\Re(s)$ sufficiently large and extends meromorphically to the entire complex plane. The local exterior square $L$-function $L(s,\pi,\wedge^2)$ is defined via the local Langlands correspondence as illustrated in \cite{Kna94} and \cite[Section 3.2]{B-P21}. We take the local Bump--Friedberg $L$-function $L(s_1,s_2,\pi,\BF)$ to be $L(s_1,\pi)L(s_2,\pi,\wedge^2)$. For our purposes, it will often be convenient to write $B(s,W_{\pi},\Phi)$ in place of $B(s,2s,W_{\pi},\Phi)$ and $L(s,\pi,\BF)$ in place of $L(s,\pi) L(2s,\pi,\wedge^2)$ when $n=2m$ is even.

\subsection{The Spherical Calculation}

We summarise a propagation formula for $\GL_n(F)$ Whittaker functions in terms of $\GL_{n-1}(F)$ Whittaker functions and a convolution section identity for radial parts. We omit the proofs, since they are essentially identical to the corresponding proofs of \hyperref[propformula]{Lemmata \ref*{propformula}} and \ref{convsectid}.

\begin{lemma}[{Cf.\ \hyperref[propformula]{Lemma \ref*{propformula}}}]
\label{formula1}
Let $\pi=\bigboxplus_{j=1}^n |\cdot|^{t_j}$ and $\pi_0\coloneqq \bigboxplus_{j=2}^n |\cdot|^{t_j}$ be the spherical induced representations of Langlands type of $\GL_n(F)$ and $\GL_{n-1}(F)$ respectively with spherical Whittaker functions $W^{\circ}_{\pi} \in \WW(\pi,\psi)$ and $W^{\circ}_{\pi_0} \in \WW(\pi_0,\psi)$. Then for $a=\diag(a_1,\ldots,a_n) \in \Agp_n(F)$, we have
\begin{multline}
\label{eqn:GLnpropagation}
W^{\circ}_{\pi}(a) = \left|\det a\right|^{t_1} \delta^{1/2}_{\Bgp_n(F)}(a) \int_{\Agp_{n-1}(F)} W^{\circ}_{\pi_0}(a') \prod_{j=1}^{n-1}  \exp(-d_F\pi \|a'_ja^{-1}_{j+1}\|^2) 	\\
\times  \exp(-d_F\pi \|a'^{-1}_ja_j\|^2)\left|\det a'\right|^{-t_1} \delta^{-1/2}_{\Bgp_{n-1}(F)}(a') \, d^{\times} a'.
\end{multline}
\end{lemma}

\begin{lemma}[{Cf.\ \hyperref[convsectid]{Lemma \ref{convsectid}}}]
\label{formula2} Let $\pi=\bigboxplus_{j=1}^n |\cdot|^{t_j}$ be a spherical induced representation of Langlands type of $\GL_n(F)$ with spherical Whittaker function $W^{\circ}_{\pi} \in \WW(\pi,\psi)$. Then for $a=\diag(a_1,\ldots,a_n) \in \Agp_n(F)$, we have
\begin{multline*}
L(s,\pi)  W^{\circ}_{\pi}(a') = \int_{\Agp_{n}(F)} W^{\circ}_{\pi}(a'a) \prod_{j=1}^{n} \exp(-d_F\pi \|a_j\|^2) \prod_{j=1}^{n-1} \exp(-d_F\pi \|a'_ja'^{-1}_{j+1}a^{-1}_{j+1}\|^2)	\\
\times \left|\det a\right|^s_F \delta^{-1/2}_{\Bgp_{n}(F)}(a) \, d^{\times} a.
\end{multline*}
\end{lemma}

We will use these identities after reducing Bump--Friedberg integrals to integrals over a torus.

\begin{proposition} 
\label{BF-Iwasawa}
Let $\pi=\bigboxplus_{j=1}^{n} |\cdot|^{t_j}$ be a spherical induced representation of Langlands type of $\GL_n(F)$ with spherical Whittaker function $W^{\circ}_{\pi} \in \WW(\pi,\psi)$, and for $m = \lfloor \frac{n}{2}\rfloor$, let $\Phi_{\ur} \in \Scr(F^m)$ be given by \eqref{eqn:Phiurdefeq}.
\begin{enumerate}[leftmargin=*,label=\emph{(\roman*)}]
\item\label{BF-Spherical-even} For $n=2m$, the Bump--Friedberg integral $B(s_1,s_2,W_{\pi}^{\circ},\Phi_{\ur})$ is equal to
\[
\int_{\Agp_{n}(F)} W^{\circ}_{\pi}(b) \exp(-d_F\pi \|b_n\|^2) \left|b_1 b_3 \cdots b_{n-1}\right|^{s_1}  \left|b_2 b_4 \cdots b_{n-2}\right|^{s_2-s_1} \delta^{-\frac{1}{2}}_{\Bgp_n(F)}(b) \, d^{\times}b.
\]
\item\label{BF-Spherical-odd} For $n=2m+1$, the Bump--Friedberg integral $B(s_1,s_2,W_{\pi}^{\circ},\Phi_{\ur})$ is equal to
\[
\int_{\Agp_n(F)} W^{\circ}_{\pi}(b) \exp(-d_F\pi \|b_n\|^2) \left|b_1 b_3 \cdots b_{n-2}\right|^{s_1}  \left|b_2 b_4\cdots b_{n - 1}\right|^{s_2-s_1} \delta^{-\frac{1}{2}}_{\Bgp_n(F)}(b) \, d^{\times} b.
\]
\end{enumerate}
\end{proposition}

\begin{proof}
We prove this only for $n = 2m$; the case $n = 2m + 1$ follows analogously. Exploiting the Iwasawa decomposition $g = ak$ and $g' = a'k'$ for $\Ngp_m(F) \backslash \GL_m(F)$, the Bump--Friedberg integral $B(s_1,s_2,W_{\pi}^{\circ},\Phi_{\ur})$ is equal to
\begin{multline*}
\int_{\Agp_m(F)} \int_{\Agp_{m}(F)} W^{\circ}_{\pi}(J(a,a')) \exp(-d_F\pi \|a'_m\|^2) \left|\det a\right|^{s_1-\frac{1}{2}} \left|\det a'\right|^{s_2-s_1 + \frac{1}{2}}  \\
\times \delta^{-1}_{\Bgp_m(F)}(a) \delta^{-1}_{\Bgp_m(F)}(a') \, d^{\times} a \, d^{\times} a'.
\end{multline*}
Since $\left|\det a\right|^{-\frac{1}{2}} \left|\det a'\right|^{\frac{1}{2}} \delta^{-1}_{\Bgp_m(F)}(a) \delta^{-1}_{\Bgp_m(F)}(a') = \delta^{-\frac{1}{2}}_{\Bgp_{n}(F)}(J(a,a'))$, this is equal to
\[
\int_{\Agp_m(F)} \int_{\Agp_{m}(F)} W^{\circ}_{\pi}(J(a,a')) \exp(-d_F\pi \|a'_m\|^2) \left|\det a\right|^{s_1}  \left|\det a'\right|^{s_2-s_1} \delta^{-\frac{1}{2}}_{\Bgp_n(F)}(J(a,a')) \, d^{\times} a \, d^{\times} a'.
\]
The result now follows by writing $a=\diag(b_1,b_3,\ldots,b_{n-1})$ and $a'=\diag(b_2,b_4,\ldots,b_{n})$ and letting $b \coloneqq (b_1,\ldots,b_n) \in \Agp_n(F)$.
\end{proof}

We are now able to prove that when $\pi$ is a spherical induced representation of Langlands type of $\GL_n(F)$, the spherical Whittaker function $W_{\pi}^{\circ} \in \WW(\pi,\psi)$ is a strong test vector for the Bump--Friedberg integral.

\begin{theorem}
\label{BFspherical}
Let $\pi=\bigboxplus_{j=1}^n |\cdot|^{t_j}$ be a spherical induced representation of Langlands type of $\GL_n(F)$
with spherical Whittaker function $W^{\circ}_{\pi} \in \WW(\pi,\psi)$, and for $m = \lfloor \frac{n}{2}\rfloor$, let $\Phi_{\ur} \in \Scr(F^m)$ be given by \eqref{eqn:Phiurdefeq}. Then the Bump--Friedberg integral $B(s_1,s_2,W_{\pi}^{\circ},\Phi_{\ur})$ is equal to
\[
L(s_1,s_2,\pi,\BF)\coloneqq L(s_1,\pi)L(s_2,\pi,\wedge^2)=\prod_{\ell=1}^n  \zeta_F(s_1+t_{\ell}) \prod_{1 \leq j < k \leq n} \zeta_F(s_2+t_j+t_k).  
\]
\end{theorem}

\begin{proof}
We prove this by induction. For the base case $n = 2$, the Bump--Friedberg integral is
\[\int_{F^{\times}} \int_{F^{\times}} W_{\pi}^{\circ}\begin{pmatrix} a_1 & 0\\ 0 & a_2 \end{pmatrix} \exp(-d_F \pi \|a_2\|^2) |a_1|^{s_1 - \frac{1}{2}} |a_2|^{s_2 - s_1 + \frac{1}{2}} \, d^{\times}a_1 \, d^{\times}a_2.\]
From \eqref{eqn:GLnpropagation}, this is equal to
\begin{multline*}
\int_{F^{\times}} \int_{F^{\times}} \exp(-d_F \pi \|a_2\|^2) |a_1|^{s_1 + t_1} |a_2|^{s_2 - s_1 + t_1}	\\
\times \int_{F^{\times}} \exp(-d_F \pi \|a' a_2^{-1}\|^2) \exp(-d_F \pi \|a'^{-1} a_1\|^2) |a'|^{t_2 - t_1} \, d^{\times}a' \, d^{\times}a_1 \, d^{\times}a_2.
\end{multline*}
We interchange the order of integration and make the change of variables $a_1 \mapsto a' a_1$ and $a' \mapsto a' a_2$, yielding
\begin{multline*}
\int_{F^{\times}} \exp(-d_F \pi \|a_1\|^2) |a_1|^{s_1 + t_1} \, d^{\times}a_1 \int_{F^{\times}} \exp(-d_F \pi \|a_2\|^2) |a_2|^{s_2 + t_1 + t_2} \, d^{\times}a_2	\\
\times \int_{F^{\times}} \exp(-d_F \pi \|a'\|^2) |a'|^{s_1 + t_2} \, d^{\times}a',
\end{multline*}
which is precisely $L(s_1,s_2,\pi,\BF)$.

Now we proceed to the induction step. We suppose that the desired identity holds for $n = 2m - 1$ and we prove this for $n = 2m$. We insert the recursive formula for $W_{\pi}^{\circ}(a)$ from \hyperref[formula1]{Lemma \ref*{formula1}} into the expression given in \hyperref[BF-Iwasawa]{Proposition \ref*{BF-Iwasawa}} \emph{\ref{BF-Spherical-even}} for the Bump--Friedberg integral $B(s_1,s_2,W_{\pi}^{\circ},\Phi_{\ur})$. We then relabel $b \in \Agp_n(F)$ with $\diag(y,b)$, where now $y \in F^{\times}$ and $b=\diag(b_1,b_2,\ldots,b_{n-1}) \in \Agp_{n-1}(F)$. The modulus character $\delta^{-1/2}_{\Bgp_{n}(F)}(b)$ is cancelled out and we arrive at
\begin{multline*}
\int_{\Agp_{n-1}(F)} \exp(-d_F\pi \|b_{n-1}\|^2)  \left|b_2 b_4 \cdots b_{n-2}\right|^{s_1+t_1}  \left|b_1 b_3 \cdots b_{n-1}\right|^{s_2-s_1+t_1} \\
\times \int_{\Agp_{n-1}(F)} W^{\circ}_{\pi_0}(b')  \prod_{j=1}^{n-1} \exp(-d_F\pi \|b'_jb^{-1}_{j}\|^2)  \prod_{j=1}^{n-2}  \exp(-d_F\pi \|b'^{-1}_{j+1}b_j\|^2) \left|\det b'\right|^{-t_1} \delta^{-\frac{1}{2}}_{\Bgp_{n-1}(F)}(b') \\
\times \int_{F^{\times}} \exp(-d_F\pi \| b'^{-1}_1y \|^2) |y|^{s_1+t_1} \, d^{\times}y \, d^{\times} b' \, d^{\times}b.
\end{multline*}
We interchange the order of integration and perform the change of variables $y \mapsto b'_1y$ and $b_j \mapsto b'_{j+1}b_j$ for $1 \leq j \leq n-2$. The Bump--Friedberg integral $B(s_1,s_2,W_{\pi}^{\circ},\Phi_{\ur})$ becomes
\begin{multline*}
\int_{\Agp_{n-1}(F)}\prod_{j=1}^{n-1}  \exp(-d_F\pi \|b_j\|^2)  \left|b_2 b_4 \cdots b_{n-2}\right|^{s_1+t_1}  \left|b_1 b_3 \cdots b_{n-1}\right|^{s_2-s_1+t_1} \\
\times \int_{\Agp_{n-1}(F)} W^{\circ}_{\pi_0}(b')  \prod_{j=1}^{n-2} \exp(-d_F\pi \|b'_jb'^{-1}_{j+1}b^{-1}_j\|^2) \exp(-d_F\pi \|b'_{n-1}b^{-1}_{n-1}\|^2)  \delta^{-\frac{1}{2}}_{\Bgp_{n-1}(F)}(b') \\
\times  \left|b'_1 b'_3 \cdots b'_{n-1}\right|^{s_1}  \left|b'_2 b'_4 \cdots b'_{n-2} \right|^{s_2-s_1}  \int_{F^{\times}} \exp(-d_F\pi \| y \|^2) |y|^{s_1+t_1} \, d^{\times}y \, d^{\times} b' \, d^{\times}b.
\end{multline*}
The integral over $F^{\times} \ni y$ is simply $ \zeta_F(s_1+t_1)$. We make the change of variables $b'_j \mapsto b'_jb_{j}$ and then interchange the order of the integration. This leads us to
\begin{multline*}
\zeta_F(s_1+t_1) \int_{\Agp_{n-1}(F)}  \exp(-d_F\pi \|b'_{n-1}\|^2)  \left|b'_1 b'_3 \cdots b'_{n-1}\right|^{s_1}  \left|b'_2 b'_4 \cdots b'_{n-2} \right|^{s_2-s_1} \delta^{-\frac{1}{2}}_{\Bgp_{n-1}(F)}(b')	\\
\times \int_{\Agp_{n-1}(F)} W^{\circ}_{\pi_0}(b'b)  \prod_{j=1}^{n-1} \exp(-d_F\pi \|b_j\|^2)  \exp(-d_F\pi \|b'_jb'^{-1}_{j+1}b^{-1}_{j+1}\|^2)	\\
 \times \left|\det b\right|^{s_2+t} \delta^{-\frac{1}{2}}_{\Bgp_{n-1}(F)}(b) \, d^{\times} b \, d^{\times}b'.
\end{multline*}
By \hyperref[formula2]{Lemma \ref*{formula2}}, the integral over $\Agp_{n-1}(F) \ni b$ is precisely $L(s_2+t_1,\pi_0)W^{\circ}_{\pi_0}(b')$. According to \hyperref[BF-Iwasawa]{Proposition \ref*{BF-Iwasawa}} \emph{\ref{BF-Spherical-odd}}, we end up with
\begin{multline*}
\zeta_F(s_1+t_1) L(s_2+t_1,\pi_0)\\
\times \int\limits_{\Ngp_{m-1}(F) \backslash \GL_{m-1}(F)}  \int\limits_{\Ngp_{m}(F) \backslash \GL_{m}(F)}  W^{\circ}_{\pi}(J(g,g')) \Phi_{\ur}(e_m g) \left|\det g\right|^{s_1} \left|\det g'\right|^{s_2-s_1} \, dg \, dg'
\end{multline*}
from which the desired identity holds by the induction hypothesis.

The same method of proof remains valid for the induction step when $n = 2m + 1$, where we suppose that the desired identity holds for $n = 2m$; the only difference is that we appeal to \hyperref[BF-Iwasawa]{Proposition \ref*{BF-Iwasawa}} \emph{\ref{BF-Spherical-even}} in place of \hyperref[BF-Iwasawa]{Proposition \ref*{BF-Iwasawa}} \emph{\ref{BF-Spherical-odd}}.
\end{proof}

When $F = \R$, \hyperref[BFspherical]{Theorem \ref*{BFspherical}} recovers an earlier result of Stade \cite[Theorem 3.3]{Sta01} proven via different means. Invoking \cite[Proposition 2.1]{Sta95}, Stade's method can also be used to prove \hyperref[BFspherical]{Theorem \ref*{BFspherical}} when $F = \C$.

\begin{remark}
Ishii \cite[Theorem 4.1]{Ish18} has proven the existence of a \emph{strong} test vector for the Bump--Friedberg integral when $F = \R$ and $\pi$ is a principal series representation.
\end{remark}

\subsection{Test Vectors for Bump--Friedberg Integrals and Friedberg--Jacquet Periods}

We proceed to the more general case where $\pi$ may be ramified. Our first step is to reduce the Bump--Friedberg integral to a double integral over $\Ngp_{m - 1}(F) \backslash \GL_{m -1}(F)$ and $\Ngp_m(F) \backslash \GL_m(F)$ when $n = 2m$ and over $\Ngp_m(F) \backslash \GL_m(F)$ and $\Ngp_m(F) \backslash \GL_m(F)$ when $n = 2m + 1$.

\begin{proposition}
\label{prop:BFKey}
Let $\pi$ be an induced representation of Langlands type of $\GL_n(F)$ with Whittaker newform $W^{\circ}_{\pi} \in \WW(\pi,\psi)$. For $m = \lfloor \frac{n}{2}\rfloor$, let $\Phi \in \Scr(F^m)$ be the Schwartz--Bruhat function given by $\Phi(x)=P(x) \exp(-d_F \pi x \prescript{t}{}{\overline{x}})$, where the distinguished homogeneous polynomial $P \in \PP_{\overline{\chi_{\pi}},c(\pi)}(F^m)$ is taken to be
\[
P(x)\coloneqq \sum^{c(\pi)}_{\substack{j=c(\chi_{\pi}) \\	j \equiv c(\chi_{\pi}) \hspace{-.2cm} \pmod{2} }}  \left(x \prescript{t}{}{\overline{x}}\right)^{\frac{c(\pi)-j}{2}}( \dim \tau_{\chi_{\pi},j}) \overline{P^{\circ}_{\chi_{\pi},j}}(x).
\]

\begin{enumerate}[leftmargin=*,label=\emph{(\roman*)}]
\item\label{BF-Tate-1} Let $n=2m$. For $\Re(s_1)$ and $\Re(s_2)$ sufficiently large, the Bump--Friedberg integral $B(s_1,s_2,W_{\pi}^{\circ},\Phi)$ is equal to
\begin{multline}
\label{eqn:FJperiod1}
L(ms_2,\omega_{\pi_{\ur}}) \int\limits_{\Ngp_{m-1}(F) \backslash \GL_{m-1}(F)} \int\limits_{\Ngp_m(F) \backslash \GL_m(F)} W^{\circ}_{\pi} \left(J\left( g,\begin{pmatrix} h' &0 \\ 0& 1 \end{pmatrix}\right)  \right) \\
\times \left|\det g\right|^{s_1 - \frac{1}{2}} \left|\det h'\right|^{s_2 - s_1 - \frac{1}{2}} \, dg \, dh'.
\end{multline}
\item\label{BF-Tate-2} Let $n=2m+1$. For $\Re(s_1)$ and $\Re(s_2)$ sufficiently large, the Bump--Friedberg integral $B(s_1,s_2,W_{\pi}^{\circ},\Phi)$ is equal to
\begin{multline}
\label{eqn:FJperiod2}
L \left(s_1 + ms_2,\omega_{\pi_{\ur}} \right) \int\limits_{\Ngp_m(F) \backslash \GL_m(F)}  \int\limits_{\Ngp_m(F) \backslash \GL_m(F)} W^{\circ}_{\pi}  \left(J\left( \begin{pmatrix} h&0 \\ 0& 1 \end{pmatrix} ,g'  \right)\right)	\\
\times  \left|\det h\right|^{s_1-1} \left|\det g'\right|^{s_2-s_1} \, dh \, dg'.
\end{multline}
\end{enumerate}
\end{proposition}

Once more, the proof of this proceeds along similar lines to that of \hyperref[prop:RSkey]{Proposition \ref*{prop:RSkey}}.

\begin{proof}
We prove this for $n = 2m$; the case $n = 2m + 1$ follows analogously.We use the Iwasawa decomposition for $\Ngp_m(F) \backslash \GL_m(F) \ni g'$ to write $g' = (z1_m)\begin{psmallmatrix} h' & 0 \\ 0 & 1 \end{psmallmatrix} k'$ and then make the change of variables $g \mapsto (z1_m)g$ on the first copy of $\Ngp_m(F) \backslash\GL_m(F) \ni g$ in order to see that the Bump--Friedberg integral $B(s_1,s_2,W_{\pi}^{\circ},\Phi)$ is equal to
\begin{multline*}
\int_{F^{\times}} \omega_{\pi}(z)|z|^{ms_2} \int\limits_{\Ngp_{m-1}(F) \backslash \GL_{m-1}(F)}  \int\limits_{\Ngp_m(F) \backslash \GL_m(F)} \left|\det g\right|^{s_1-\frac{1}{2}} \left|\det h'\right|^{s_2-s_1 - \frac{1}{2}}	\\
\times \int_{K_m}  W^{\circ}_{\pi} \left(J\left( g,\begin{pmatrix} h' &0 \\ 0& 1 \end{pmatrix}k'\right) \right)     \Phi(ze_m k') \, dk' \, dg \, dh' \, d^{\times}z.
\end{multline*}
We insert the identity \eqref{eqn:Whitnewformintegral} for $W_{\pi}^{\circ}(g)$ with $J\left( g,\begin{psmallmatrix} h' & 0 \\ 0 & 1 \end{psmallmatrix}k'\right)$ in lieu of $g$ and then make the change of variables $k' \mapsto k'^{-1}$ and $k \mapsto J(1_m,k')k$. We arrive at
 \begin{multline*}
\int_{F^{\times}}  \omega_{\pi}(z)|z|^{ms_2} \int\limits_{\Ngp_{m-1}(F) \backslash \GL_{m-1}(F)} \int\limits_{\Ngp_m(F) \backslash \GL_m(F)} \left|\det g\right|^{s_1-\frac{1}{2}} \left|\det h'\right|^{s_2-s_1 - \frac{1}{2}} \\
\times \int_{K_n} W^{\circ}_{\pi} \left(J\left( g,\begin{pmatrix} h' &0 \\ 0& 1 \end{pmatrix}\right) k \right)\\
\times \dim \tau_{\chi_{\pi},c(\pi)}  \int_{K_m} P^{\circ}_{\chi_{\pi},c(\pi)}(e_n k^{-1} J(1_m,k'^{-1}))  \Phi(ze_m k'^{-1}) \, dk' \, dk \, dg \, dh' \, d^{\times}z.
\end{multline*}
By the addition theorem, \hyperref[prop:addition]{Proposition \ref*{prop:addition}}, the last line turns into
\[
\sum_{\ell=1}^{\dim \tau_{\chi_{\pi},c(\pi)}} Q_{\ell}(e_n k^{-1}) \int_{K_m} \overline{Q_{\ell}}(e_n J(1_m,k'))  \Phi(ze_m k'^{-1}) \, dk',
\]
where $\{ Q_{\ell}\}$ is an orthonormal basis of $\HH_{\chi_{\pi},c(\pi)}(F^n)$.

Since $P \in \PP_{\overline{\chi_{\pi}},c(\pi)}(F^m)$, we observe that for $z \in F^{\times}$ and $k' \in K_m$,
\[\Phi(ze_m k'^{-1}) =\overline{\chi_{\pi}} \left( \frac{z}{\|z\|} \right) \|z\|^{c(\pi)} \exp(-d_F\pi \|z\|^2 ) 
\sum^{c(\pi)}_{\substack{j=c(\chi_{\pi}) \\   j \equiv c(\chi_{\pi}) \hspace{-.2cm} \pmod{2}}}  ( \dim \tau_{\chi_{\pi},j}) P^{\circ}_{\chi_{\pi},j}(e_m k')\]
by the homogeneity of $P^{\circ}_{\chi_{\pi},j}$ as in \eqref{eqn:homogeneity} together with the identity \eqref{eqn:Pconj}. By \eqref{eqn:Preproducing}, the sum over $j$ is the reproducing kernel for $\PP_{\chi_{\pi},c(\pi)}(F^m)$, and so the integral over $K_m \ni k'$ is simply $\overline{Q_{\ell}}(e_n)$ by \eqref{eqn:Pmatrixcoeff}. Using the addition theorem, \hyperref[prop:addition]{Proposition \ref*{prop:addition}}, in \emph{reverse} and then using \eqref{eqn:Whitnewformintegral} to evaluate the integral over $K_n \ni k$, we find that the Bump--Friedberg integral $B(s_1,s_2,W_{\pi}^{\circ},\Phi)$ is equal to
\begin{multline*}
\int_{F^{\times}}  \omega_{\pi}(z) \overline{\chi_{\pi}} \left( \frac{z}{\|z\|} \right) \|z\|^{c(\pi)} |z|^{ms_2} \exp(-d_F\pi \|z\|^2 ) \, d^{\times}z	\\
\times \int\limits_{\Ngp_{m-1}(F) \backslash \GL_{m-1}(F)} \int\limits_{\Ngp_m(F) \backslash \GL_m(F)} W^{\circ}_{\pi} \left(J\left( g,\begin{pmatrix} h' &0 \\ 0& 1 \end{pmatrix}\right) \right) \left|\det g\right|^{s_1-\frac{1}{2}} \left|\det g'\right|^{s_2-s_1 - \frac{1}{2}} \, dg \, dh'.
\end{multline*}
It remains to recall \hyperref[lem:omegapitoomegapiur]{Lemma \ref*{lem:omegapitoomegapiur}}, which, by \eqref{eqn:Lsomega}, shows that the integral over $\R^{\times} \ni z$ is $L (ms_2,\omega_{\pi_{\ur}} )$.
\end{proof}

\begin{remark}
Just as for \hyperref[prop:RSkey]{Propositions \ref*{prop:RSkey}} and \ref{prop:FlickerKey}, the same proof remains valid in the nonarchimedean setting.
\end{remark}

Finally, we use \hyperref[BFspherical]{Theorem \ref*{BFspherical}} to show that when $\pi$ is an induced representation of Langlands type of $\GL_n(F)$, $W_{\pi}^{\circ}$ is a weak test vector for the Bump--Friedberg integral. Once more, this is proven by reducing the problem to the spherical case.

\begin{theorem}
\label{thm:BF-reduction}
With the notation and hypotheses of  \hyperref[prop:BFKey]{Proposition \ref*{prop:BFKey}}, the Bump--Friedberg integral 
\begin{multline}
\label{eqn:BFint}
B(s_1,s_2,W_{\pi}^{\circ},\Phi)	\\
\coloneqq \begin{dcases*}
\int\limits_{\Ngp_m(F) \backslash \GL_m(F)}  \int\limits_{\Ngp_m(F) \backslash \GL_m(F)}  W^{\circ}_{\pi}(J(g,g')) \Phi(e_m g')  \left|\det g\right|^{s_1-\frac{1}{2}} \left|\det g'\right|^{s_2-s_1 + \frac{1}{2}} \, dg \, dg' &    \\
& \hspace{-3cm} for $n=2m$,  \\
\int\limits_{\Ngp_m(F) \backslash \GL_m(F)}  \int\limits_{\Ngp_{m+1}(F) \backslash \GL_{m+1}(F)}  W^{\circ}_{\pi}(J(g,g')) \Phi(e_{m+1} g)  \left|\det g\right|^{s_1} \left|\det g'\right|^{s_2-s_1} \, dg \, dg' &	\\
& \hspace{-3cm} for $n=2m+1$,
\end{dcases*}
\end{multline}
is equal to $L(s_1,s_2,\pi_{\ur},\BF)\coloneqq L(s_1,\pi)L(s_2,\pi_{\ur},\wedge^2)$.
\end{theorem}

\begin{proof}
Again, we prove this for $n = 2m$; the case $n = 2m + 1$ follows analogously. Owing to \hyperref[lem:WtoWur]{Lemma \ref*{lem:WtoWur}} and \hyperref[prop:BFKey]{Proposition \ref*{prop:BFKey}}, the Bump--Friedberg integral $B(s_1,s_2,W_{\pi}^{\circ},\Phi)$ is equal to
\begin{multline*}
L \left( ms_2,\omega_{\pi_{\ur}} \right)   \int\limits_{\Ngp_{m-1}(F) \backslash \GL_{m-1}(F)}  \int\limits_{\Ngp_m(F) \backslash \GL_m(F)} W^{\circ}_{\pi_{\ur}} \left(J\left( g,\begin{pmatrix} h' & 0\\ 0& 1 \end{pmatrix}\right)  \right)	\\
\times \left|\det g\right|^{s_1-\frac{1}{2}} \left|\det h'\right|^{s_2-s_1 - \frac{1}{2}} \, dg \, dh'.
\end{multline*}
On the other hand, the same calculation shows that this is equal to $B(s_1,s_2,W_{\pi_{\ur}}^{\circ},\Phi_{\ur})$ with $\Phi_{\ur} \in \Scr(F^m)$ given by \eqref{eqn:Phiurdefeq}. It remains to invoke \hyperref[BFspherical]{Theorem \ref*{BFspherical}}.
\end{proof}

In general, the na\"{i}ve Bump--Friedberg $L$-function $L(s_1,s_2,\pi_{\ur},\BF)$ is not equal to $L(s_1,s_2,\pi,\BF)$. Nonetheless, these two $L$-functions are closely related.

\begin{proposition}
Given an induced representation of Whittaker type $\pi$ of $\GL_n(F)$, there exists a polynomial $p(s_2)$ for which 
\[
L(s_1,s_2,\pi_{\ur},\BF)=p(s_2)L(s_1,s_2,\pi,\BF).
\]
\end{proposition}

\begin{proof}
It is sufficient to show that $L(s,\pi_{\ur},\wedge^2) = p(s) L(s,\pi,\wedge^2)$.
For $\pi = \bigboxplus_{j = 1}^{r} \pi_j$, we have that
\[L(s,\pi,\wedge^2) = \prod_{j = 1}^{r} L(s,\pi_j,\wedge^2) \prod_{1 \leq j < \ell \leq r} L(s, \pi_j \times \pi_{\ell})\]
via the local Langlands correspondence (cf.\ \cite[\S 1.3]{Mat17}). Recalling \hyperref[prop:RSpolynomial]{Proposition \ref*{prop:RSpolynomial}}, it thereby suffices to consider the case where $\pi$ is essentially square-integrable.

If $F = \C$, so that $\pi = e^{i\kappa \arg} |\cdot|_{\C}^{t}$, then $L(s,\pi,\wedge^2) = L(s,\pi_{\ur},\wedge^2) = 1$. Similarly, if $F = \R$ and $\pi = \sgn^{\kappa} |\cdot|_{\R}^{t}$, then $L(s,\pi,\wedge^2) = L(s,\pi_{\ur},\wedge^2) = 1$. Finally, if $F = \R$ and $\pi = D_{\kappa} \otimes \left|\det\right|_{\R}^{t}$, then since $\zeta_{\R}(s) \coloneqq \pi^{-s/2} \Gamma(s/2)$, and recalling the fact that $\Gamma(s + 1) = s\Gamma(s)$, we have that for $\kappa' \in \{0,1\}$ satisfying $\kappa \equiv \kappa' \pmod{2}$,
\[\frac{L(s,\pi_{\ur},\wedge^2)}{L(s,\pi,\wedge^2)} = \frac{\zeta_{\R}(s + 2t + \kappa)}{\zeta_{\R}(s + 2t + \kappa')} = \pi^{-\frac{\kappa - \kappa'}{2}} \prod_{m = 0}^{\frac{1}{2}(\kappa - \kappa') - 1} \left(\frac{s + 2t + \kappa'}{2} + m\right).\qedhere\]
\end{proof}

When $s_1 = 1/2$ and $s_2 = 1$, the integrals appearing in \eqref{eqn:FJperiod1} and \eqref{eqn:FJperiod2} are known as the \emph{Friedberg--Jacquet period} \cite{FJ93}. In general, these integrals need not converge at $(s_1,s_2) = (1/2,1)$. When $\pi$ is \emph{unitary}, however, convergence is guaranteed by the following lemma, whose proof we omit since it is standard \cite[Propositions 3.4 and 5.1]{Mat15}; once more, it follows from bounds for Whittaker functions by a gauge \cite[\S 4 Propositions 2 and 3]{JS90}.

\begin{lemma}
\label{BF-convergence}
Let $\pi$ be a unitary generic irreducible Casselman--Wallach representation of $\GL_n(F)$. For any $W_{\pi} \in \WW(\pi,\psi)$, the integrals
\[
\begin{dcases*}
\int\limits_{\Ngp_{m-1}(F) \backslash \GL_{m-1}(F)} \int\limits_{\Ngp_m(F) \backslash \GL_m(F)} W^{\circ}_{\pi} \left(J\left( g,\begin{pmatrix} h' &0 \\ 0& 1 \end{pmatrix}\right)  \right) \left|\det g\right|^{s - \frac{1}{2}} \left|\det h'\right|^{s - \frac{1}{2}} \, dg \, dh' &    \\
& \hspace{-3cm} for $n=2m$,  \\
\int\limits_{\Ngp_m(F) \backslash \GL_m(F)} \int\limits_{\Ngp_m(F) \backslash \GL_m(F)} W^{\circ}_{\pi}  \left(J\left( \begin{pmatrix} h&0 \\ 0& 1 \end{pmatrix} ,g'  \right)\right) \left|\det h\right|^{s-1} \left|\det g'\right|^s \, dh \, dg' &	\\
& \hspace{-3cm} for $n=2m+1$,
\end{dcases*}
\]
converge absolutely for $\Re (s) \geq 1/2$.
\end{lemma}

We associate Friedberg--Jacquet  periods to certain values of Bump--Friedberg $L$-functions. Indeed, the nonvanishing of these integrals is explained by Matringe  \cite[Proposition 3.5]{Mat15}.

\begin{theorem}
\label{thm:FJperiod}
Let $\pi$ be a unitary generic irreducible Casselman--Wallach representation of $\GL_n(F)$ with Whittaker newform $W_{\pi}^{\circ} \in \WW(\pi,\psi)$. Then the Friedberg--Jacquet period
\begin{equation}
\label{eqn:FJperiod}
\begin{dcases*}
\int\limits_{\Ngp_{m-1}(F) \backslash \GL_{m-1}(F)} \int\limits_{\Ngp_m(F) \backslash \GL_m(F)} W^{\circ}_{\pi}  \left(J\left( g,\begin{pmatrix} h' &0 \\ 0 & 1 \end{pmatrix}\right)  \right) \, dg \, dh'  & for $n = 2m$,	\\
\int\limits_{\Ngp_m(F) \backslash \GL_m(F)}  \int\limits_{\Ngp_m(F) \backslash \GL_m(F)} W^{\circ}_{\pi}  \left(J\left( \begin{pmatrix} h&0 \\ 0& 1 \end{pmatrix} ,g'  \right)\right) \left| \frac{\det h}{\det g'} \right|^{-\frac{1}{2}} \, dh \, dg' & for $n = 2m + 1$,
\end{dcases*}
\end{equation}
is equal to
\[
\frac{L\left(\frac{1}{2},\pi\right) L(1,\pi_{\ur},\wedge^2)}{L \left(\frac{n}{2},\omega_{\pi_{\ur}} \right)}.
\]
\end{theorem}

\begin{proof}
We put $s_1=s$ and $s_2=2s$. Then \hyperref[BF-convergence]{Lemma \ref*{BF-convergence}} ensures that our conclusion can be deduced by taking $s_1 = 1/2$ and $s_2 = 1$ in \hyperref[prop:BFKey]{Proposition \ref*{prop:BFKey}}.
\end{proof}

For $n = 2m$, we say that $\pi$ is \emph{$\Hgp_{m,m}(F)$-distinguished} if $\Hom_{\Hgp_{m,m}(F)}(\pi,\mathbf{1})$ is nontrivial. 
The following result is a weaker version of the analogous results for the Rankin--Selberg period, \hyperref[rmk:InvarinatRSperiod]{Remark \ref*{rmk:InvarinatRSperiod}}, and the Flicker--Rallis period, \hyperref[rmk:InvarinatFRperiod]{Remark \ref*{rmk:InvarinatFRperiod}}.

\begin{proposition}
\label{prop:FJdistinction}
Let $\pi$ be a unitary generic irreducible Casselman--Wallach representation of $\GL_{2m}(F)$ that is $\Hgp_{m,m}(F)$-distinguished. Suppose that $(\pi,V_{\pi})$ occurs as a local component of a unitary cuspidal automorphic representation. Then the linear functional $\vartheta^{\sharp} : \WW(\pi,\psi) \to \C$ given by
\begin{equation}
\label{eqn:varthetasharp}
\vartheta^{\sharp}(W_{\pi}) \coloneqq \int\limits_{\Ngp_{m-1}(F) \backslash \GL_{m-1}(F)} \int\limits_{\Ngp_m(F) \backslash \GL_m(F)} W_{\pi}  \left(J\left( g,\begin{pmatrix} h' &0 \\ 0 & 1 \end{pmatrix}\right)  \right) \, dg \, dh'  
\end{equation}
gives rise to a nontrivial $\Hgp_{m,m}(F)$-invariant linear functional on the Whittaker model $\WW(\pi,\psi)$.
\end{proposition}

While we expect that the assumption that $(\pi,V_{\pi})$ appears in the local component of an automorphic representation is superfluous, the proof that we give under this assumption is valid at the very least for the case of our interest, which pertains to the global setting of period integrals of cuspidal automorphic forms. Our proof requires a local and global argument. We therefore defer the proof to \hyperref[sec:GlobalFJ]{Section \ref*{sec:GlobalFJ}} after we have introduced global automorphic forms. It behoves us to highlight the fact that this global approach is overkill and indirect; one may directly prove the desired result by using a theory of distributions \cite{Bar03,Kem15b}. For the sake of brevity, we only include this indirect approach in keeping the spirit of \cite{GJR01,Zha14}.

\section{Global Applications}
\label{sec:Global}

We now consider the global analogues of the problems investigated in \hyperref[sect:RSsect]{Sections \ref*{sect:RSsect}}, \ref{sect:Flickersect}, and \ref{sect:BFsect}. These pertain to period integrals of automorphic forms on $\GL_n(\A_F)$, where $F$ is a global number field of absolute discriminant $D_{F \slash \Q}$ and $\A_F$ denotes the ring of ad\`{e}les of $F$. The period integrals of interest involve integration over spaces of the form $\Zgp(\A_F) G(F) \backslash G(\A_F)$ for various reductive groups $G$, where $\Zgp$ denotes the centre of $G$. For the sake of notational brevity, we write
\[[G] \coloneqq \Zgp(\A_F) G(F) \backslash G(\A_F).\]
We also write $K_n$ to denote the maximal compact subgroup of $\GL_n(\A_F)$.

Let $\psi_{\A_{\Q}}$ denote the standard additive character of $\A_{\Q}$ that is unramified at every place of $\Q$.
We define the additive character $\psi_{\A_F}$ of $\A_F$ by $\psi_{\A_F}\coloneqq\psi_{\A_{\Q}} \circ \Tr_{\A_F \slash \A_{\Q}}$. The conductor of $\psi_{\A_F}$ is the inverse different $\mathfrak{d}^{-1}$ of $F$. We choose a finite id\`{e}le $d \in \A_F^{\times}$ representing $\mathfrak{d}^{-1}$ such that $\psi_{\A_F} = \bigotimes_v\psi_v^{d_v}$, where $\psi_v$ is an unramified additive character of $F_v$ and the twisted character $\psi_v^{d_v}(x) \coloneqq \psi_v(d_v x)$ is of conductor $\mathfrak{d}_v^{-1}$.

Let $(\pi,V_{\pi})$ be a cuspidal automorphic representation of $\GL_n(\A_F)$ for $n \geq 2$, where $V_{\pi}$ is a space of automorphic forms on $\GL_n(\A_F)$. We define a global Whittaker function $W_{\varphi_{\pi}}$ associated to $\varphi_{\pi} \in V_{\pi}$ by
\[
W_{\varphi_{\pi}}(g)\coloneqq\int\limits_{\Ngp_n(F) \backslash  \Ngp_n(\A_F) } \varphi_{\pi}(ug) \overline{\psi_{\A_F}(u)} \,du.
\]
If $W_{\varphi_{\pi}}$ is a pure tensor, it can be decomposed as $W_{\varphi_{\pi}}=\prod_v W_{\varphi_{\pi},v}$  with $W_{\pi_v}\coloneqq W_{\varphi_{\pi},v} \in \WW(\pi_v,\psi_v^{d_v})$, where the generic irreducible admissible smooth representation $\pi_v$ is the local component of the automorphic representation $\pi = \bigotimes_v \pi_v$.

\begin{definition}
Let $(\pi,V_{\pi})$ be a cuspidal automorphic representation of $\GL_n(\A_F)$ with $\pi = \bigotimes_v \pi_v$. 
At each place $v$ of $F$, let $W^{\circ}_{\pi_v} \in \WW(\pi_v,\psi_v)$ denote the local Whittaker newform.
We define the \emph{global newform} $\varphi_{\pi}^{\circ} \in V_{\pi}$ to be the decomposable vector such that
$W_{\varphi^{\circ}_{\pi}} = \prod_v W_{\varphi^{\circ}_{\pi},v}$ with $W_{\varphi^{\circ}_{\pi},v}(g_v) \coloneqq W^{\circ}_{\pi_v}(\diag(d^{n-1}_v,\dotsm,d_v,1)g_v) \in \WW(\pi_v,\psi_v^{d_v}) $.
\end{definition}

Let $(\pi,V_{\pi})$ and $(\sigma,V_{\sigma})$ be cuspidal automorphic representations of $\GL_n(\A_F)$ and $\GL_m(\A_F)$
with $\pi = \bigotimes_v \pi_v$ and $\sigma = \bigotimes_v \sigma_v$. Throughout, we will take $S$ to be a finite set of places such that $\pi_v$, $\sigma_v$, and $\psi_v$ are all unramified whenever $v \notin S$. The finite set $S$ can vary depending on $\pi$, $\sigma$, and $F$, but always satisfies these properties. We define the partial Rankin--Selberg $L$-function by
\[
\Lambda^S(s,\pi \times \sigma)\coloneqq\prod_{v \notin S} L(s,\pi_v \times \sigma_v)
\]
for $\Re(s)$ sufficiently large. We similarly define the partial Asai $L$-function $\Lambda^S(s,\pi, \As)$, the partial exterior square $L$-function $\Lambda^S(s,\pi, \wedge^2)$; we also define the partial Bump--Friedberg $L$-function $\Lambda^S(s,\pi, \BF)$ to be $\Lambda^S(s,\pi) \Lambda^S(2s,\pi,\wedge^2)$.

\subsection{Global Rankin--Selberg Periods and the Petersson Inner Product}

Given a Schwartz--Bruhat function $\Phi \in \Scr(\A_F^n)$, we may form the $\Theta$-series
\[
\Theta_{\Phi}(a,g) \coloneqq \sum_{\xi \in F^n} \Phi(a \xi g) \quad \text{for $a \in \A_F^{\times}$ and $g \in \GL_n(\A_F)$}.
\]
Associated to this $\Theta$-series is an Eisenstein series, which is essentially the Mellin transform of $\Theta$. To be more precise, for a (unitary) Hecke character 
$\eta : F^{\times} \backslash \A_F^{\times} \rightarrow \C$, we set 
\begin{equation}
\label{DEFeisenstein}
E(g,s;\Phi,\eta) \coloneqq \left|\det g\right|^s_{\A_F} \int_{F^{\times} \backslash \A_F^{\times}} \Theta'_{\Phi}(a,g) \eta(a)|a|^{ns}_{\A_F} \, d^{\times}a,
\end{equation}
where $\Theta'_{\Phi}(a,g) \coloneqq \Theta_{\Phi}(a,g) - \Phi(0)$. This is absolutely convergent for $\Re(s) > 1$ and extends to a meromorphic function of $s \in \C$. The Eisenstein series $E(g,s;\Phi,\eta)$ is entire unless $\eta$ is the trivial unitary Hecke character, in which case it has a simple pole at $s=1$ with residue \cite[Lemma 4.2]{JS81a} 
\begin{equation}
\label{EisenResidue}
\frac{\vol(F^{\times}\backslash \A^1_F)}{n} \widehat{\Phi}(0),
\end{equation}
where the volume is taken with regards to the Tamagawa measure. Notably, the implicit constant denoted by $c$ in \cite[Lemma 4.2]{JS81a} is determined to be the volume of $F^{\times}\backslash \A^1_F$ in \cite[Proof of Proposition 3.1]{Zha14}.

\begin{theorem}[{Cf.\ \cite[Proposition 5.7]{Hum20a}}]
\label{thm:globalRSint}
Let $(\pi,V_{\pi})$ and $(\sigma,V_{\sigma})$ be cuspidal automorphic representations of $\GL_n(\A_F)$ with global newforms $\varphi^{\circ}_{\pi} \in V_{\pi}$ and $\varphi^{\circ}_{\sigma} \in V_{\sigma}$. Then there exists a right $K_n$-finite Schwartz--Bruhat function $\Phi \in \Scr(\A_F^n)$ such that for $\Re(s)$ sufficiently large, the global $\GL_n \times \GL_n$ Rankin--Selberg integral
\begin{equation}
\label{eqn:globalRSint}
I(s,\varphi_{\pi}^{\circ},\varphi_{\sigma}^{\circ},\Phi) \coloneqq \int_{[\GL_n]}
\varphi^{\circ}_{\pi}(g) \varphi^{\circ}_{\sigma}(g)E(g,s;\Phi,\omega_{\pi}\omega_{\sigma}) \, dg
\end{equation}
is equal, up to multiplication by a positive constant dependent only on the normalisation of the measure $dg$, to the product of $D_{F \slash \Q}^{\frac{n(n-1)s}{2}}$ and of the global completed na\"{i}ve Rankin--Selberg $L$-function
$\Lambda(s,\pi_{\ur} \times \sigma_{\ur})\coloneqq\prod_vL(s,\pi_{v,\ur} \times \sigma_{v,\ur})$.
\end{theorem}

\begin{proof}
This integral is Eulerian: by unfolding, $I(s,\varphi_{\pi}^{\circ},\varphi_{\widetilde{\pi}}^{\circ},\Phi)$ is equal, up to multiplication by a positive constant dependent only on the normalisation of the measure $dg$, to
\[
\prod_v \Psi(s,W_{\varphi^{\circ}_{\pi},v},W_{\varphi^{\circ}_{\widetilde{\pi}},v},\Phi_v)
\]
provided that $\Re(s)$ is sufficiently large. Upon making the change of variables 
\[
g_v \mapsto \diag(d_v^{1-n},\cdots,d^{-1}_v,1)g_v,
\]
we arrive at
\[
D^{\frac{n(n-1)s}{2}}_{F \slash \Q} \prod_v \Psi(s,W^{\circ}_{\pi_v},W^{\circ}_{\widetilde{\pi_v}},\Phi_v). 
\]
The result now follows from \hyperref[thm:GLnxGLn]{Theorem \ref*{thm:GLnxGLn}} for archimedean places and from \cite[Theorem 3.2]{Jo21} for nonarchimedean places.
\end{proof}

We turn our attention to Rankin--Selberg periods. We say that a pair of cuspidal automorphic representations $(\pi,V_{\pi})$ and $(\sigma,V_{\sigma})$ of $\GL_n(\A_F)$ admits a \emph{Rankin--Selberg period} (or $\GL_n(\A_F)$-period) if there exist cuspidal automorphic forms $\varphi_{\pi} \in V_{\pi}$ and $\varphi_{\sigma} \in V_{\sigma}$ such that
\[
\int_{[\GL_n]} \varphi_{\pi}(g) \varphi_{\sigma}(g) \, dg \neq 0. 
\] 

\begin{remark}
Jacquet and Shalika \citelist{\cite{JS81a}*{Lemma 4.4}; \cite{JS81b}*{Proposition 3.6}} have shown that the partial Rankin--Selberg $L$-function $\Lambda^S(s,\pi \times \sigma)$ has a pole at $s=1$ if and only if the pair of cuspidal automorphic representations $(\pi,V_{\pi})$ and $(\sigma,V_{\sigma})$ of $\GL_n(\A_F)$ admits a Rankin--Selberg period and $\omega_{\pi} \omega_{\sigma}$ is trivial. The partial Rankin--Selberg $L$-function in the assertion can freely be replaced by the completed Rankin--Selberg $L$-function $\Lambda(s,\pi \times \sigma)$ by virtue of \cite[Theorem 1.2]{CPS04}.
\end{remark}

Based on the work of Jacquet and Shalika \cite[Proposition 3.6]{JS81b}, W.\ Zhang \cite[Proposition 3.1]{Zha14} compares the Petersson inner product with local Rankin--Selberg periods $\beta_v(W_{\pi_v},W_{\sigma_v})$ given by \eqref{eqn:beta}, which in turn have a connection with $\GL_n(F_v)$-distinguished representations $\pi_v \otimes \sigma_v$ of $\GL_n(F_v) \times \GL_n(F_v)$ in the context of \hyperref[rmk:InvarinatRSperiod]{Remark \ref*{rmk:InvarinatRSperiod}}. We further refine this formula to relate it to the special value of the local Rankin--Selberg $L$-function at $s=1$.

\begin{theorem}
\label{thm:RSGlobalSV}
Let $(\pi,V_{\pi})$ be a unitary cuspidal automorphic representation of $\GL_n(\A_F)$ with global newform $\varphi^{\circ}_{\pi} \in V_{\pi}$; let $\varphi^{\circ}_{\widetilde{\pi}} \in V_{\widetilde{\pi}}$ be the corresponding global newform of $(\widetilde{\pi},V_{\widetilde{\pi}})$. Then there exists a right $K_n$-finite Schwartz--Bruhat function $\Phi \in \Scr(\A_F^n)$ such that the global $\GL_n \times \GL_n$ Rankin--Selberg period
\begin{equation}
\label{eqn:RSGlobalSV}
\int_{[\GL_n]} \varphi^{\circ}_{\pi}(g) \varphi^{\circ}_{\widetilde{\pi}}(g) \, dg
\end{equation}
is equal, up to multiplication by a positive constant dependent only on the normalisation of the measure $dg$, to
\[
\frac{nD^{\frac{n(n-1)}{2}}_{F \slash \Q} }{\widehat{\Phi}(0) \vol(F^{\times} \backslash \A^1_F) } \Res_{s=1}\Lambda^S(s,\pi \times \widetilde{\pi})\prod_{v \in S} L(1,\pi_{v,\ur} \times \widetilde{\pi}_{v,\ur}).
\]
\end{theorem}

\begin{proof}
From \eqref{EisenResidue}, the residue of $I(s,\varphi_{\pi}^{\circ},\varphi_{\widetilde{\pi}}^{\circ},\Phi)$ at $s=1$ is 
\[
\frac{\vol(F^{\times}\backslash \A^1_F)}{n} \widehat{\Phi}(0) \int_{[\GL_n]} \varphi^{\circ}_{\pi}(g) \varphi^{\circ}_{\widetilde{\pi}}(g) \, dg.
\]
The result then follows from \hyperref[thm:globalRSint]{Theorem \ref*{thm:globalRSint}}.
\end{proof}

\begin{remark}[Vanishing of periods]
The Rankin--Selberg period is known to vanish unless $\sigma \cong \widetilde{\pi}$ according to \cite[Lemma 4.4]{JS81a}.
\end{remark}

As before, let us now turn to the modified $\GL_n \times \GL_n$ Rankin--Selberg integral by Sakellaridis.

\begin{theorem}
\label{thm:Saknnglobal}
Let $(\pi,V_{\pi})$ and $(\sigma,V_{\sigma})$ be cuspidal automorphic representations of $\GL_n(\A_F)$ with global newforms $\varphi^{\circ}_{\pi} \in V_{\pi}$ and $\varphi^{\circ}_{\sigma} \in V_{\sigma}$. Then there exist a right $K_n$-finite Schwartz--Bruhat function $\Phi \in \Scr(\A_F^n)$ and a bi-$K_n$-finite Schwartz--Bruhat function $\Phi' \in \Scr(\Mat_{n \times n}(\A_F))$ such that for $\Re(s_1)$ and $\Re(s_2)$ sufficiently large, the global modified $\GL_n \times \GL_n$ Rankin--Selberg integral by Sakellaridis
\begin{equation}
\label{RS-Sakellaridis-equalrank}
\int\limits_{\Pgp^{\diag}_n(F) \backslash \GL_n(\A_F)\times \GL_n(\A_F)}  \varphi^{\circ}_{\pi}(g_1) \varphi^{\circ}_{\sigma}(g_2) \Phi'(g_1^{-1}g_2) \Phi(e_ng_1) \left| \frac{\det g_2}{\det g_1} \right|^{s_1}_{\A_F} \left|\det g_1\right|_{\A_F}^{s_2}  \, dg_2 \, dg_1
\end{equation}
is equal, up to multiplication by a positive constant dependent only on the normalisation of the measures $dg_1$ and $dg_2$, to the product of $D_{F \slash \Q}^{\frac{n(n-1)s_2}{2}}$ and of the global completed na\"{i}ve $L$-function
\[
\Lambda\left(s_2, \pi_{\ur} \times \sigma_{\ur} \right) \Lambda \left(s_1-\frac{n-1}{2}, \sigma \right).
\]
\end{theorem}

\begin{proof}
The result is a direct consequence of \cite[Remark 5.10]{Hum20a} coupled with unfolding the standard global $\GL_n \times \GL_n$ Rankin--Selberg integral (cf.\ \cite[Theorem 2.1]{CPS04}) upon making the change of variables $g_2 \mapsto g_1g_2$. Another way to prove this is to apply the unfolding technique as in the proof of \cite[Theorem 5.2.2]{Sak12}. In the domain of convergence, \eqref{RS-Sakellaridis-equalrank} is equal to
\[
\int\limits_{\Ngp^{\diag}_n(\A_F)  \backslash \GL_n(\A_F) \times \GL_n(\A_F)} W_{\varphi^{\circ}_{\pi}}(g_1) W_{\varphi^{\circ}_{\sigma}}(g_2) \Phi'(g^{-1}_1g_2) \Phi(e_ng_1)  \left| \frac{\det g_2}{\det g_1} \right|^{s_1}_{\A_F} \left|\det g_1\right|_{\A_F}^{s_2} \, dg_2 \, dg_1.
\]
To conclude, we appeal to \hyperref[thm:Saknn]{Theorem \ref*{thm:Saknn}} along with \hyperref[cor:Saknn-nonarch]{Corollary \ref*{cor:Saknn-nonarch}}. 
\end{proof}

A similar result holds for the modified $\GL_n \times \GL_{n - 1}$ Rankin--Selberg integral by Sakellaridis.

\begin{theorem}
\label{thm:Saknn-1global}
Let $(\pi,V_{\pi})$ and $(\sigma,V_{\sigma})$ be cuspidal automorphic representations of $\GL_n(\A_F)$ and $\GL_{n-1}(\A_F)$ respectively with global newforms $\varphi^{\circ}_{\pi} \in V_{\pi}$ and $\varphi^{\circ}_{\sigma} \in V_{\sigma}$, and suppose that $\sigma$ is everywhere unramified. Then there exists a bi-$K_n$-finite Schwartz--Bruhat function $\Phi' \in \Scr(\Mat_{n \times n}(\A_F))$ such that for $\Re(s_1)$ and $\Re(s_2)$ sufficiently large, the global modified $\GL_n \times \GL_{n-1}$ Rankin--Selberg integral by Sakellaridis
\begin{equation}
\label{eqn:GLnXGLmRS}
\int\limits_{\GL^{\diag}_{n-1}(F) \backslash \GL_{n-1}(\A_F)\times \GL_{n-1}(\A_F)} \varphi^{\circ}_{\pi} \begin{pmatrix} g_1 &0 \\ 0& 1 \end{pmatrix} \varphi^{\circ}_{\sigma}(g_2) \Phi'(g_1^{-1}g_2) \left| \frac{\det g_2}{\det g_1} \right|^{s_1}_{\A_F} \left|\det g_1\right|_{\A_F}^{s_2}  \, dg_2 \, dg_1
\end{equation}
is equal, up to multiplication by a positive constant dependent only on the normalisation of the measures $dg_1$ and $dg_2$, to the product of $\omega^{-1}_{\sigma}(d) D_{F \slash \Q}^{\frac{n(n-1)s_2}{2}}$ and of the global completed $L$-function
\[
\Lambda \left(s_2+\frac{1}{2}, \pi \times \sigma \right) \Lambda \left(s_1-\frac{n-2}{2}, \sigma \right).
\]
\end{theorem}

\begin{proof}
We make the change of variables $g_2 \mapsto g_2g_1$, then combine \cite[Remark 5.10]{Hum20a} with \cite[Proposition 5.5]{Hum20a}. Alternatively, as we have seen it in \cite[Theorem 5.2.5]{Sak12}, we factorise \eqref{eqn:GLnXGLmRS} from
\[
\int\limits_{\Ngp^{\diag}_{n-1}(\A_F)  \backslash \GL_{n-1}(\A_F) \times \GL_{n-1}(\A_F)} W_{\varphi^{\circ}_{\pi}}  \begin{pmatrix} g_1 &0 \\ 0& 1 \end{pmatrix}  W_{\varphi^{\circ}_{\sigma}}(g_2) \Phi'(g^{-1}_1g_2)   \left| \frac{\det g_2}{\det g_1} \right|^{s_1}_{\A_F} \left|\det g_1\right|_{\A_F}^{s_2} \, dg_2 \, dg_1.
\]
Upon making the change of variables $g_{1,v} \mapsto \diag(d_v^{1-n},\cdots,d^{-1}_v,1)g_{1,v}$, the desired identity follows from \hyperref[thm:Saknn-1]{Theorem \ref*{thm:Saknn-1}} aligned with \hyperref[cor:Saknn-1-nonarch]{Corollary \ref*{cor:Saknn-1-nonarch}}.
\end{proof}

\subsection{Global Flicker--Rallis Periods}

We start with the existence of a weak test vector for the global Flicker integral.

\begin{theorem}
\label{thm:globalFlickerint}
Let $E$ be a quadratic extension of $F$, and let $(\pi,V_{\pi})$ be a cuspidal automorphic representation of $\GL_n(\A_E)$ with global newform $\varphi^{\circ}_{\pi} \in V_{\pi}$. Then there exists a right $K_n$-finite Schwartz--Bruhat function $\Phi \in \Scr(\A_F^n)$ such that for $\Re(s)$ sufficiently large, the global $\GL_n$ Flicker integral
\begin{equation}
\label{eqn:Flickerintglobal}
I(s,\varphi^{\circ}_{\pi},\Phi)\coloneqq \int_{[\GL_n]} \varphi^{\circ}_{\pi}(g) E\left(g,s;\Phi,\omega_{\pi}|_{\A_F^{\times}}\right) \, dg
\end{equation}
is equal, up to multiplication by a positive constant dependent only on the normalisation of the measures $dg$, to the product of $D_{F \slash \Q}^{\frac{n(n-1)s}{2}}$ and the global completed na\"{i}ve Asai $L$-function
$\Lambda(s,\pi_{\ur}, \As)\coloneqq\prod_vL(s,\pi_{v,\ur}, \As)$.
\end{theorem}

\begin{proof}
This integral is Eulerian: by inserting the definition of Eisenstein series (cf.\ \eqref{DEFeisenstein}) and the Fourier--Whittaker expansion of $\varphi_{\pi}$, and then unfolding \cite[p.\ 303]{Fli88}, we see that $I(s,\varphi^{\circ}_{\pi},\Phi)$ is equal, up to multiplication by a positive constant dependent only on the normalisation of the measure $dg$, to
\[
\prod_v \Psi(s,W_{\varphi^{\circ}_{\pi},v},\Phi_v)
\]
provided that $\Re(s)$ is sufficiently large. Upon making the change of variables 
\[
g_v \mapsto \diag(d_v^{1-n},\cdots,d^{-1}_v,1)g_v,
\]
we arrive at
\[
D^{\frac{n(n-1)s}{2}}_{F \slash \Q} \prod_v \Psi(s,W^{\circ}_{\pi_v},\Phi_v). 
\]
The result now follows from \hyperref[thm:Flicker]{Theorem \ref*{thm:Flicker}} for archimedean places and from \cite[Theorem 4.2]{Jo21} for nonarchimedean places.
\end{proof}

We turn our attention toward global Flicker--Rallis periods. We say that a cuspidal automorphic representation $(\pi,V_{\pi})$ of $\GL_n(\A_E)$ admits a \emph{Flicker--Rallis period} (or  $\GL_n(\A_F)$-period) if there exists a cuspidal automorphic form $\varphi_{\pi}$ in the space $V_{\pi}$ such that
\[
\int_{[\GL_n]} \varphi_{\pi}(g) \, dg \neq 0.
\]
Such a representation is referred to be \emph{$\GL_n(\A_F)$-distinguished}.

\begin{remark}
It is a result of \cite{Fli88} and \cite{FZ95} that the global partial Asai $L$-function $\Lambda^S(s,\pi,\As)$ has a pole at $s=1$ if and only if $\omega_{\pi}|_{\A^{\times}_F} = 1$ and $\pi$ is 
$\GL_n(\A_F)$-distinguished.
\end{remark}

W. Zhang \cite[Proposition 3.2]{Zha14} expresses the global Flicker--Rallis period on the space $V_{\pi}$ as a product of local $\GL_n(F_v)$-distinguished linear functionals $\vartheta_v^{\flat}$ given by \eqref{eqn:varthetaflat}; the underlying idea originates from the work of Gelbart, Jacquet, and Rogawski \cite[p.\ 184--186]{GJR01}. In particular, the global Flicker--Rallis period attached to a global newform $\varphi^{\circ}_{\pi}$ is related to the product of the special value of the local Asai $L$-function at $s=1$.

\begin{theorem}
\label{thm:globalFRperiod}
Let $E$ be a quadratic extension of $F$, and let $(\pi,V_{\pi})$ be a unitary cuspidal automorphic representation of $\GL_n(\A_E)$ with global newform $\varphi^{\circ}_{\pi} \in V_{\pi}$ for which the central character $\omega_{\pi}$ satisfies $\omega_{\pi}|_{\A^{\times}_F} = 1$. Then there exists a right $K_n$-finite Schwartz--Bruhat function $\Phi \in \Scr(\A_F^n)$ such that the global Flicker--Rallis period
\begin{equation}
\label{eqn:globalFRperiod}
\int_{[\GL_n]} \varphi^{\circ}_{\pi}(g) \, dg
\end{equation}
is equal, up to multiplication by a positive constant dependent only on the normalisation of the measure $dg$, to
\[
\frac{nD^{\frac{n(n-1)}{2}}_{F \slash \Q}  }{\widehat{\Phi}(0)\vol(F^{\times} \backslash \A^1_F)} \Res_{s=1}\Lambda^S(s,\pi,\As) \prod_{v \in S}L(1,\pi_{v,\ur},\As).
\]
\end{theorem}

Here we have that $L(1,\pi_{v,\ur},\As)=L(1,\pi_{v,\ur} \times \pi_{v,\ur})$ and $L(1,\pi_{v},\As) = L(1,\pi_{v} \times \pi_{v})$ if $v$ splits in $E$, so that $E_v = F_{v} \oplus F_{v}$.

\begin{proof}
From \eqref{EisenResidue}, the residue of $I(s,\varphi_{\pi}^{\circ},\varphi_{\widetilde{\pi}}^{\circ},\Phi)$ at $s=1$ is 
\[
\frac{\vol(F^{\times}\backslash \A^1_F)}{n} \widehat{\Phi}(0) \int_{[\GL_n]} \varphi^{\circ}_{\pi}(g) \varphi^{\circ}_{\widetilde{\pi}}(g) \, dg.
\]
The result then follows from \hyperref[thm:globalFlickerint]{Theorem \ref*{thm:globalFlickerint}}.
\end{proof}

\subsection{Global Friedberg--Jacquet Periods}
\label{sec:GlobalFJ}

Our immediate goal is to show the existence of a weak test vector for the global Bump--Friedberg integral.

\begin{theorem}
\label{thm:globalBF}
Let $(\pi,V_{\pi})$ be a cuspidal automorphic representation of $\GL_n(\A_F)$ with global newform $\varphi^{\circ}_{\pi} \in V_{\pi}$. For $m = \lfloor \frac{n}{2}\rfloor$, there exists a right $K_n$-finite Schwartz--Bruhat function $\Phi \in \Scr(\A_F^m)$ such that for $\Re(s_1)$ and $\Re(s_2)$ sufficiently large, 
\begin{enumerate}[leftmargin=*,label=\emph{(\roman*)}]
\item\label{globalBF1} for $n=2m$, the global Bump--Friedberg integral
\begin{equation}
\label{eqn:globalBFinteven}
Z(s_1,s_2,\varphi^{\circ}_{\pi},\Phi) \coloneqq \int\limits_{[\GL_m \times \GL_m]} \varphi^{\circ}_{\pi}(J(g,g')) E(g',s_2;\Phi,\omega_{\pi}) \left| \frac{\det g}{\det g'} \right|^{s_1-\frac{1}{2}}_{\A_F}  \, dg \, dg'
\end{equation}
is equal, up to multiplication by a positive constant dependent only on the normalisation of the measures $dg$ and $dg'$, to the product of $D^{m(s_1-1/2)+m(m-1)s_2}_{F \slash \Q}$ and of the global completed na\"{i}ve Bump--Friedberg $L$-function $\Lambda(s_1,s_2,\pi_{\ur}, \BF) \coloneqq \prod_v L(s_1,\pi_{v,\ur}) L(s_2,\pi_{v,\ur},\wedge^2)$;
\item for $n = 2m + 1$, the global Bump--Friedberg integral
\begin{multline}
\label{eqn:globalBFintodd}
Z(s_1,s_2,\varphi^{\circ}_{\pi},\Phi)\\
\coloneqq \int\limits_{[\GL_{m + 1} \times \GL_m]} \varphi^{\circ}_{\pi}(J(g,g')) \, E\left(g,\frac{s_1 + ms_2}{m + 1};\Phi,\omega_{\pi}\right) \left( \frac{|\det g'|_{\A_F}}{|\det g|^{m/(m+1)}_{\A_F}} \right)^{-s_1+s_2} \, dg \, dg'
\end{multline}
is equal, up to multiplication by a positive constant dependent only on the normalisation of the measures $dg$ and $dg'$, to the product of $D^{ms_1+m^2s_2}_{F \slash \Q}$ and of the global completed na\"{i}ve Bump--Friedberg $L$-function $\Lambda(s_1,s_2,\pi_{\ur}, \BF) \coloneqq \prod_v L(s_1,\pi_{v,\ur})L(s_2,\pi_{v,\ur},\wedge^2)$.
\end{enumerate}
\end{theorem}

\begin{proof}
This integral is Eulerian: appealing to the standard unfolding technique due to Matringe \cite[Theorem 4.4]{Mat15} (cf.\ \cite{BF90}), we see that $Z(s_1,s_2,\varphi_{\pi}^{\circ},\Phi)$ is equal, up to multiplication by a positive constant dependent only on the normalisation of the measures $dg$ and $dg'$, to
\[
\prod_v B(s_1,s_2,W_{\varphi_{\pi}^{\circ},v},\Phi_v)
\]
provided that $\Re(s_1)$ and $\Re(s_2)$ are sufficiently large. Upon making the change of variables $g_v \mapsto \diag(d_v^{1-n},\cdots,d^{-3}_v,d_v^{-1})g_v$ and $g_v' \mapsto \diag(d_v^{2-n},\cdots,d^{-2}_v,1)g_v'$ for $n = 2m$ and $g_v \mapsto \diag(d_v^{1-n},\cdots,d^{-2}_v,1)g_v$ and $g_v' \mapsto \diag(d_v^{2-n},\cdots,d^{-3}_v,d_v^{-1})g_v'$ for $n = 2m + 1$, we arrive at
\[\begin{dcases*}
D^{m(s_1-1/2)+m(m-1)s_2}_{F \slash \Q} \prod_v B(s_1,s_2,W_{\pi_v}^{\circ},\Phi_v) & for $n = 2m$,	\\
D^{ms_1+m^2s_2}_{F \slash \Q} \prod_v B(s_1,s_2,W_{\pi_v}^{\circ},\Phi_v) & for $n = 2m + 1$,
\end{dcases*}\]
The result now follows from \hyperref[thm:BF-reduction]{Theorem \ref*{thm:BF-reduction}} for archimedean places and from \cite[Theorem 5.1]{MY13} for nonarchimedean places.
\end{proof}

We switch our attention to global Friedberg--Jacquet periods. We say that a cuspidal automorphic representation $(\pi,V_{\pi})$ of $\GL_{2m}(\A_F)$ admits a \emph{Friedberg--Jacquet period} (or $\Hgp_{m,m}(\A_F)$-period) if there exists a cuspidal automorphic form $\varphi_{\pi} \in V_{\pi}$ such that
\[
\int\limits_{[\GL_m \times \GL_m]} \varphi_{\pi}(J(g,g')) \, dg \, dg' \neq 0.
\]
Such a representation is referred to be \emph{$\Hgp_{m,m}(\A_F)$-distinguished}.

\begin{remark}
The main result of Matringe \cite[Theorem 4.7]{Mat15} tells us that the partial global Bump--Friedberg $L$-function $\Lambda^S(s,\pi,\BF)$ has a pole at $s=1/2$ if and only if $\omega_{\pi}$ is trivial and $\pi$ is $\Hgp_{m,m}(\A_F)$-distinguished.
\end{remark} 

We establish results analogous to \cite[Propositions 3.1 and 3.2]{Zha14}, which describes the explicit decomposition of the global Friedberg--Jacquet period in terms of local $\Hgp_{m,m}(F_v)$-distinguished linear functionals $\vartheta_v^{\sharp}$ given by \eqref{eqn:varthetasharp}.

\begin{proposition}
\label{prop:GlobalFJdistinction}
Let $(\pi,V_{\pi})$ be a unitary cuspidal automorphic representation of $\GL_{2m}(\A_F)$ with trivial central character. Then for every pure tensor $\varphi_{\pi} \in V_{\pi}$, the global Friedberg--Jacquet period
\[
\int\limits_{[\GL_m \times \GL_m]} \varphi_{\pi}(J(g,g')) \, dg \, dg'
\]
is equal, up to multiplication by a positive constant dependent only on the normalisation of the measures $dg$ and $dg'$, to
\[
\frac{m}{\vol(F^{\times} \backslash \A^1_F)} \Res_{s=1/2}\Lambda^S(s,\pi,\BF) \prod_{v \in S} \vartheta_v^{\sharp}(W_{\varphi_{\pi},v}).
\]
\end{proposition}

The proof requires \hyperref[prop:FJdistinction]{Proposition \ref*{prop:FJdistinction}}, which we first prove.

\begin{proof}[Proof of \text{\hyperref[prop:FJdistinction]{Proposition \ref*{prop:FJdistinction}}}]
The proof is inspired by work of Gelbart, Jacquet, and Rogawski \cite[p.\ 185--186]{GJR01}. The result follows from the global theory. To be more precise, we first observe that $\omega_{\pi_v}$ must be trivial, as $(\pi_v,V_{\pi_v})$ is $\Hgp_{m,m}(F_v)$-distinguished. The global Bump--Friedberg integral $Z(s,2s,\varphi_{\pi},\Phi)$ is Eulerian (cf. \cite[Theorem 4.4]{Mat15}), so that it is equal, up to multiplication by a positive constant dependent only on the normalisation of the measures $dg$ and $dg'$, to
\[
\prod_v B(s,W_{\varphi_{\pi},v},\Phi_v).
\]
We may choose a factorisable Schwartz--Bruhat function $\Phi \in \Scr(\A_F^n)$ such that $\Phi_v$ is the characteristic function of $\OO_v^m$ for all nonarchimedean places $v \notin S$, and a factorisable cuspidal automorphic form $\varphi_{\pi} \in V_{\pi}$ such that $W_{\varphi_{\pi},v}=W^{\circ}_{\pi_v}$ with respect to the unramified character $\psi_v$ for all nonarchimedean places $v \notin S$. Taking \eqref{EisenResidue} into account, we see that by taking the residue of $Z(s,2s,\varphi_{\pi},\Phi)$ at $s=1/2$,
\begin{multline}
\label{BFResidue}
\frac{\vol(F^{\times}\backslash \A^1_F)}{m} \widehat{\Phi}(0)\int\limits_{[\GL_m \times \GL_m]} \varphi_{\pi}(J(g,g')) \, dg \, dg' \\
=\Res_{s=1/2}\Lambda^S(s,\pi,\BF) \prod_{v \in S} B\left(\frac{1}{2},W_{\varphi_{\pi},v},\Phi_v\right).  
\end{multline}
We now consider $B(1/2,W_{\varphi_{\pi},v},\Phi_v)$ for each nonarchimedean place $v \in S$ and each archimedean place $v$. On the one hand, we use the Iwasawa decomposition, and then evaluate the resulting integral over $\Ngp_m(F_v) \backslash \GL_m(F_v) \ni g_v$ and $\Ngp_m(F_v) \backslash \Pgp_m(F_v) \ni p'_v$, leading to the identity
\[
B\left(\frac{1}{2},W_{\varphi_{\pi},v},\Phi_v\right) = \int_{K_{m,v}}\int_{F_v^{\times}} \vartheta_v^{\sharp}(\pi_v(J(1,k'_v))W_{\varphi_{\pi},v}) \Phi_v(z'_v e_m k'_v)  \left| z'_v \right|^m_v \, d^{\times}z'_v \, dk'_v.
\]
On the other had, via the Iwasawa decomposition once more,
\[\widehat{\Phi_v}(0) = \int_{F_v^n} \Phi_v(x_v) \, dx_v =\int\limits_{\Pgp_m(F_v)  \backslash \GL_m(F_v)}  \Phi_v(e_m g'_v) \, dg'_v = \int_{K_{m,v}} \int_{F_v^{\times}} \Phi_v(z'_v e_m k'_v) \left| z'_v \right|^m_v \, d^{\times}z'_v \, dk'_v.
\]
Comparing the integrals over $K_{m,v} \ni k'_v$ and $F_v^{\times} \ni z'_v$ of both sides of \eqref{BFResidue}, there should be a linear functional $\gamma^{\sharp}$ on the Whittaker model $\WW(\pi_v,\psi_v)$ such that
\begin{multline*}
\int_{K_{m,v}}\int_{F_v^{\times}} \vartheta_v^{\sharp}(\pi_v(J(1,k'_v))W_{\varphi_{\pi},v}) \Phi_v(z'_v e_m k'_v)  \left| z'_v \right|^m_v \, d^{\times}z'_v \, dk'_v \\
= \gamma_v^{\sharp}(W_{\varphi_{\pi},v})\int_{K_{v}}\int_{F_v^{\times}} \Phi_v(z'_v e_m k'_v) \left| z'_v \right|^m_v \, d^{\times}z'_v \, dk'_v
\end{multline*}
for any Whittaker function $W_{\varphi_{\pi},v} \in \WW(\pi_v,\psi_v)$ and any Schwartz--Bruhat function $\Phi_v \in \Scr(F_v^m)$. Each $\Phi_v \in \Scr(F^m_v)$ defines a smooth function on $\GL_n(F_v)$, left invariant by $\Pgp_m(F_v)$, via $g \mapsto \Phi(e_mg)$. This implies that any smooth function $f$ on $K_{m,v}$ that is invariant under $\Pgp_{(m-1,1)}(F_v) \cap K_{m,v}$ satisfies
\[
\int_{K_{m,v}} \int_{F_v^{\times}} \vartheta_v^{\sharp}(\pi_v(J(1,k'_v))W_{\varphi_{\pi},v}) f(k'_v)  \left| z'_v \right|^m_v \, d^{\times}z'_v \, dk'_v = \gamma_v^{\sharp}(W_{\varphi_{\pi},v})\int_{K_{m,v}} f(k'_v) \, dk'_v.
\]
We can move one step further to claim that the same relation holds for all smooth functions $f$ on $K_{m,v}$, as $\vartheta_v^{\sharp}$ is a $\Pgp_{2m}(F_v) \cap \Hgp_{m,m}(F_v)$-invariant form. It follows that $\gamma^{\sharp}_v = \theta^{\sharp}_v$ and that $\vartheta_v^{\sharp}$ is invariant under $\{ J(1,k'_v) : k'_v \in K_{m,v} \}$. In summary,  $\vartheta_v^{\sharp}$ is invariant under $\Hgp_{m,m}(F_v)$.
\end{proof}

\begin{proof}[Proof of  \text{\hyperref[prop:GlobalFJdistinction]{Proposition \ref*{prop:GlobalFJdistinction}}}]
We only deal with the case when $(\pi,V_{\pi})$ affords a $\Hgp_{m,m}(\A_F)$-period, for otherwise the desired identity is trivially true as both sides are equal to zero. Evidently, $(\pi_v,V_{\pi_v})$ is $\Hgp_{m,m}(F_v)$-distinguished, which can only occur provided that $\omega_{\pi_v}$ is trivial. We may choose a factorisable Schwartz--Bruhat function $\Phi \in \Scr(\A_F^n)$ such that $\Phi_v$ is the characteristic function of $\OO_v^m$ for all nonarchimedean places $v \notin S$, so that $\widehat{\Phi_v}(0)=1$, and a factorisable cuspidal automorphic form $\varphi_{\pi} \in V_{\pi}$ such that $W_{\varphi_{\pi},v}=W^{\circ}_{\pi_v}$ with respect to the unramified character $\psi_v$ for all $v \notin S$. Recalling \eqref{BFResidue}, we deduce that it is sufficient to show that $B(1/2,W_{\varphi_{\pi},v},\Phi_v)=\vartheta_v^{\sharp}(W_{\varphi_{\pi},v}) \widehat{\Phi_v}(0)$ for each nonarchimedean place $v \in S$ and each archimedean place $v$. To that end, we expand
\[
B\left(\frac{1}{2},W_{\varphi_{\pi},v},\Phi_v\right) = \int\limits_{\Ngp_m(F_v) \backslash \GL_m(F_v)}  \int\limits_{\Ngp_m(F_v) \backslash \GL_m(F_v)} W_{\varphi_{\pi},v}(J(g_v,g'_v)) \Phi_v(e_m g'_v)  \left|\det g'_v\right|_v \, dg_v \, dg'_v 
\]
as a triple integral, and then rewrite it in terms of $\vartheta_v^{\sharp}$ (cf.\ \cite[\S 4]{AKT04}), yielding
\begin{multline*}
\int\limits_{\Pgp_m(F_v) \backslash \GL_m(F_v)} \int\limits_{\Ngp_m(F_v) \backslash \Pgp_m(F_v)}  \int\limits_{\Ngp_m(F_v) \backslash \GL_m(F_v)} W_{\varphi_{\pi},v}(J(g_v,p'_vg'_v)) \Phi_v(e_m g'_v)  \, dg_v \, dp'_v \, dg'_v \\
= \int\limits_{\Pgp_m(F_v)  \backslash \GL_m(F_v)}   \vartheta_v^{\sharp}(\pi_v(J(1,g'_v))W_{\varphi_{\pi},v})  \Phi_v(e_m g'_v) \, dg'_v.
\end{multline*}
But we know from \hyperref[prop:FJdistinction]{Proposition \ref*{prop:FJdistinction}} that $ \vartheta_v^{\sharp}$ is $\Hgp_{m,m}(F_v)$-invariant. Then
\[B\left(\frac{1}{2},W_{\varphi_{\pi},v},\Phi_v\right) = \vartheta_v^{\sharp}(W_{\varphi_{\pi},v}) \int\limits_{\Pgp_m(F_v)  \backslash \GL_m(F_v)}  \Phi_v(e_m g'_v) \, dg'_v = \vartheta_v^{\sharp}(W_{\varphi_{\pi},v}) \widehat{\Phi_v}(0).\qedhere\]
\end{proof}

\begin{remark}[Vanishing of periods]
For $n=2m+1$ odd, it is straightforward that the residue of $Z(s,2s,\varphi^{\circ}_{\pi},\Phi)$ at $s=1$ is
\[
\frac{\vol(F^{\times}\backslash \A^1_F)}{m+1} \widehat{\Phi}(0)\int\limits_{[\GL_{m + 1} \times \GL_m]} \varphi^{\circ}_{\pi}(J(g,g')) \, dg \, dg'.
\]
Unfortunately, this period integral is known to vanish as a consequence of \cite[\S 2.1 and Proposition 2.1]{FJ93} (cf.\ \cite[p.\ 594]{Mat15}).
\end{remark}

For $\varphi^{\circ}_{\pi}$ the global newform, the global Friedberg--Jacquet period can be explicitly evaluated so that it is \emph{Eulerian}.

\begin{theorem}
\label{thm:globalFJperiod}
Let $(\pi,V_{\pi})$ be a unitary cuspidal automorphic representation of $\GL_{2m}(\A_F)$ with global newform $\varphi^{\circ}_{\pi} \in V_{\pi}$ for which the central character is trivial. Then the global Friedberg--Jacquet period
\begin{equation}
\label{eqn:globalFJperiod}
\int\limits_{[\GL_m \times \GL_m]} \varphi^{\circ}_{\pi}(J(g,g')) \, dg \, dg'
\end{equation}
is equal, up to multiplication by a positive constant dependent only on the normalisation of the measures $dg$ and $dg'$, to
\[
\frac{mD^{m(m-1)}_{F \slash \Q}}{\widehat{\Phi}(0) \vol(F^{\times} \backslash \A^1_F)}  \Res_{s=1/2}\Lambda^S(s,\pi,\BF)  \prod_{v\in S} L\left(\frac{1}{2},\pi_{v,\ur}, \BF\right).
\]
\end{theorem}

\begin{proof}
From \eqref{EisenResidue}, the residue of $Z(s,2s,\varphi_{\pi}^{\circ},\Phi)$ at $s=1/2$ is 
\[
\frac{\vol(F^{\times}\backslash \A^1_F)}{m} \widehat{\Phi}(0) \int\limits_{[\GL_m \times \GL_m]} \varphi^{\circ}_{\pi}(J(g,g')) \, dg \, dg'.
\]
The result then follows from \hyperref[thm:globalBF]{Proposition \ref*{thm:globalBF}} \emph{\ref{globalBF1}}.
\end{proof}

\begin{remark}[Jacquet--Shalika Periods]
We let $\MM_m(F)$ denote the set of $m \times m$ matrices and $\NN_m(F)$ denote the subset of upper triangular matrices. The Shalika subgroup $\Sgp_{2m}(F)$ is defined to be
 \[
\Sgp_{2m}(F)\coloneqq   \left\{ \begin{pmatrix} 1_m & X \\0 & 1_m \end{pmatrix} \begin{pmatrix} g &0 \\ 0& g \end{pmatrix} :  X \in \MM_m(F), \ g \in \GL_m(F) \right\}.
 \]
We fix a nontrivial Shalika character $\Theta$ on $\Sgp_{2m}(F)$ such that
\[
\Theta \left( \begin{pmatrix} 1_m & X \\0 & 1_m \end{pmatrix} \begin{pmatrix} g &0 \\ 0& g \end{pmatrix}  \right)\coloneqq\psi(\Tr X).
\]
We say that a cuspidal automorphic representation $(\pi,V_{\pi})$ of $\GL_{2m}(\A_F)$ admits a \emph{Jacquet--Shalika period} (or $(\Sgp_{2m}(\A_F),\Theta)$-period) \cite[(11)]{Duh19} if there exists a cuspidal automorphic form $\varphi_{\pi} \in V_{\pi}$ for which
\[
 \int_{[\GL_m]} \int\limits_{\MM_m(F) \backslash \MM_m(\A_F)}   \varphi_{\pi}  \left( \begin{pmatrix} 1_m & X \\0 & 1_m \end{pmatrix} \begin{pmatrix} g &0 \\ 0& g \end{pmatrix} \right) \overline{\psi}(\Tr X)  \, dX \, dg \neq 0.
\] 
Such a representation is referred to be \emph{$(\Sgp_{2m}(\A_F),\Theta)$-distinguished}. The celebrated result of Jacquet and Shalika \cite{JS90} shows that the partial global exterior square $L$-function $\Lambda^S(s,\pi,\wedge^2)$
has a pole at $s=1$ if and only if $\omega_{\pi}$ is trivial and $(\pi,V_{\pi})$ admits a $(\Sgp_{2m}(\A_F),\Theta)$-period. For the rest of the discussion, we assume that $(\pi,V_{\pi})$ is unitary such that $(\pi,V_{\pi})$ affords the $(\Sgp_{2m}(\A_F),\Theta)$-period. We analogously define a local Shalika functional $ \vartheta_v^{\dagger} :\WW(\pi_v,\psi_v) \rightarrow \C$ by
\[
\vartheta_v^{\dagger}(W_{\pi_v}) \coloneqq \int\limits_{\Ngp_m(F_v) \backslash \Pgp_m(F_v)} \int\limits_{\NN_m(F_v) \backslash \MM_m(F_v)} W_{\pi_v}  \left( w_{m,m} \begin{pmatrix} 1_m & X_v \\ 0 & 1_m  \end{pmatrix} \begin{pmatrix} p_v & 0 \\ 0 & p_v \end{pmatrix} \right) \overline{\psi}(\Tr X_v) \, dX_v \, dp_v.
\]
The linear functional $\vartheta_v^{\dagger}$ is \emph{a priori} ($\Pgp_{2m}(F_v) \cap  \Sgp_{2m}(F_v),\Theta_v)$-quasi-invariant. Thankfully, it can be extended to a $(\Sgp_{2m}(F_v),\Theta_v)$-quasi-invariant functional (cf.\ the proof of \hyperref[prop:FJdistinction]{Proposition \ref*{prop:FJdistinction}} and \cite[Lemma 5.4]{Duh19}). Via unfolding \cite[Proposition 6.5]{JS90}, the Eulerian factorisation of the Jacquet--Shalika period \cite[(18)]{Duh19} can be deduced, with some changes of notations, just as in the proof of
\hyperref[prop:GlobalFJdistinction]{Proposition \ref*{prop:GlobalFJdistinction}}:
\[\int_{[\GL_m]} \int\limits_{\MM_m(F) \backslash \MM_m(\A_F)}  \varphi_{\pi}  \left( \begin{pmatrix} 1_m & X \\ 0& 1_m \end{pmatrix} \begin{pmatrix} g &0 \\0 & g \end{pmatrix} \right) \overline{\psi}(\Tr X) \, dX \, dg\]
is equal, up to multiplication by a positive constant dependent only on the normalisation of the measures $dg$ and $dX$, to
\[
\frac{m}{ \vol(F^{\times} \backslash \A^1_F)} \Res_{s = 1} \Lambda^S(s,\pi,\wedge^2) \prod_{v\in S} \vartheta_v^{\dagger}(W_{\pi_v}),
\]
where $S$ contains the archimedean places. When $v$ is an archimedean place, so that $F_v \in \{\R,\C\}$ is an archimedean local field, it is our belief that the Jacquet--Shalika period integral involving unramified data is \emph{no longer} a nonzero polynomial multiple of the exterior square $L$-function $L(s,\pi_v,\wedge^2)$, as mentioned in \hyperref[sec:Intro-Comparision]{Section \ref*{sec:Intro-Comparision}}. This phenomenon can even be observed for a spherical induced representation of Langlands type $\pi_v$ of $\GL_2(F_v)$ from the identity
\[
\int_{F_v^{\times}}  W^{\circ}_{\pi_v} \begin{pmatrix} z_v &0 \\ 0& z_v \end{pmatrix} \Phi_{\ur}(z_v) |z_v|_v^s \,  d^{\times} z_v = L(s,\pi_v,\wedge^2 ) \frac{1}{4\pi i} \int_{\sigma - i\infty}^{\sigma + i\infty} L(w,\pi_v) \, dw.
\]
\end{remark}

\end{document}